\documentclass{siamart171218}

\usepackage{lipsum}
\usepackage{amsfonts}
\usepackage{graphicx}
\usepackage{epstopdf}
\usepackage{algorithmic}
\usepackage{tikz}
\ifpdf
  \DeclareGraphicsExtensions{.eps,.pdf,.png,.jpg}
\else
  \DeclareGraphicsExtensions{.eps}
\fi

\newcommand{\nmbr}[1]{\mathbb{#1}}
\newcommand{\bc}[1]{{\mathcal #1}}
\newcommand{\alg}[1]{{\mathbf #1}}
\newcommand{\dual}[2]{\left\langle #1,#2 \right\rangle}

\graphicspath{{./figs/}}

\headers{Spectrum Of Preconditioned Discretized Operators}
{T. Gergelits, K.-A. Mardal, B. F. Nielsen, Z. Strako{\v s}}

\title{Laplacian preconditioning of elliptic PDEs: Localization of the eigenvalues of the discretized operator\thanks{Submitted to the editors September 10, 2018.
\funding{The work of Tom{\'a}{\v s} Gergelits and Zden{\v e}k Strako{\v s} has been supported by the Grant Agency of the Czech Republic under the contract No. 17-04150J. The work of Bj{\o}rn F. Nielsen has been supported by The Research Council of Norway, project number 239070. The work of Tom{\'a}{\v s} Gergelits has been also supported by the Charles University, project GA UK No. 172915.}}}

\author{
Tom{\'a}{\v s} Gergelits\footnotemark[2]\ \footnotemark[3]
\and Kent-Andr\'{e} Mardal\footnotemark[4]\ \footnotemark[5]
\and Bj{\O}rn Fredrik Nielsen\footnotemark[6]
\and Zden{\v e}k Strako{\v s}\footnotemark[3]
}

\usepackage{amsopn}

\ifpdf
\hypersetup{
  pdftitle={Laplacian preconditioning of elliptic PDEs: Localization of the eigenvalues of the discretized operator},
  pdfauthor={T. Gergelits, K.-A. Mardal, B. F. Nielsen, Z. Strako{\v s}}
}
\fi

\begin{document}

\maketitle

\renewcommand{\thefootnote}{\fnsymbol{footnote}}
\footnotetext[2]{Institute of Computer Sciences, Czech Academy of Science, Prague, Czech Republic (\email{gergelits@cs.cas.cz}).}
\footnotetext[3]{Faculty of Mathematics and Physics, Charles University, Prague, Czech Republic (\email{gergelits@karlin.mff.cuni.cz}, \email{strakos@karlin.mff.cuni.cz}).}
\footnotetext[4]{Department of Mathematics, University of Oslo, 0316 Oslo, Norway 
  (\email{kent-and@math.uio.no}).}
\footnotetext[5]{Department of Numerical Analysis and Scientific Computing, Simula Research Laboratory, 1325 Lysaker, Norway  (\email{kent-and@simula.uio.no}).}
\footnotetext[6]{Faculty of Science and Technology, Norwegian University of Life Sciences, NO-1432 {\AA}s, Norway
  (\email{bjorn.f.nielsen@nmbu.no}).}

\renewcommand{\thefootnote}{\arabic{footnote}}

\begin{abstract}
  In the paper \textit{Preconditioning by inverting the {L}aplacian; an analysis of the eigenvalues. IMA Journal of Numerical Analysis 29, 1 (2009), 24--42}, Nielsen, Hackbusch and Tveito study the operator generated by using the inverse of the Laplacian as preconditioner for second order elliptic PDEs $\nabla \cdot (k(x) \nabla u)  = f$. They prove that the range of $k(x)$ is contained in the spectrum of the preconditioned operator, provided that $k$ is continuous. Their rigorous analysis only addresses mappings defined on infinite dimensional spaces, but the numerical experiments in the paper suggest that a similar property holds in the discrete case.
Motivated by this investigation, we analyze the eigenvalues of the matrix $\bf{L}^{-1}\bf{A}$, where $\bf{L}$ and ${\bf{A}}$ are the stiffness matrices associated with the Laplace operator and general second order elliptic operators, respectively. Without any assumption about the continuity of $k(x)$, we prove the existence of a one-to-one pairing between the eigenvalues of  $\bf{L}^{-1}\bf{A}$ and the intervals determined by the images under $k(x)$ of the supports of the FE nodal basis functions. As a consequence, we can show that the nodal values of $k(x)$ yield accurate approximations of the eigenvalues of $\bf{L}^{-1}\bf{A}$. Our theoretical results are illuminated by several numerical experiments.
\end{abstract}



\section{Introduction}
\label{introduction}
The classical analysis of Krylov subspace solvers for matrix problems with {Hermitian} matrices relies on their spectral properties; see, e.g.,  \cite{BAxe94,BHac94}. Typically one seeks a preconditioner which {yields parameter}  independent bounds for the extreme eigenvalues; see, e.g., \cite{FMP90,Hip06,Mar11,Gun14,MSB15} for a discussion of this issue in terms of {\em operator preconditioning}.
This approach has the advantage that only the largest and smallest eigenvalues {(in the }absolute sense if an indefinite problem is solved) must be studied, and the bounds for the required number of Krylov subspace iterations can become independent of the mesh size and other important parameters. This is certainly of great importance, but it does not automatically represent a solution to the challenge of identifying efficient preconditioning. Efficiency of the preconditioning in this approach requires that the convergence bounds based on a single number characteristics, such as the condition number, guarantee sufficient accuracy of the computed approximation to the solution within an acceptable number of iterations.

Since Krylov subspace methods are strongly nonlinear in the input data (matrix and the initial residual), more information about the spectrum is needed\footnote{Here we assume that the system matrix is Hermitian, otherwise the spectral information may not be descriptive for convergence of Krylov subspace methods; see \cite{GrePtaStr96,GreStr94}.}
in order to capture the actual convergence behavior with its desirable superlinear character.
This has been pointed out by several studies \cite{Axe86,Axe86b,Jen77,JenMal78,SluVor86,Nie13}, and the acceleration of the convergence of the method of conjugate gradients (CG) has been linked with the presence of large outlying eigenvalues and clustering of the eigenvalues. Since Krylov subspace methods for systems with Hermitian matrices use short recurrences, exact arithmetic considerations must be complemented with a thorough rounding error analysis, otherwise it can in practice be misleading or even completely useless. The deterioration of convergence due to rounding errors in the presence of large outlying eigenvalues has been reported, based on experiments, already in \cite{Lan52}; see also \cite{EngGiRuSt59}, \cite[p.~72]{Jen77}, the discussion in \cite[p.~559]{SluVor86} and the summary in \cite[Section~5.6.4, pp.~279--280]{LSB13}. 

In investigating the convergence behavior of Krylov subspace methods for Hermitian problems, we thus have to deal with two phenomena acting against each other. Large outlying eigenvalues (or well-separated clusters of large eigenvalues) can in theory, assuming exact arithmetic, be linked with acceleration of CG convergence. However, in practice, using finite precision computations, it can cause deterioration of the convergence rate. This intriguing situation has been fully understood thanks to the seminal work of Greenbaum \cite{Gre89} with the fundamental preceeding analysis of the Lanczos method by Paige \cite{Pai71,Pai80}; see also \cite{1992GreenStra,Str91,MeuStr06,Meu06} and the recent paper \cite{GerStr14} that addresses the question of validity of the CG composite convergence bounds based on the so-called effective condition number.
For general non-Hermitian matrices, spectral information may not be descriptive; see, e.g., \cite{GreStr94, GrePtaStr96} and \cite[Section~5.7]{LSB13}.

We will briefly outline the mathematical background behind the understanding of the CG convergence behavior. For Hermitian positive definite matrices (in infinite dimension, for self-adjoint, bounded, and coercive operators) CG can be associated with the Gauss-Christoffel quadrature of the Riemann-Stieltjes integral
\[\int \lambda^{-1}\,d\omega(\lambda);\]
see \cite{Gre89}, \cite[Section~14]{HesSti52}, \cite[Section~3.5 and Chapter~5]{LSB13}, \cite[Section~5.2 and Chapter~11]{MSB15}.
The nondecreasing and right continuous distribution function $\omega(\lambda)$ is given by the spectral decomposition of the given matrix $\alg{B}$ (operator) and the normalized initial residual $\alg{q}$,
\[\alg{B} = \sum \lambda_l\alg{v}_l\alg{v}_l^* \ \left(= \int \lambda\, d \bc{E}_{\lambda}\right),\]
where $\bc{E}_{\lambda}$ is the spectral function representing a family of projections,
\[
\alg{q}^*\alg{B}\alg{q} = \sum\lambda_l {|\alg{v}_l^*\alg{q}|}^2\equiv \sum \lambda_l\, \omega_l \ \left(= \int\lambda\, d\omega(\lambda)\right),
\]
$\omega_l = |\alg{v}_l^*\alg{q}|^2$, $(d\omega(\lambda) = \alg{q}^*d \bc{E}_{\lambda} \alg{q})$; see \cite[Chapter II, Section~7]{vNeu96} or \cite[Chapter III]{VorB65}. For more references on this topic, see \cite[Section~5.2]{MSB15}.
As a consequence, which has been observed in many experiments, preconditioning that leads to favorable distributions of the eigenvalues of the preconditioned (Hermitian) matrix can lead to much faster convergence than preconditioning that only focuses on {minimizing} the condition number. (As pointed out above, any analysis that aims at relevance to practical computations must also include effects of rounding errors).

Motivated by these facts and the results in \cite{Nie09}, the purpose of
this paper is to show that approximations of {\em all} the eigenvalues of
a classical generalized eigenvalue problem are readily available. 
More specifically, assuming that the function $k(x)$ is uniformly positive, bounded and measurable, we will study finite element (FE) discretizations of
\begin{equation}
\label{eq:A1}
\begin{split}
\nabla \cdot (k(x) \nabla u) &= \lambda \Delta u \quad \mbox{in } \Omega \subset \mathbb{R}^d, \\
u &= 0 \quad \mbox{on } \partial \Omega,
\end{split}
\end{equation}
$d=1, \, 2$ or $3$, which yields a system of linear equations in the form
\begin{equation}
\label{A2}
\alg{A} \alg{v} = \lambda \alg{L} \alg{v}.
\end{equation}

As mentioned above, mathematical properties of the continuous problem~\cref{eq:A1} are studied in \cite{Nie09}. In particular, the authors of that paper prove that\footnote{The spectrum of the operator $\mathcal{L}^{-1} \mathcal{A}$ on an infinite dimensional
normed linear space
is defined as
\[
\mathrm{sp}(\mathcal{L}^{-1} \mathcal{A})
\equiv \left\{ \lambda \in \mathbb{C}; \,  \mathcal{L}^{-1} \mathcal{A} - \lambda \mathcal{I} 
\mbox{
does not have a bounded inversion
} \right\}.
\]
}
\[
k(x) \in \mathrm{sp}(\mathcal{L}^{-1} \mathcal{A})
\]
for all $x \in \Omega$ at which {$k(x)$} is {\em continuous}, where
\begin{align}
\label{eq:intro:A}
& \mathcal{A}: H_0^1(\Omega) \mapsto H^{-1}(\Omega), \quad \langle \mathcal{A} u, v \rangle = \int_{\Omega} k \nabla u  \cdot  \nabla v, \quad u,v \in H_0^1(\Omega),  \\ 
& \mathcal{L}: H_0^1(\Omega) \mapsto H^{-1}(\Omega), \quad \langle \mathcal{L} u, v \rangle = \int_{\Omega} \nabla u  \cdot  \nabla v, \quad u,v \in H_0^1(\Omega). \label{eq:intro:L}
\end{align}
The authors also conjecture that the spectrum of the discretized preconditioned operator $\alg{L}^{-1} \alg{A}$ can be approximated by the nodal values of $k(x)$.
In the present text we show, without the continuity assumption on the coefficient function, how the 
function values of $k(x)$ {are related to the generalized spectrum of the discretized operators (matrices) in \cref{A2}}. Our main results state that:
\begin{itemize}
	\item There exists a (potentially non-unique) pairing of the eigenvalues of $\alg{L}^{-1} \alg{A}$ and the intervals determined by the images under $k(x)$ of the supports of the FE nodal basis functions; see \cref{th:theorem} in \cref{analysis}.
\item The function values of {$k(x)$} at the nodes of the {finite element} grid can be paired with the individual eigenvalues of the discrete preconditioned operator $\alg{L}^{-1} \alg{A}$. Furthermore, these functions values yield accurate approximations of the eigenvalues;
see \cref{th:taylor} in \cref{analysis}.
%
\end{itemize}

The text is organized as follows. Notation, assumptions and a motivating example are presented in \cref{notation}. \Cref{analysis} contains theoretical results.
The proof of the pairing in \cref{th:theorem} uses the classical Hall's theorem from the theory of bipartite graphs. \Cref{th:taylor} then follows as a simple consequence. The numerical experiments in \cref{experiments}  illustrate the results of our analysis. Moreover, using \cref{th:theorem}, the discussion at the end of \cref{experiments} explains the CG convergence behavior observed in the example presented in \cref{notation}. The text closes with concluding remarks in \cref{remarks}.

\section{{Notation and an introductory example}}
\label{notation}
We consider a self-adjoint second order elliptic PDE in the form
\begin{align}
\label{eq:problem}
-\nabla\cdot (k(x) \nabla u) &= f\quad\mbox{for}\  x\in\Omega, \\
u &= 0\quad\mbox{for}\  x\in\partial\Omega,\nonumber
\end{align}
and the corresponding generalized eigenvalue problem
\cref{eq:A1}
with the domain $\Omega\subset \nmbr{R}^d$, $d \in \{1,2,3\}$ and the given function $f\in L^2(\Omega)$. We assume that the real valued scalar function $k( x): \nmbr{R}^{d}\rightarrow \nmbr{R}$ is {measurable and bounded, i.e., $k(x) \in L^{\infty}(\Omega)$, and that it is} uniformly positive, i.e.,
\[
k(x)\geq\alpha>0,\quad x\in\Omega.
\]
Let $V\equiv H_0^1(\Omega)$ denote  the  Sobolev space of functions defined on $\Omega$ with zero trace at $\partial\Omega$ and  with the standard inner product. 
The weak formulations of the problems \cref{eq:problem,eq:A1} are to seek $u\in V$, respectively $u\in V$ and $\lambda\in \nmbr{R}$, such that
\begin{align}\label{eq:problem_inf}
\bc{A}u = f, \qquad\mbox{respectively}\qquad\bc{A}u = \lambda\bc{L}u
\end{align}
where
$\bc{A},\, \bc{L} : V \rightarrow V^{\#}$, $f \in V^{\#}$ are defined in \cref{eq:intro:A,eq:intro:L} and the function $f\in L^2(\Omega)$ is identified with the associated linear functional $f \in V^{\#}$ defined by
\begin{equation}
	\dual{f}{v} \equiv \int_{\Omega} fv\, .
\end{equation}
Discretization via the conforming finite-element method leads to the discrete operators
\[
\bc{A}_h,\, \bc{L}_h : V_h \rightarrow V_h^{\#}
\]
where the  finite dimensional subspace $V_h$ is  spanned by the polynomial discretization basis functions
$\phi_1,\ldots,\phi_N$ with the local supports 
\[
\bc{T}_i = \mbox{supp}(\phi_i), \, i = 1,\ldots,N.
\]
The matrix representations $\alg{A}_h$ and $\alg{L}_h$ {are} defined as
\begin{align}\label{eq:problem_dis:A}
	[\alg{A}_h]_{ij} &= \dual{\bc{A}_h\phi_j}{\phi_i} = \int_{\Omega} \nabla \phi_i\cdot k\nabla\phi_j, \\
	[\alg{L}_h]_{ij} &= \dual{\bc{L}_h\phi_j}{\phi_i} = \int_{\Omega} \nabla \phi_i\cdot \nabla \phi_j,\, \quad i,j=1,\ldots,N.
\label{eq:problem_dis:L}
\end{align}
In the text below we will, for the sake of simple notation, omit the subscript $h$ and write $\alg{A}\equiv\alg{A}_h$ and $\alg{L}\equiv\alg{L}_h$.

\begin{figure}[!ht]
	\centering
	\includegraphics[width=0.495\textwidth]{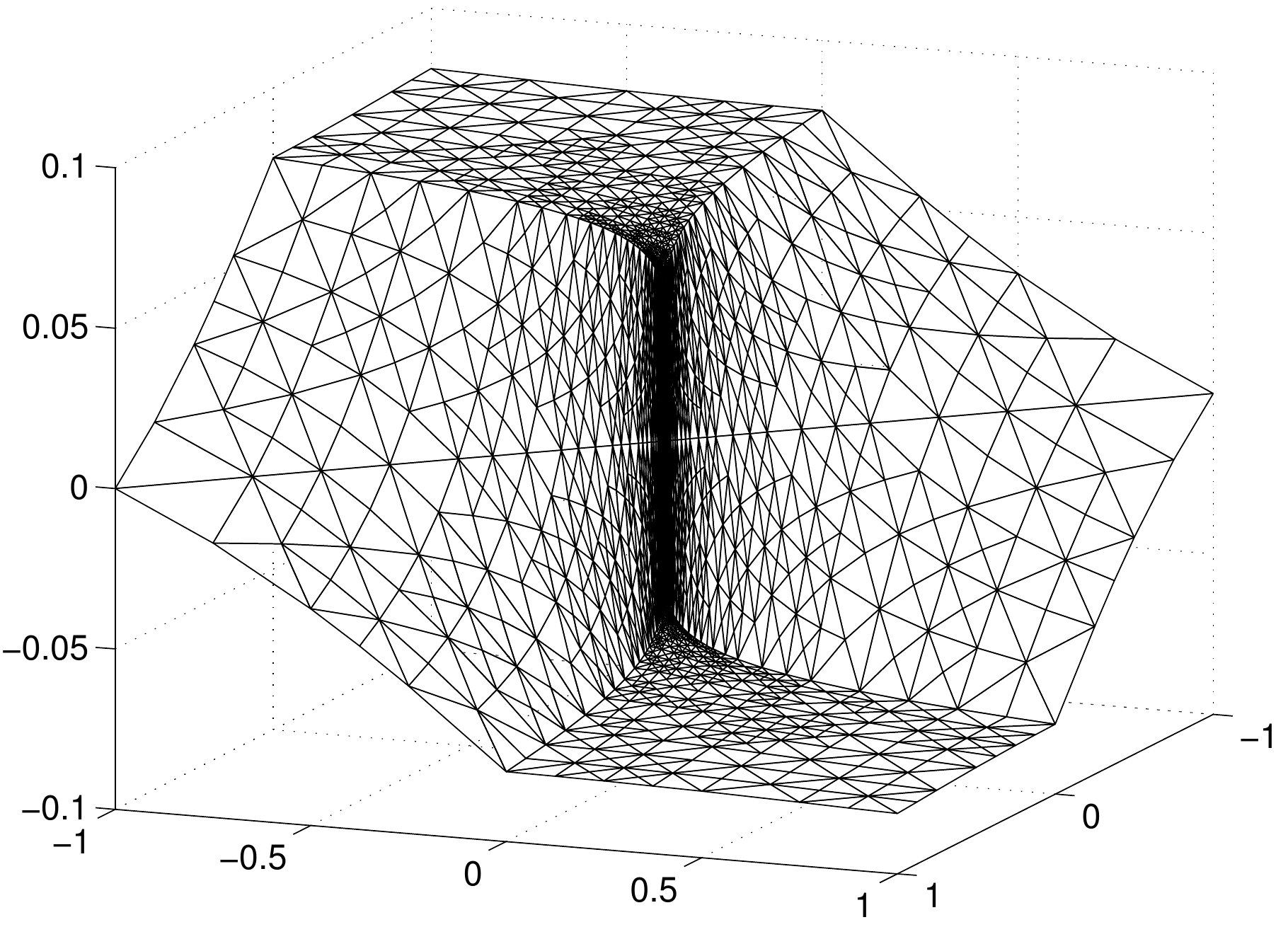}
	\includegraphics[width=0.495\textwidth]{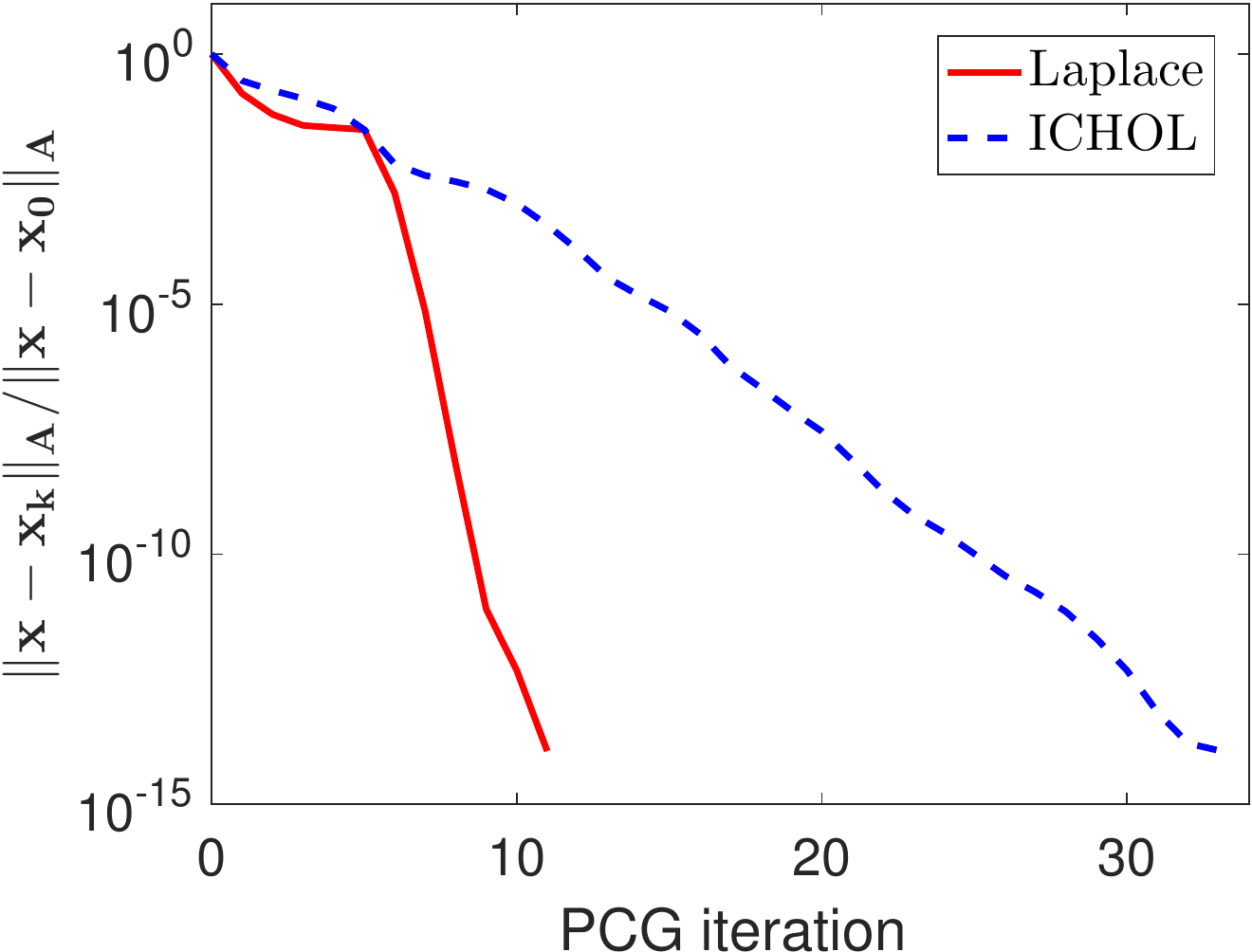}
	\caption{Left: The solution of the problem \cref{eq:motivation} {on the background of the linear FE triangulation}. Right: The relative energy norm of the {PCG} error as a function of the iteration steps. The Laplace {operator} preconditioning (solid line) is much more efficient than the incomplete Choleski preconditioning (dashed line), despite the fact that the condition numbers are $161.54$ and close to $16$, respectively. This can be explained by the differences in the associated distribution functions (see the end of \cref{experiments} below).}
	\label{fig:inhomogeneousIIsolution}
\end{figure}
\subsection*{An example}
The following example illustrates in detail the motivation {outlined} in section~1, i.e. that the condition number may be misleading in characterization of the convergence behavior of the CG method.  Consider the boundary value problem
\begin{eqnarray}
\label{eq:motivation}
	- \nabla \cdot (k(x) \nabla u) &=& 0 \quad \text{in} \ \Omega \,,\qquad u = u_D \quad \text{on} \ \partial \Omega\,,
\end{eqnarray}
where the domain $\Omega \equiv (-1,1) \times (-1,1)$ is divided into four subdomains $\Omega_i$, $i=1, \, 2, \, 3, \, 4$, corresponding to the axis quadrants numbered counterclockwise. Let $k(x)$ be piecewise constant on the individual subdomains $\Omega_i$, $k_1 = k_3 \approx 161.45$, $k_2 = k_4 = 1$. The Dirichlet boundary conditions are described in \cite[Section~5.3]{MorNocSie02}.

The numerical solution $u$ of this problem and the linear FE discretization, using the standard uniform triangulation, are shown in the left part of \cref{fig:inhomogeneousIIsolution}. The resulting algebraic {problem} $\alg{A}\alg{x}=\alg{b}$ is solved by the preconditioned conjugate gradient method (PCG).
In the right panel of \cref{fig:inhomogeneousIIsolution} we see the relative energy norm of the error as a function of iteration steps for the Laplace operator preconditioning (solid line) and for the {preconditioning using the} algebraic incomplete Choleski factorization  of the matrix $\alg{A}$ (ICHOL) with the drop-off tolerance $10^{-2}$ (dashed line) {{where} the problem {has} $N = 3969$ degrees of freedom}.
Despite the fact that the spectral condition number $\lambda_{max}/\lambda_{min}$ of the symmetrized preconditioned matrix for the Laplace operator preconditioning is an order of magnitude larger than for the  ICHOL preconditioning, close to $161$ and close to $16$, respectively, PCG with the Laplace operator preconditioning clearly demonstrates much faster convergence.
This is due to the differences in the distribution of the eigenvalues with the nonnegligible components of the initial residuals  in the direction of the associated eigenvectors and effects of rounding errors.

\begin{figure}[!ht]
	\centering
	\includegraphics[width=0.49\textwidth]{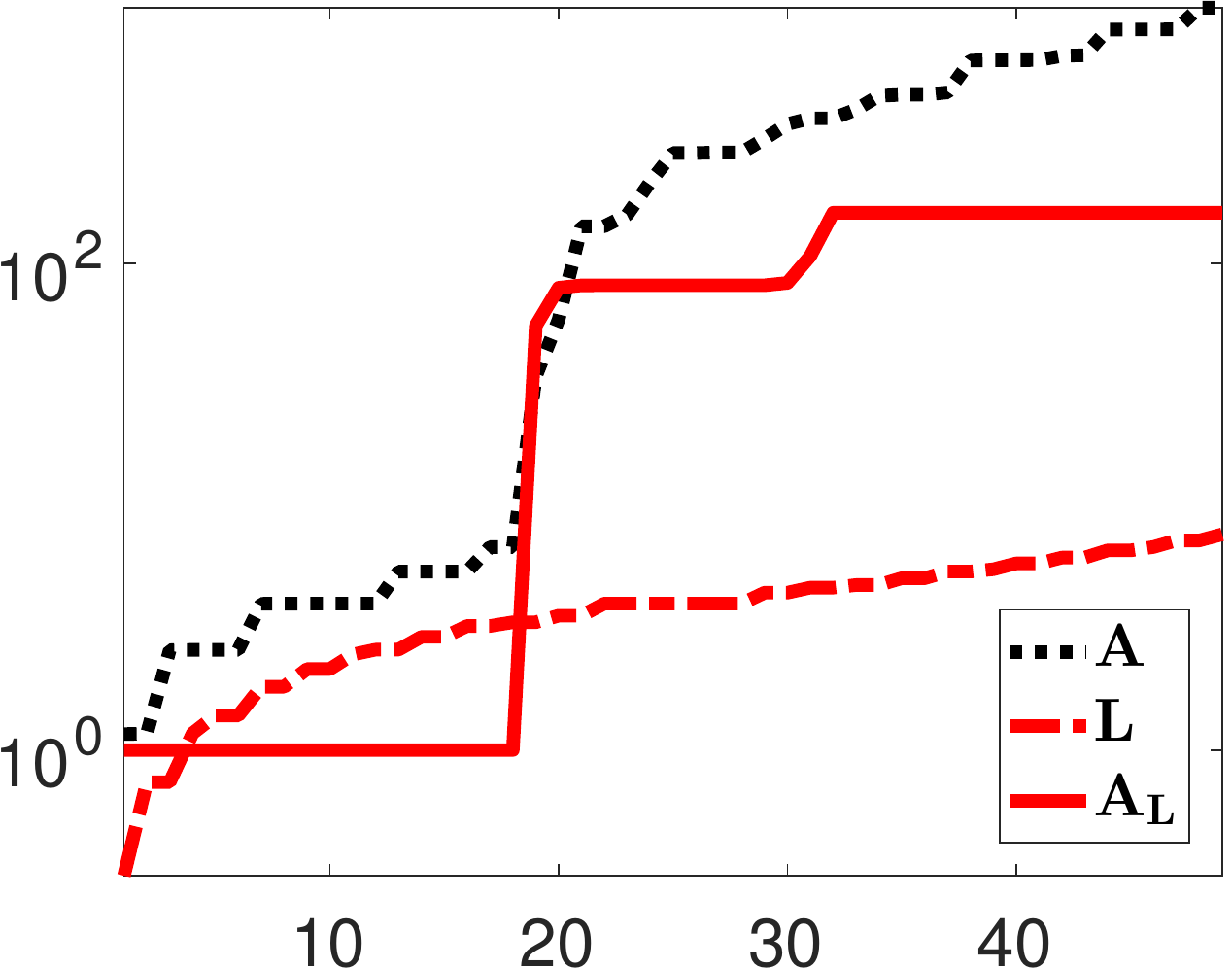}
	\includegraphics[width=0.5\textwidth]{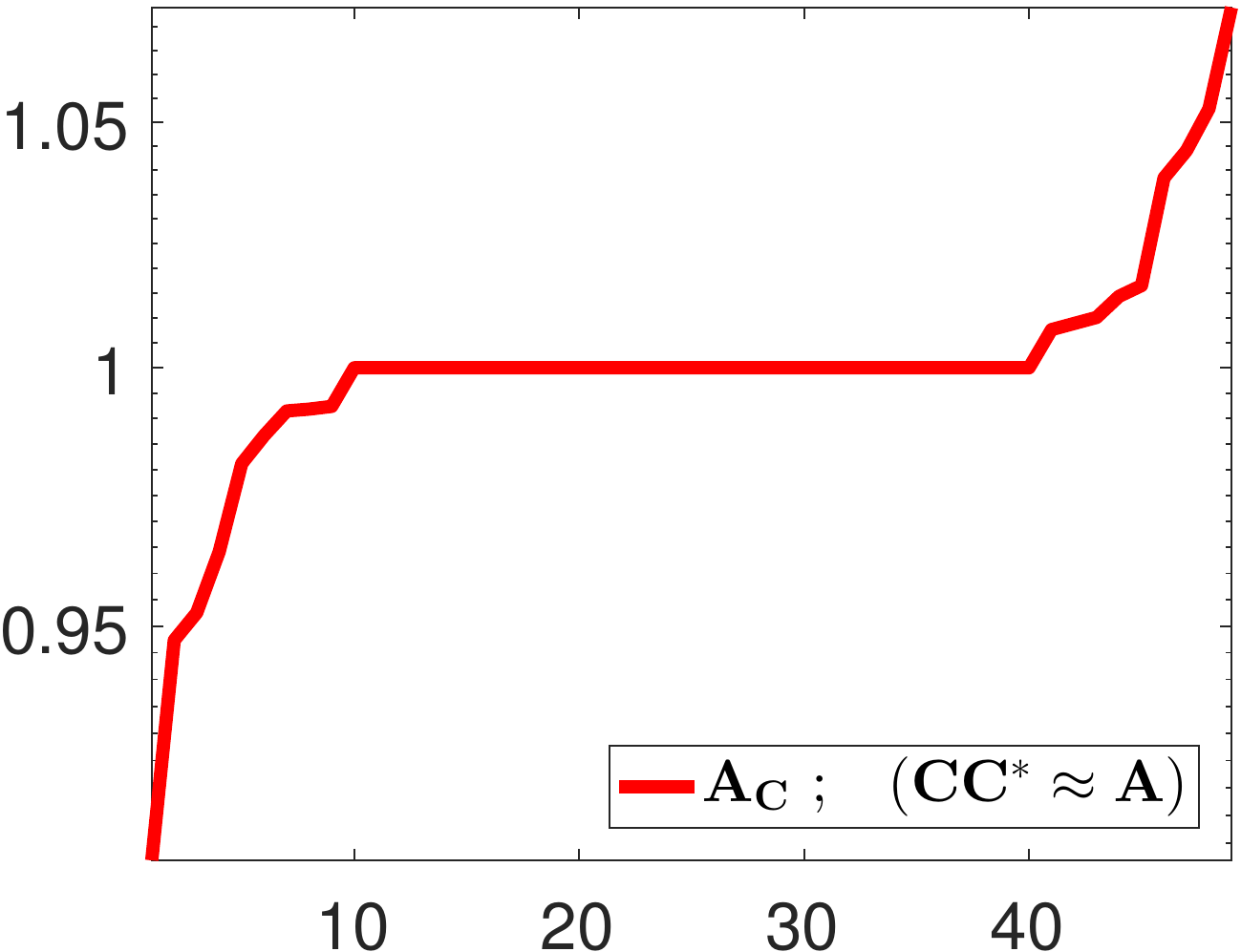}
	
	\hspace{0.005\textwidth}	
	\includegraphics[width=0.48\textwidth]{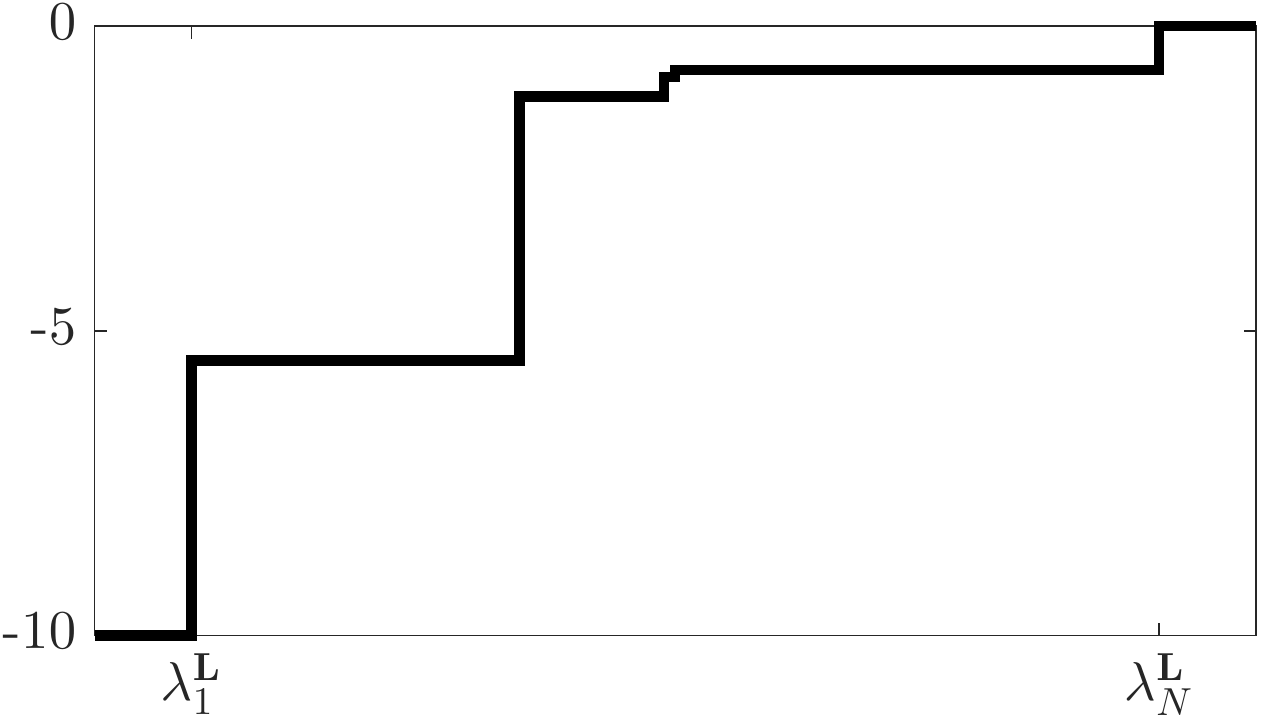}
	\hspace{0.005\textwidth}
	\includegraphics[width=0.48\textwidth]{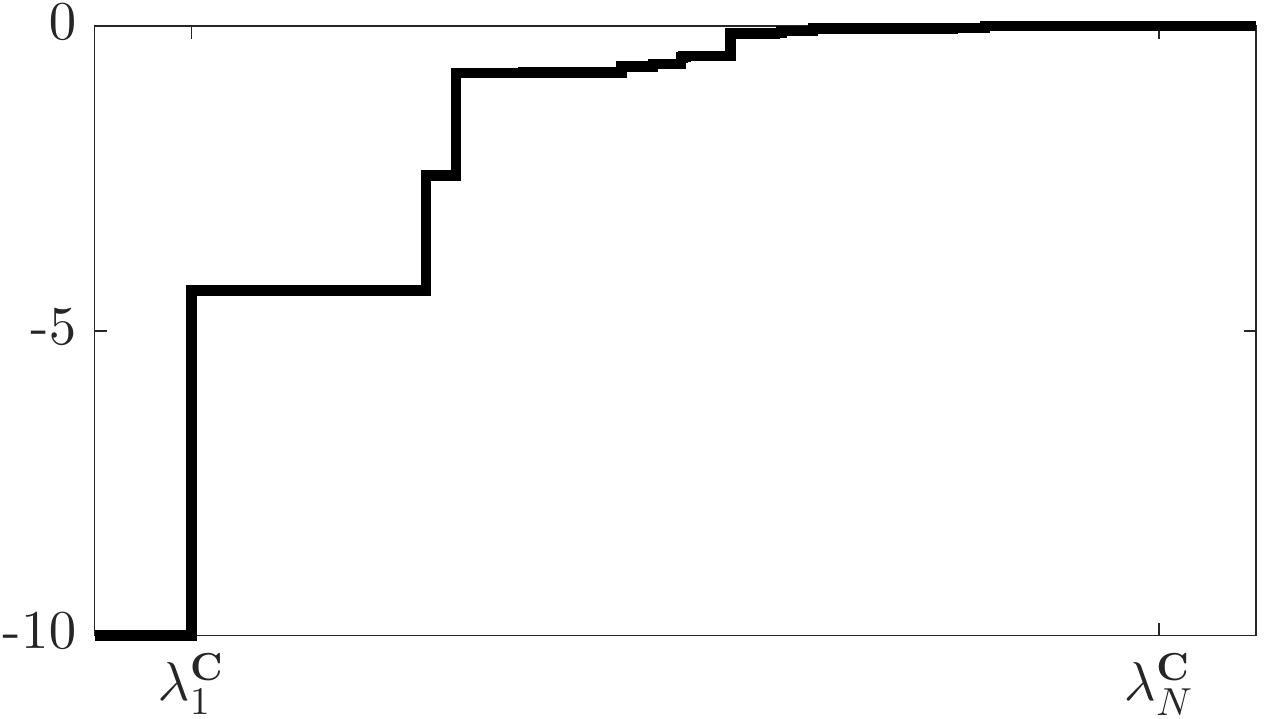}	
	\caption{Top: Comparison of the spectra of the  matrices $\alg{A}$, {$\alg{A}_{\alg{L}}$ and $\alg{A}_{\alg{C}}$}, $N = 49$ degrees of freedom. Due to a small drop-off tolerance, the eigenvalues of $\alg{A}$ and $\alg{C}\alg{C}^*$ are graphically indistinguishable. Therefore the right part only shows the eigenvalues of $\alg{A}_{\alg{C}}$ (using a different scale than the left part of the figure). Bottom: Comparison of the distribution functions $\omega^{\alg{L}}(\lambda)$ (left) and $\omega^{\alg{C}}(\lambda)$ (right) associated with the preconditioned problems. The vertical axes are in the logarithmic scale and $\lambda_1^{\alg{L}} = 1$, $\lambda_N^{\alg{L}} = 161.45$, $\lambda_1^{\alg{C}} = 0.91$, $\lambda_N^{\alg{C}} = 1.07$.}
	\label{fig:spectra}
\end{figure}

The spectra and distribution functions associated with the discretized preconditioned problems are given in \cref{fig:spectra} for $N = 49$ degrees of freedom and in \cref{fig:spectra_large} for $N = 3969$ degrees of freedom.
Here, $\alg{L} = \alg{L}^{1/2}\alg{L}^{1/2}$ is the matrix associated with the discretized Laplace operator and $\alg{C}\alg{C}^*\approx\alg{A}$ is the matrix resulting from ICHOL using the drop-off tolerance $10^{-2}$, with the eigenvalues and eigenvectors of the associated generalized eigenvalue problems (see \cref{A2})
\begin{align*}
\alg{A}\alg{v}^{\alg{L}}_i &= \lambda_i^{\alg{L}}\alg{L}\alg{v}^{\alg{L}}_i,\quad i=1,\ldots,N, \\
\alg{A}\alg{v}^{\alg{C}}_i &= \lambda_i^{\alg{C}}\alg{C}\alg{C}^*\alg{v}^{\alg{C}}_i,\quad i=1,\ldots,N.
\end{align*}
The weights of the distribution function $\omega^{\alg{L}}(\lambda)$, respectively, $\omega^{\alg{C}}(\lambda)$, associated with the eigenvalues $\lambda_i^{\alg{L}}$, respectively, $\lambda_i^{\alg{C}}$,  $i = 1\ldots,N$, related to the preconditioned algebraic systems
\begin{figure}[!ht]
	\centering
	\includegraphics[width=0.48\textwidth]{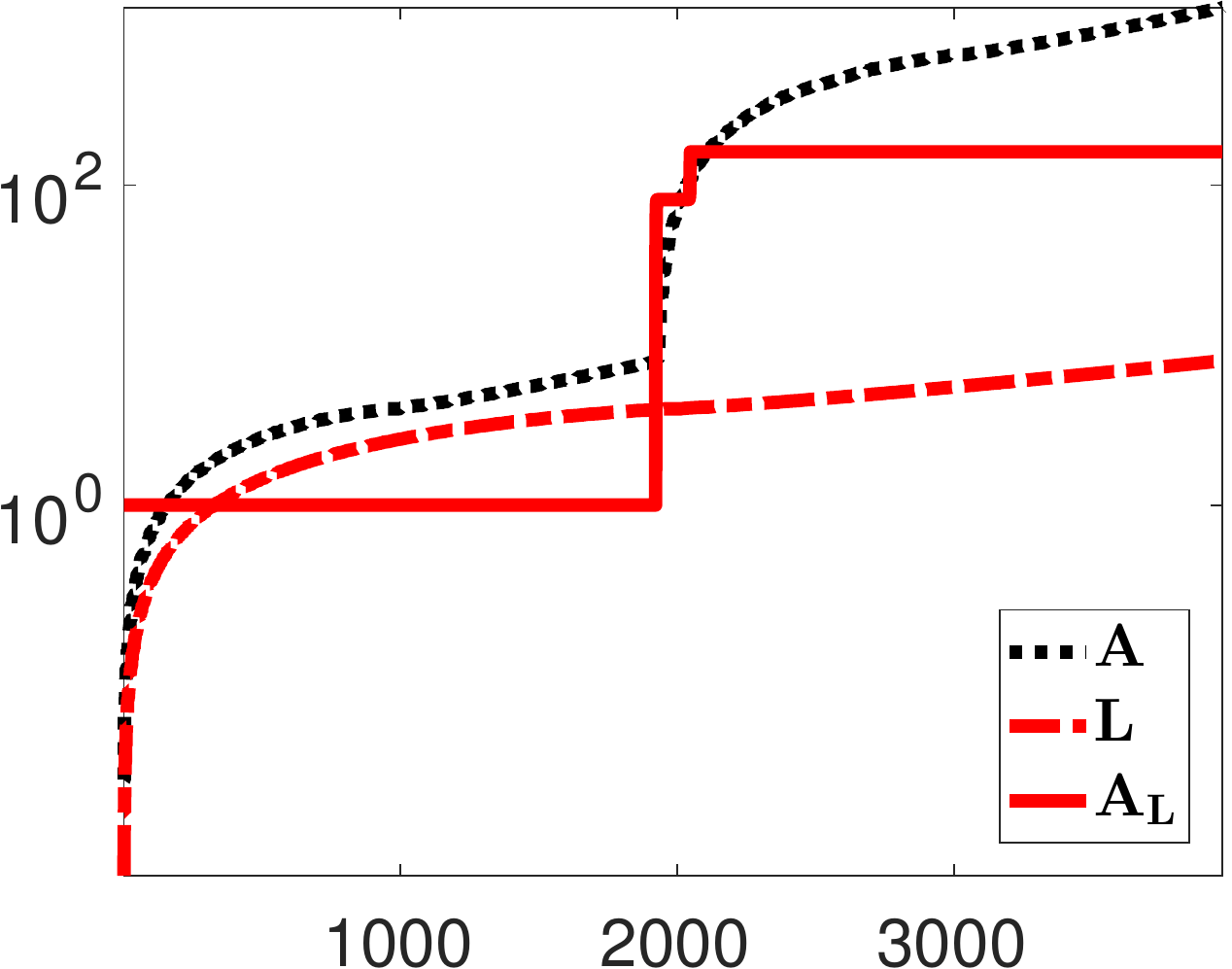}
	\includegraphics[width=0.48\textwidth]{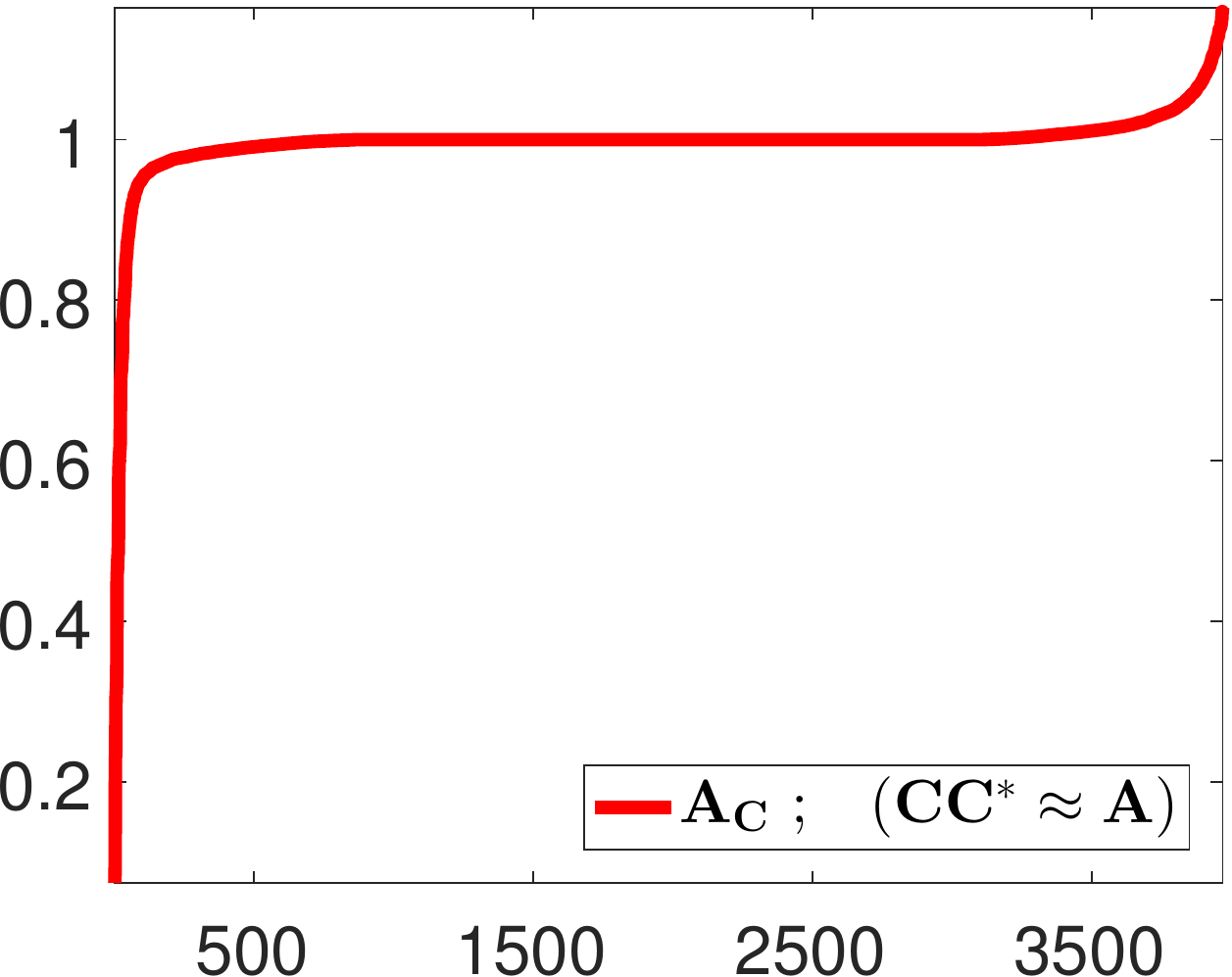}
	
	\includegraphics[width=0.48\textwidth]{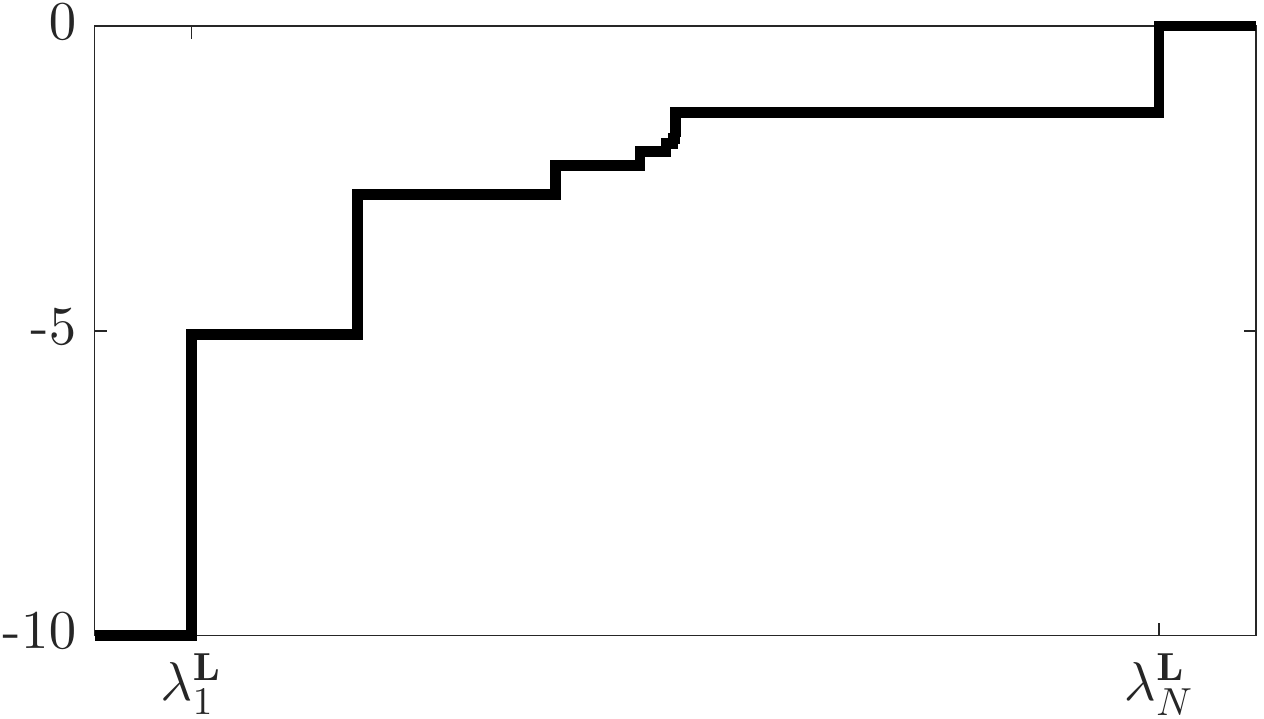}
	\includegraphics[width=0.48\textwidth]{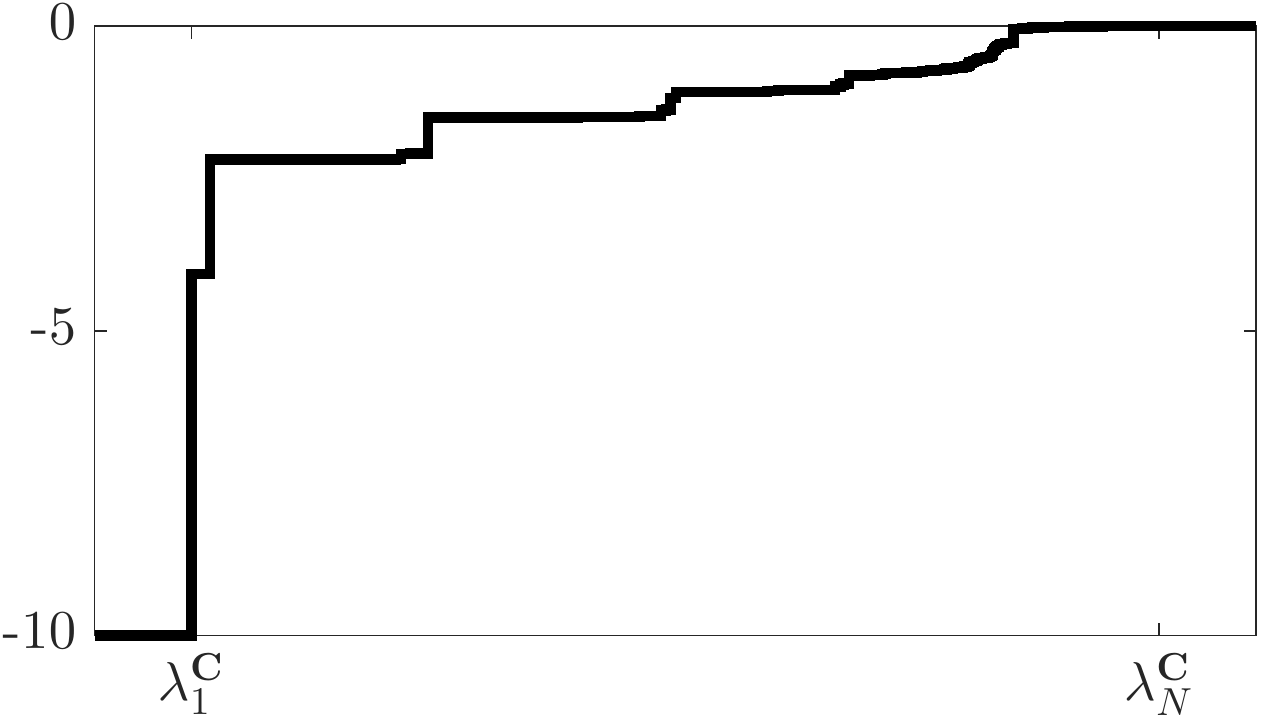}	
	
	\caption{Top: Comparison of the spectra of the  matrices $\alg{A}$, {$\alg{A}_{\alg{L}}$ and $\alg{A}_{\alg{C}}$}, $N = 3969$ degrees of freedom. Due to a small drop-off tolerance, the eigenvalues of $\alg{A}$ and $\alg{C}\alg{C}^*$ are graphically indistinguishable. Therefore the right part only shows the eigenvalues of $\alg{A}_{\alg{C}}$ (using a different scale than the left part of the figure). Bottom: Comparison of the distribution functions $\omega^{\alg{L}}(\lambda)$ (left) and $\omega^{\alg{C}}(\lambda)$ (right) associated with the preconditioned problems. The vertical axes are in the logarithmic scale and $\lambda_1^{\alg{L}} = 1$, $\lambda_N^{\alg{L}} = 161.45$, $\lambda_1^{\alg{C}} = 7.4\times10^{-2}$, $\lambda_N^{\alg{C}} = 1.16$.}
	\label{fig:spectra_large}
\end{figure}
\begin{align*}
\alg{A}_{\alg{L}}\,(\alg{L}^{1/2}\alg{x}) &= \alg{L}^{-1/2}\alg{b},\quad \alg{A}_{\alg{L}} = \alg{L}^{-1/2}\alg{A}\alg{L}^{-1/2},
\end{align*}
respectively
\begin{align*}
\alg{A}_{\alg{C}}\,(\alg{C}^{*}\alg{x}) &= \alg{C}^{-1}\alg{b},\quad \alg{A}_{\alg{C}} = \alg{C}^{-1}\alg{A}\alg{C}^{-*},
\end{align*}
are given by
\begin{equation}
\begin{aligned}
\omega_i^{\alg{L}} = |(\bar{\alg{v}}_i^{\alg{L}})^*\alg{q}^{\alg{L}}|^2, \quad i = 1,\ldots,N,\\
\omega_i^{\alg{C}} = |(\bar{\alg{v}}_i^{\alg{C}})^*\alg{q}^{\alg{C}}|^2,\quad i = 1,\ldots,N.
\end{aligned}
\label{eq:weights}
\end{equation}
Here,
\[
\bar{\alg{v}}_i^{\alg{L}}
	= \frac{\alg{L}^{1/2}\alg{v}_i^{\alg{L}}}{\|\alg{L}^{1/2}\alg{v}_i^{\alg{L}}\|}\,
	\quad \mbox{and} \quad
\bar{\alg{v}}_i^{\alg{C}}
	= \frac{\alg{C}^{*}\alg{v}_i^{\alg{C}}}{\|\alg{C}^{*}\alg{v}_i^{\alg{C}}\|}\,\]
are the eigenvectors of the Hermitian and positive definite matrix $\alg{A}_{\alg{L}}$, respectively, $\alg{A}_{\alg{C}}$,
and
\[
\alg{q}^{\alg{L}} = \frac{\alg{L}^{-1/2}\alg{b}}{\|\alg{L}^{-1/2}\alg{b}\|},\quad
\alg{q}^{\alg{C}} = \frac{\alg{C}^{-1}\alg{b}}{\|\alg{C}^{-1}\alg{b}\|}.
\]
(We use the initial guess $\alg{x}_0 = 0$).
The distribution function $\omega^{\alg{C}}(\lambda)$ has its points of increase much more evenly distributed in the spectral interval $[\lambda_1(\alg{A}_{\alg{C}}), \lambda_N(\alg{A}_{\alg{C}})]$, which leads to a difference in the PCG convergence behavior.
We will return to this issue, and offer a full explanation of the observed CG convergence behavior, after proving the main results and presenting their numerical illustrations.



\section{Analysis}
\label{analysis}
{As mentioned above, {we will not only show that some function values of $k(x)$ are related to the spectrum of  $\alg{L}^{-1}\alg{A}$, but that there exists a one-to-one correspondence, i.e., a pairing, between the individual eigenvalues of $\alg{L}^{-1}\alg{A}$ and quantities given by the function values of $k(x)$ in relation to the supports of the FE basis functions. The proof does not require that $k(x)$ is continuous.}
If, moreover, $k(x)$ is constant on a part of the domain $\Omega$ that contains fully the supports of one or more basis functions, then the {function} value of $k(x)$ determines the associated eigenvalue {\em exactly} and the number of the involved supports bounds from below the multiplicity of the associated eigenvalue. If $k(x)$ is slowly changing over the support of some basis function, then we get a very accurate localization of the associated eigenvalue.}

Our approach is {based} upon the intervals
\begin{equation}\label{eq:range}
k(\bc{T}_j) \equiv [ \min_{ x\in\bc{T}_j} k(x),\max_{ x\in\bc{T}_j} k(x)],\quad j = 1,\ldots,N,
\end{equation}
where $\bc{T}_j=\mbox{supp}(\phi_j)$.\footnote{If $k(x)$ is continuous on $\bc{T}_j$, then $k(\bc{T}_{j})$ coincides with the range of $k(x)$ over $\bc{T}_j$.}
We will first formulate two main results. \cref{th:theorem} localizes the positions of {\em all} the individual eigenvalues of the matrix $\alg{L}^{-1}\alg{A}$ by pairing them with the intervals $k(\bc{T}_j)$ given in \cref{eq:range}. Using the given pairing, \cref{th:taylor} describes the closeness of the eigenvalues to the nodal {function} values of the scalar function $k(x)$.

{The proof of \cref{th:theorem} {combines perturbation theory for matrices with a classical result from the theory of bipartite graphs. 
For clarity of exposition, the proof will be presented} after stating the corollaries of \cref{th:theorem}.} 

\begin{figure}[h!t]
\centering
\begin{tikzpicture}[scale = 0.9]
\tikzstyle{every node}=[font=\small]
\draw (-3,0) -- (-0.5,0);
\draw[loosely dotted] (-0.5,0) -- (1,0);
\draw (1,0) -- (5,0);
\draw[loosely dotted] (5,0) -- (6.5,0);
\draw (6.5,0) -- (10.5,0);
\draw [blue,dashed,thick,domain=0:360] plot [smooth] ({-2+.3*cos(\x)}, {.3*sin(\x)});
\draw [blue,dashed,thick,domain=0:360] plot [smooth] ({-1.5+.5*cos(\x)}, {.5*sin(\x)});
\draw [blue,dashed,thick,domain=0:360] plot [smooth] ({-.8+.6*cos(\x)}, {.6*sin(\x)});
\draw [blue,dashed,thick,domain=0:360] plot [smooth] ({8.25+.75*cos(\x)}, {.75*sin(\x)});
\draw [blue,dashed,thick,domain=0:360] plot [smooth] ({7.5+1.5*cos(\x)}, {1.5*sin(\x)});
\draw [blue,dashed,thick,domain=0:360] plot [smooth] ({9+1*cos(\x)}, {1*sin(\x)});
\draw [blue,dashed,thick,domain=0:360] plot [smooth] ({2+1.5*cos(\x)}, {1.5*sin(\x)});
\draw [blue,dashed,thick,domain=0:360] plot [smooth] ({4+1.5*cos(\x)}, {1.5*sin(\x)});
\coordinate[circle,inner sep=1.4pt,fill=red] () at (-1.9,0);
\coordinate[circle,inner sep=1.4pt,fill=red] () at (-1.75,0);
\coordinate[circle,inner sep=1.4pt,fill=red] () at (-1.1,0);
\coordinate[circle,inner sep=1.4pt,fill=red] () at (2.1,0);
\coordinate[circle,inner sep=1.4pt,fill=red] (${}^*$) at (4.5,0);
\coordinate[circle,inner sep=1.4pt,fill=red] () at (9.5,0);
\coordinate[circle,inner sep=1.4pt,fill=red] () at (8.1,0);
\coordinate[circle,inner sep=1.4pt,fill=red] () at (7.8,0);
\end{tikzpicture}
\caption{Illustration of \cref{th:theorem}. The diameters of the dashed circles indicate the {size  of the} intervals {$k({\bc{T}_j})$}, $j = 1,\ldots,N$. The dots represent the eigenvalues $\lambda_j,\ j = 1,\ldots, N$, of the matrix $\alg{L}^{-1}\alg{A}$. We can find {a} pairing between the intervals $k({\bc{T}_j})$ and the eigenvalues $\lambda_i$, but {the pairing may not be uniquely determined}.} 
\label{fig:theorem}
\end{figure}
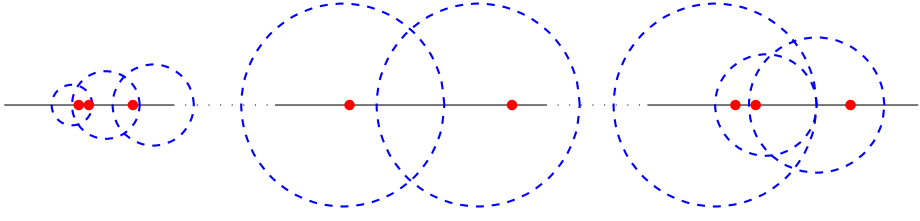
%
\begin{theorem}[{Pairing {the} eigenvalues and
the intervals $k(\bc{T}_{j})$, $j=1,\ldots,N$}]\label{th:theorem}
Using the previous notation, let $0< \lambda_1\leq\lambda_2\leq\ldots\leq\lambda_N$ be the eigenvalues of $\alg{L}^{-1}\alg{A}$, where $\alg{A}$ and $\alg{L}$ are defined by \cref{eq:problem_dis:A,eq:problem_dis:L} respectively (with the subscript $h$ dropped). As in \cref{eq:A1}, let $k(x)$ be measurable and bounded, i.e., $k(x)\in L^{\infty}(\Omega)$. Then there exists a (possibly non-unique) permutation $\pi$ such that the eigenvalues of the matrix $\alg{L}^{-1}\alg{A}$ satisfy
\begin{equation}\label{eq:theorem}
\lambda_{\pi(j)} \in k({\bc{T}_{j}}),\quad j = 1,\ldots,N,
\end{equation}
{where the intervals $k({\bc{T}_{j}})$ are defined {in} \cref{eq:range}.}
\end{theorem}

\noindent
The statement is illustrated in \cref{fig:theorem}.  {The proof of the following corollary uses the one-to-one pairing of the intervals \cref{eq:range}, and therefore also the values of $k(x)$ at the nodes of the discretization mesh,  with the eigenvalues $\lambda_{\pi(j)}$.}

\begin{corollary}[{Pairing the eigenvales and} the nodal values]\label{th:taylor}
Using the notation of \cref{th:theorem}, consider any discretization mesh node $ \hat{x}$ such that $\hat{x}\in {\bc{T}_j}$.
Then the associated eigenvalue $\lambda_{\pi(j)}$ of the matrix $\alg{L}^{-1}\alg{A}$ satisfies
\begin{equation}\label{eq:bound:loose}
	|\lambda_{\pi(j)} - k( \hat{x})| \leq \max_{ x\in\bc{T}_j}|k( x) - k( \hat{x})|.
\end{equation}
{If, in addition,  $k( x) \in \bc{C}^2({\bc{T}_j})$, then}
\begin{align}\nonumber
	|\lambda_{\pi(j)} - k( \hat{x})|
		&\leq \max_{ x\in\bc{T}_j}|k( x) - k( \hat{x})| \\
			 \label{eq:h_taylor}
		&\leq \hat{h}\|\nabla k( \hat{x})\|
			+ \tfrac{1}{2}\hat{h}^2\max_{ x\in\bc{T}_j}\|D^2k( x)\| 
\end{align}
where $\hat{h} = \max_{ x\in\bc{T}_j}\|x - \hat{x}\|$ {and $D^2k(x)$ is the second order derivative} of the function $k(x)$.\footnote{{See \cite[Section 1.2]{CiaB02} for the definition {of the second order derivative.}}}
\end{corollary}
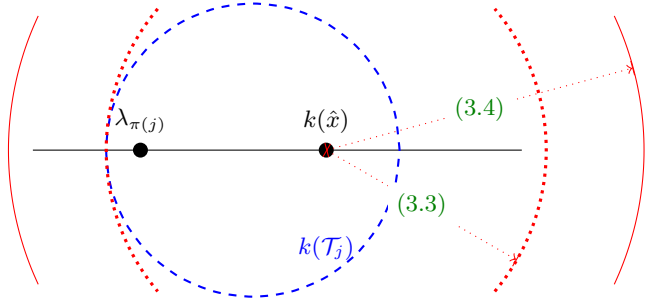
\begin{figure}[h!t]
\centering
\begin{tikzpicture}[scale = 0.65]
\tikzstyle{every node}=[font=\small]
\draw (0,0) -- (10,0);
\coordinate[circle,inner sep=2pt,fill=black,label=above:{$k( \hat{x})$}] (Ki) at (6,0);
\draw [dashed,blue,thick,domain=0:360] plot [smooth] ({4.5+3*cos(\x)}, {3*sin(\x)});
\draw [red,dotted,very thick,domain=-40:40] plot [smooth] ({6+4.5*cos(\x)}, {4.5*sin(\x)});
\draw [red,dotted,very thick,domain=140:220] plot [smooth] ({6+4.5*cos(\x)}, {4.5*sin(\x)});
\draw [red,domain=-25:25] plot [smooth] ({6+6.5*cos(\x)}, {6.5*sin(\x)});
\draw [red,domain=155:205] plot [smooth] ({6+6.5*cos(\x)}, {6.5*sin(\x)});
\node[text=blue] (AAA) at (6,-2) {$k({\mathcal T}_j)$};
\draw[|->|, red,dotted,rotate around={15:(6,0)}] (6,0) -- (12.5,0)
	node [midway,fill=white] {\cref{eq:h_taylor}};
\draw[|->|, red,dotted,rotate around={-30:(6,0)}] (6,0) -- (10.5,0)
	node [midway,fill=white] {\cref{eq:bound:loose}};
\coordinate[circle,inner sep=2pt,fill=black,label=above:{$\lambda_{\pi(j)}$}] (LI) at (2.2,0);
\end{tikzpicture}
\caption{Illustration of \cref{th:taylor}. The relation \cref{eq:theorem} (indicated by the dashed {blue} circle) {can give significantly} better localization of the position of the individual eigenvalues than the bounds \cref{eq:bound:loose} (indicated by the dotted {red} circle) and \cref{eq:h_taylor} (indicated by the solid {red} circle). {When $k(x)$ is constant over $\bc{T}_j$, then $k(\bc{T}_j)$ reduces to one point $\lambda_{\pi(j)}$; see also \cref{eq:bound:loose}. The bound \cref{eq:h_taylor} is weaker that \cref{eq:theorem,eq:bound:loose}, but the evaluation of its first term might be easier in practice.}}
\label{fig:bounds}
\end{figure}
\begin{proof}
Since both $\lambda_{\pi(j)}\in k(\bc{T}_j)$ and $k(\hat{x})\in k(\bc{T}_j)$, it trivially follows that
\[|\lambda_{\pi(j)} - k( \hat{x})| \leq \max_{ x\in\bc{T}_j}|k( x) - k( \hat{x})|.\]
Moreover, for any $x\in\bc{T}_j$, the multidimensional Taylor expansion (see, e.g., \cite[p.~11, Section 1.2]{CiaB02}) gives for {$k( x) \in \bc{C}^2({\bc{T}_j})$} that
\begin{align*}
	k(x)-k(\hat{x})
		& = \nabla k( \hat{x}) (x-\hat{x}) \\
		& + \tfrac{1}{2}D^2k(\hat{x} +\alpha(x-\hat{x}))(x-\hat{x},x-\hat{x})
\end{align*}
where $\alpha\in[0,1]$, with the absolute value {obeying} 
\begin{align*}
|k(x)-k(\hat{x})| &\leq \|\nabla k( \hat{x})\| \|(x-\hat{x})\| \\
		&+ \tfrac{1}{2}\|x-\hat{x}\|^2\|D^2k(\hat{x} +\alpha(x-\hat{x}))\|,
\end{align*}
giving the statement.
\end{proof}

{We now give the proof of} \cref{th:theorem}.
\Cref{th:lemma2} {below} and its \cref{th:cor} identify the groups of eigenvalues in any union of intervals
\begin{equation}\label{eq:interval}
\bar{k}(\bc{T}_{\bc{J}})\equiv \bigcup_{j\in\bc{J}}k(\bc{T}_j),\quad \bc{J}\subset\{1,\ldots,N\}.
\end{equation}
{This enables us to apply Hall's theorem, see \cite[Theorem~5.2]{Bondy1976} or, e.g., \cite[Theorem~1]{Hall35}, to prove \cref{th:theorem}. (For the sake of completeness, we have also formulated Hall's result below in \cref{th:hall}.)}
%
%
\begin{lemma}\label{th:lemma2}
{Using the notation introduced above, let} $\bc{J}\subset\{1,\ldots,N\}$ and $\bc{T}_{\bc{J}} = \cup_{j\in\bc{J}}\bc{T}_j$.
Then there exist at least $p = |\bc{J}|$ eigen\-values $\tilde{\lambda}_1,\ldots,\tilde{\lambda}_p$ of 
$\alg{L}^{-1}\alg{A}$ such that
\begin{equation}\label{eq:lemma2}
\tilde{\lambda}_{\ell} \in [\min_{x\in\bc{T}_\bc{J}} k(x),\max_{x\in\bc{T}_\bc{J}} k(x)],
\quad \ell = 1,\ldots,p.
\end{equation}
\end{lemma}
\begin{proof}
In brief, {the proof {is based on} the theory of eigenvalue perturbations of matrices. We} locally modify the scalar function $k(x)$ by setting it {equal to a positive constant} $K$ in the union $\bc{T}_\bc{J}$ of the supports $\bc{T}_j$, $j\in\bc{J}$. {This will result, after discretization, in a modified matrix $\tilde{\alg{A}}_{\bc{J}}$ such that $K$ is an eigenvalue of $\alg{L}^{-1}\tilde{\alg{A}}_{\bc{J}}$ of at least $p$  multiplicity. An easy bound for the eigenvalues of
\begin{equation}\label{eq:error:matrix}
\alg{L}^{-1}\alg{E}_{\bc{J}},
	\quad \mbox{where} \quad
\alg{E}_{\bc{J}} = \alg{A} - \tilde{\alg{A}}_{\bc{J}},
\end{equation}
combined with a standard perturbation theorem for matrices, then provide a bound for the associated $p$ eigenvalues of $\alg{L}^{-1}\alg{A}$. {A particular} choice of the positive constant $K$ will finish the proof.}

Let
\begin{align*}
\tilde{k}_\bc{J}( x) &=
	\begin{cases}
		K \qquad 	\mbox{for} \quad  x\in\bc{T}_{\bc{J}}, \\
		k( x) \quad \mbox{elsewhere}; 
	\end{cases}
\end{align*}
{with}
\begin{align*}
\dual{\tilde{\bc{A}}_{\bc{J},h}u}{v} &\equiv
	\int_{\Omega} \nabla u\cdot \tilde{k}_\bc{J}\nabla v\, ,\quad u,v\in V_h,
\end{align*}
{where, analogously to \cref{eq:problem_dis:A},}
\begin{align*}
[\tilde{\alg{A}}_{\bc{J}}]_{lj} &=
	\dual{\tilde{\bc{A}}_{\bc{J},h}\phi_j}{\phi_l} =
	\int_{\Omega} \nabla \phi_l\cdot \tilde{k}_{\bc{J}}\nabla\phi_j\,.
\end{align*}

Since $\tilde{k}_{\bc{J}}$ is constant on each $\bc{T}_{j},\, j\in{\bc{J}},$ and the support of the basis function $\phi_j$ is $\bc{T}_j$, it holds for any $v\in V_h$ that
\begin{align*}
\dual{\tilde{\bc{A}}_{\bc{J},h}\phi_j}{v}
	&= \int_{\Omega}   \nabla \phi_j\cdot \tilde{k}_{\bc{J}} \nabla v\,
	 = \int_{\bc{T}_j} \nabla \phi_i\cdot \tilde{k}_{\bc{J}} \nabla v\,
	\\
	&= K\int_{\bc{T}_j} \nabla \phi_j\cdot \nabla v\,
	 = K\dual{ {\bc{L}}_h\phi_j}{v},\quad j \in{\bc{J}}.
\end{align*}
Thus, $K$ is an eigenvalue of the operator $\bc{L}_h^{-1}\tilde{\bc{A}}_{\bc{J},h}$ associated with the eigenfunctions $\phi_j$, $j\in\bc{J}$, and therefore $K$ is the eigenvalue of the matrix $\alg{L}^{-1}\tilde{\alg{A}}_{\bc{J}}$ with the multiplicity at least $p$.
{This can also be verified {by} construction by observing that }
\[\tilde{\alg{A}}_{\bc{J}}\alg{e}_j = K\,\alg{L}\alg{e}_j,\quad j \in{\bc{J}}.\]

Consider now the eigenvalues of {$\alg{L}^{-1}\alg{E}_{\bc{J}}$; see} \cref{eq:error:matrix}.
The Rayleigh quotient for {an} eigenpair $(\theta,\alg{q})$, $\|\alg{q}\|=1$, and the associated eigenfunction $q = \sum_{j=1}^N\nu_j\phi_j$, where $\alg{q}^T = [\nu_1,\ldots,\nu_N]$, satisfies
\begin{align*}
\theta &= \frac{\alg{q}^T\alg{E}_{\bc{J}}\alg{q}}{\alg{q}^T\alg{L}\alg{q}}
		= \frac{\alg{q}^T(\alg{A} - \tilde{\alg{A}}_{\bc{J}})\alg{q}}{\alg{q}^T\alg{L}\alg{q}}
		= \frac{\dual{ ( \bc{A}_h - \tilde{\bc{A}}_{\bc{J},h})q}{q}}{\dual{ {\bc{L}}_hq}{q}}
		\\
	   &= \frac{\int_{\Omega} \nabla q \cdot (k( x) - \tilde{k}_{\bc{J}}( x))\nabla q\, d x }
	   			{\int_{\Omega} \|\nabla q\|^2\, d x}
		= \frac{\int_{\bc{T}_{\bc{J}}} (k( x) -K)\|\nabla q\|^2\, d x }
				{\int_{\Omega} \|\nabla q\|^2\, d x},
\end{align*}
giving
\begin{equation}
\label{BF1}
|\theta| \leq \max_{ x\in\bc{T}_{\bc{J}}}|k( x) - K|.
\end{equation}

{Next, consider the symmetric matrices} 
\[
\alg{A}_{\alg{L}} = \alg{L}^{-1/2}\alg{A}\alg{L}^{-1/2}, \quad
\alg{E}_{\alg{L}} = \alg{L}^{-1/2}\alg{E}_{\bc{J}}\alg{L}^{-1/2},\quad
\tilde{\alg{A}}_{\alg{L}} = \alg{L}^{-1/2}\tilde{\alg{A}}_{\bc{J}}\alg{L}^{-1/2}.
\]
{According to a standard result from the perturbation theory of matrices, see, e.g., \cite[Corollary~4.9, p.~203]{SSB90}, 
we find that }
\[
\lambda_s({\alg{A}}_{\alg{L}}) = \lambda_s(\tilde{\alg{A}}_{\alg{L}} + \alg{E}_{\alg{L}})
		\in [\lambda_s(\tilde{\alg{A}}_{\alg{L}}) + \theta_{min} ,
			 \lambda_s(\tilde{\alg{A}}_{\alg{L}}) + \theta_{max} ],\quad s=1,\ldots,N,
\]
where $\theta_{min}$ and $\theta_{max}$ are the smallest and largest
eigenvalues of $\alg{E}_{\alg{L}}$ respectively.
Since the matrices $\alg{L}^{-1}\alg{A}$, $\alg{L}^{-1}\alg{E}_{\bc{J}}$
and $\alg{L}^{-1}\tilde{\alg{A}}_{\bc{J}}$ have the same spectrum as the
matrices $\alg{A}_{\alg{L}}$, $\alg{E}_{\alg{L}}$ and
$\tilde{\alg{A}}_{\alg{L}}$, respectively, it follows that
\[
\lambda_s(\alg{L}^{-1}{\alg{A}}) = \lambda_s(\alg{L}^{-1}\tilde{\alg{A}}_{\bc{J}} + \alg{L}^{-1}\alg{E}_{\bc{J}})
		\in [\lambda_s(\alg{L}^{-1}\tilde{\alg{A}}_{\bc{J}}) + \theta_{min} ,
			 \lambda_s(\alg{L}^{-1}\tilde{\alg{A}}_{\bc{J}}) + \theta_{max} ].
\]	
Due to \cref{BF1}, 
\begin{align*}
\theta_{min} & \geq -\max_{ x\in\bc{T}_{\bc{J}}}|k( x) - K|, \\
\theta_{max} & \leq \max_{ x\in\bc{T}_{\bc{J}}}|k( x) - K|, 
\end{align*}
and thus, 
since $K$ is at least a $p$-multiple eigenvalue of $ \alg{L}^{-1}\tilde{\alg{A}}_{\bc{J}}$, there exist $p$ eigenvalues  $\tilde{\lambda}_1,\ldots,\tilde{\lambda}_p$ of $ \alg{L}^{-1}{\alg{A}}$ such that
\begin{equation}\label{eq:K:bound}
\tilde{\lambda}_{\ell} \in
	[\, K - \max_{ x\in\bc{T}_{\bc{J}}} |k( x) - K|,\,
		K + \max_{ x\in\bc{T}_{\bc{J}}} |k( x) - K|
	\,], \quad \ell=1,\ldots,p.
\end{equation}
Setting
\[
K = \tfrac{1}{2}(\min_{ x\in\bc{T}_{\bc{J}}} k(x) +
				   \max_{ x\in\bc{T}_{\bc{J}}} k(x))
\]
gives
\[
\tilde{\lambda}_{\ell} \in
	[\min_{ x\in\bc{T}_{\bc{J}}} k(x),
	 \max_{ x\in\bc{T}_{\bc{J}}} k(x)],\quad \ell=1,\ldots,p.
\]
\end{proof}

{Applying \cref{th:lemma2} $N$ times} with $\bc{J} = \{1\}$, $\bc{J} = \{2\},\ldots,\,\bc{J} = \{N\}$, we see that, for the support of any basis function $\phi_j$ there is an eigenvalue $\tilde{\lambda}$ of $\alg{L}^{-1}\alg{A}$ such that $\tilde{\lambda} \in k({\bc{T}_j})$.
Moreover, {as an additional important consequence,} for any subset $\bc{J}\subset\{1,\ldots,N\}$ the associated union of intervals $\bar{k}(\bc{T}_{\bc{J}})$ (see \cref{eq:interval}) contains at least  $p = |\bc{J}| $ eigen\-values of  $\alg{L}^{-1}\alg{A}$; see the following corollary.
%
%
\begin{corollary}\label{th:cor}
Let, {as above,} $\bc{J}\subset\{1,\ldots,N\}$ and $\bc{T}_{\bc{J}} = \cup_{j\in\bc{J}}\bc{T}_j$.
Then there exist at least $p = |\bc{J}| $ eigenvalues $\tilde{\lambda}_1,\ldots,\tilde{\lambda}_p$ of 
$\alg{L}^{-1}\alg{A}$ such that
\begin{equation}\label{eq:corr}
\tilde{\lambda}_{\ell} \in \bar{k}(\bc{T}_{\bc{J}})\equiv \bigcup_{j\in\bc{J}}k(\bc{T}_j),
\quad \ell=1,\ldots,p.
\end{equation}
{Moreover, taking} $\bc{J} = \{1,\ldots,N\}$, \cref{eq:corr} immediately implies that any eigenvalue  $\tilde{\lambda}$ of  $\alg{L}^{-1}\alg{A}$ belongs to (at least one) interval $k(\bc{T}_j)$, $j \in \{1,\ldots,N\}$.
\end{corollary}
\begin{proof}
Since  $\bar{k}(\bc{T}_{j}) = k(\bc{T}_j)$, for any $j\in\bc{J}$, is an interval \cref{eq:range}, the set $\bar{k}(\bc{T}_{\bc{J}})$ consists of at most $p$ intervals.
We decompose $\bar{k}(\bc{T}_{\bc{J}})$ into $\tilde{p}$ mutually disjoint intervals, $\tilde{p}\leq p$, 
\[
\bar{k}(\bc{T}_{\bc{J}_{i}}) \equiv  \bigcup_{j\in\bc{J}_i}k(\bc{T}_j), \quad i = 1,\ldots,\tilde{p}.
\]
\Cref{th:lemma2} then assures that each interval $\bar{k}(\bc{T}_{\bc{J}_{i}})$ contains at least $|\bc{J}_{i}|$ eigenvalues of $\alg{L}^{-1}\alg{A}$.
Summing up, at least $\sum_{i=1,\ldots,\tilde{p}}|\bc{J}_i| = |\bc{J}|$ eigenvalues of $\alg{L}^{-1}\alg{A}$ must be contained in the union $\bar{k}(\bc{T}_{\bc{J}})$.
\end{proof}

{In order to finalize the proof of \cref{th:theorem}, we still need to show the existence of a one-to-one pairing between the individual eigenvalues and the individual intervals $k(\bc{T}_{j}),\ j = 1,\ldots,N$.}
The relationship between the intervals $k(\bc{T}_{j})$, $j = 1,\ldots,N$, and the eigenvalues of $\alg{L}^{-1}\alg{A}$ {described in}  \cref{th:lemma2} and \cref{th:cor} can be represented {by} the following bipartite graph.
Let, as above,
$0< \lambda_1\leq\lambda_2\leq\ldots\leq\lambda_N$ be the eigenvalues of $\alg{L}^{-1}\alg{A}$.
Consider the bipartite graph
\begin{equation}\label{eq:graph}
(\bc{S},\bc{I},E)
\end{equation}
with the sets of nodes $\bc{S} = \bc{I} = \{1,\ldots,N\}$ and the set of edges $E$, where 
\[
\{s,i\}\in E\quad \mbox{if and only if}\quad \lambda_s\in k(\bc{T}_i),\quad s\in\bc{S},\ i\in \bc{I}.
\]
A subset of edges $M\subset E$ is called matching if no edges from $M$ share a common node; see \cite[Section 5.1]{Bondy1976}.
We will use the following famous theorem.

%
\begin{theorem}[Hall's theorem]\label{th:hall}
Let $(\bc{S},\bc{I},E)$ be a bipartite graph. 
{Given $\bc{J}\subset\bc{I}$, let} $G(\bc{J})\subset\bc{S}$ denote the set of all nodes adjacent to any node from $\bc{J}$, i.e.,
\[G(\bc{J}) = \{s\in\bc{S};\ \exists i\in\bc{J} \mbox{ such that } \{s,i\}\in E \}. \]
Then there exists a matching $M\subset E$ that covers $\bc{I}$ if and only if
\begin{equation}
|G(\bc{J})|\geq |\bc{J}|\quad\mbox{for any}\quad \bc{J}\subset\bc{I};
\label{eq:cond}
\end{equation}
see, e.g., \cite[Theorem~5.2]{Bondy1976} and the original  formulation \cite[Theorem~1]{Hall35}.
\end{theorem}
\noindent Now we are ready to finalize our argument.
\subsection*{Proof of \cref{th:theorem}}
Consider  the bipartite graph {defined by} \cref{eq:graph} and let $G(\bc{J})\subset\bc{S}$ be the set of all nodes {(representing the eigenvalues)} adjacent to any node from $\bc{J}$, $\bc{J}\subset\bc{I}$ {(representing the intervals)}.
In other words, $G(\bc{J})$ represents the indices  of all eigenvalues $\{\lambda_s;\ s\in G(\bc{J})\}$ located in $\bar{k}(\bc{T}_{\bc{J}}) = \cup_{j\in\bc{J}}k(\bc{T}_{j})$.
\Cref{th:cor} of \cref{th:lemma2} {assures that  assumption \cref{eq:cond} in \cref{th:hall} is satisfied, i.e. }
\begin{equation}\label{eq:graph:cond}
|G(\bc{J})|\geq |\bc{J}|.
\end{equation}
{Thus, according to} \cref{th:hall}, there exist a matching $M\subset E$ that covers $\bc{I}$.
Since $|\bc{I}|= |\bc{S}|$, this matching defines the permutation $\pi(i),\ i = 1,2,\ldots,N$, such that
\[
\lambda_{\pi(i)} \in k(\bc{T}_i),\quad i=1,\ldots,N,
\]
which finishes the proof.
\qed


\section{Numerical experiments}
\label{experiments}
In this section we will illustrate the theoretical results by a series of numerical experiments. 
We will investigate how well the nodal values of $k$ correspond to the eigenvalues and assess the sharpness of the estimates in \cref{th:taylor} in a few examples, 
including both uniform and local mesh refinement. Furthermore, we will compute the corresponding intervals  $k(\bc{T}_j),\ j = 1,\ldots,N$ and consider the pairing in \cref{th:theorem}.    

\subsubsection*{Test problems}
We will consider four test problems defined  on the domain $\Omega \equiv (0,1)\times(0,1)$ where we slightly abuse the notation above
and let $k=k(x,y)$. The first three problems use a continuous coefficient function $k(x,y)$:
\begin{eqnarray*}
	\mbox{(P1)}\quad & k(x,y) &= \sin(x+y), \\
	\mbox{(P2)}\quad & k(x,y) &=  1+50\exp(-5(x^2+y^2)), \\
	\mbox{(P3)}\quad & k(x,y) &= 2^7(x^7+y^7).
\end{eqnarray*}
The fourth problem uses a discontinuous function $k(x,y)$,
\begin{align*}
\mbox{(P4)}\quad  k(x,y)\ =& 
\begin{cases}
	\mbox{(P1)}\quad \mbox{for}\ (x,y)\in (0,1)\times (\tfrac{1}{2},1),\\
	\mbox{(P2)}\quad \mbox{elsewhere}.
\end{cases}
\end{align*}
Numerical experiments were computed  using FEniCS~\cite{2012Fenics}  and Matlab.\footnote{FEniCS version 2017.2.0 and MATLAB Version: 8.0.0.783 (R2012b).} If not specified otherwise, we consider a triangular uniform mesh with piecewise linear discretization basis functions.

\subsection{Illustration of \cref{th:theorem} and \cref{th:taylor}}
\label{sec:illustration}
In \cref{fig:figure2} we show the nodal values of $k(x,y)$ and the corresponding eigenvalues, both sorted
in increasing order on the unit square with $N=81$ degrees of freedom.  
Clearly, there is a close correspondence between the nodal values and the eigenvalues even at this relatively coarse
resolution, but there are some notable differences for (P3) and (P4) that are clearly visible.

\begin{figure}[!ht]
\centering
\includegraphics[width=.495\textwidth]{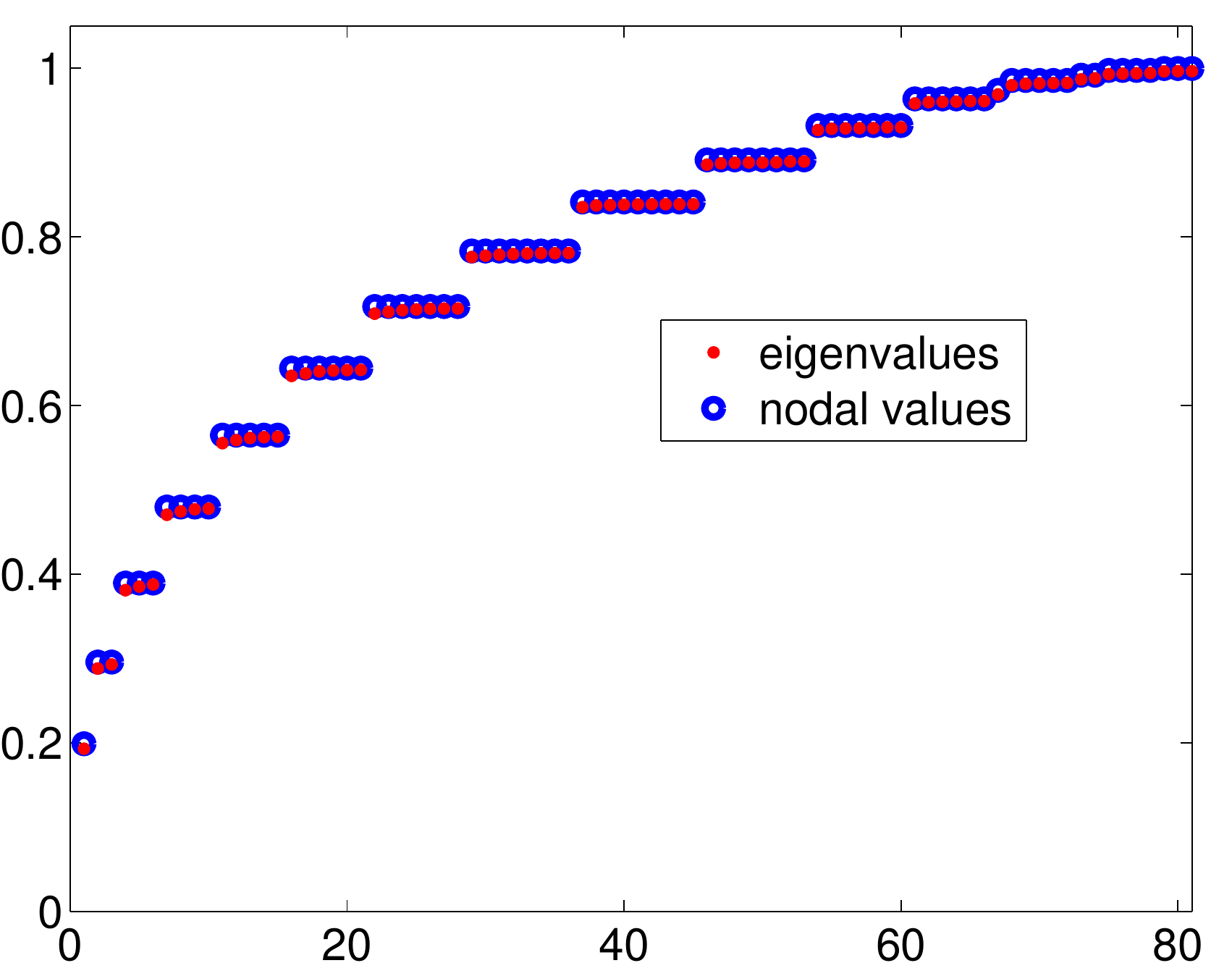}
\includegraphics[width=.495\textwidth]{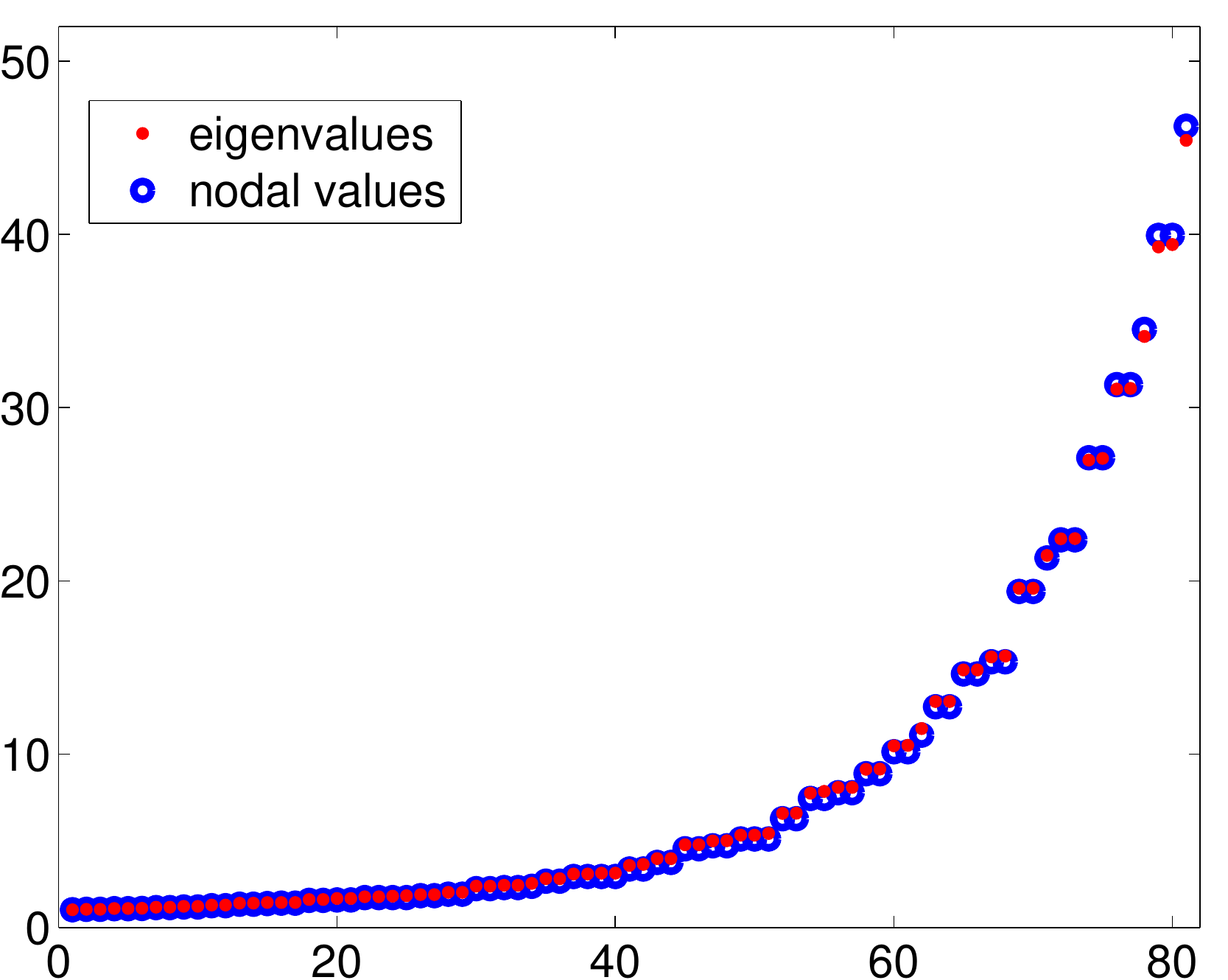}

\includegraphics[width=.495\textwidth]{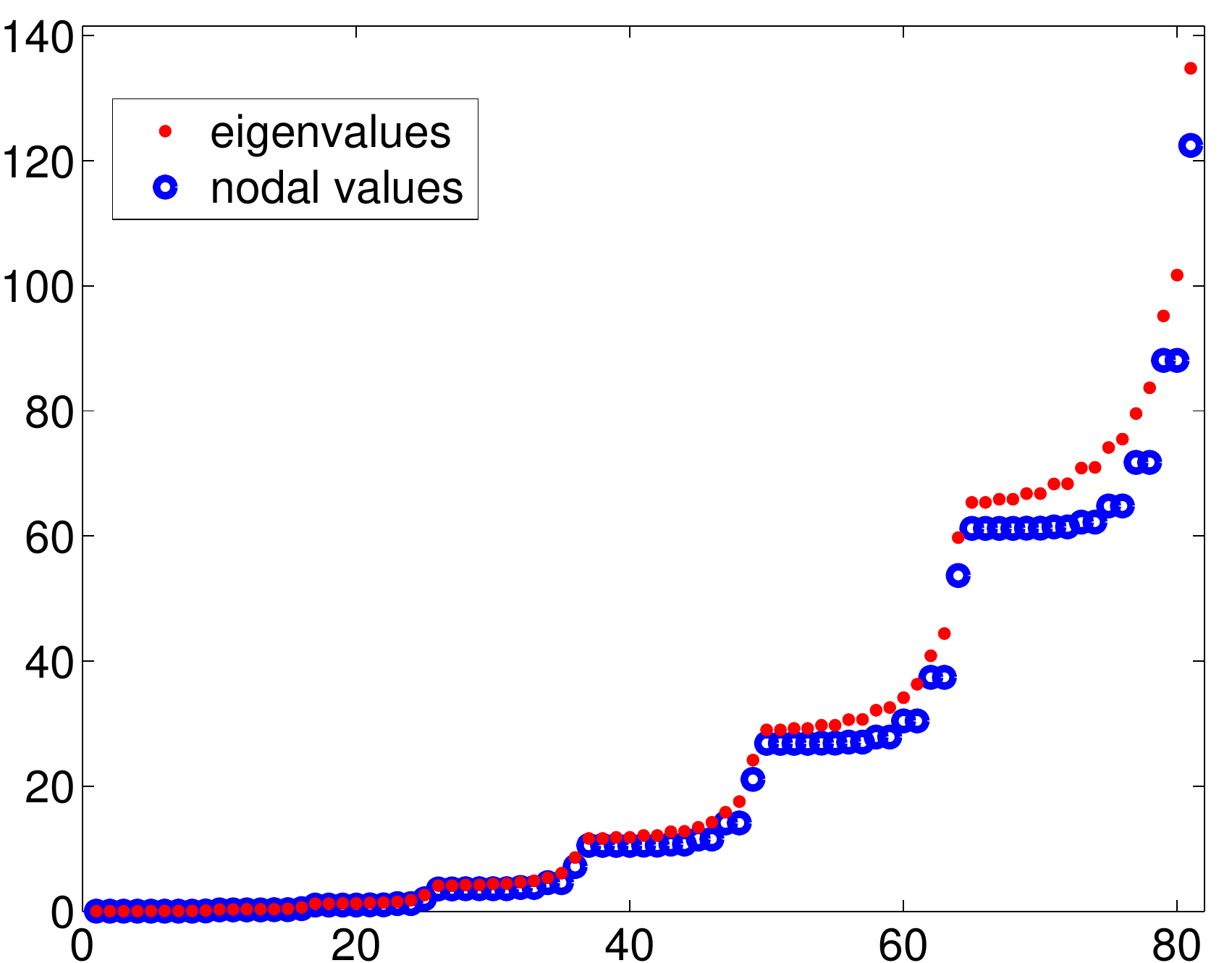}
\includegraphics[width=.495\textwidth]{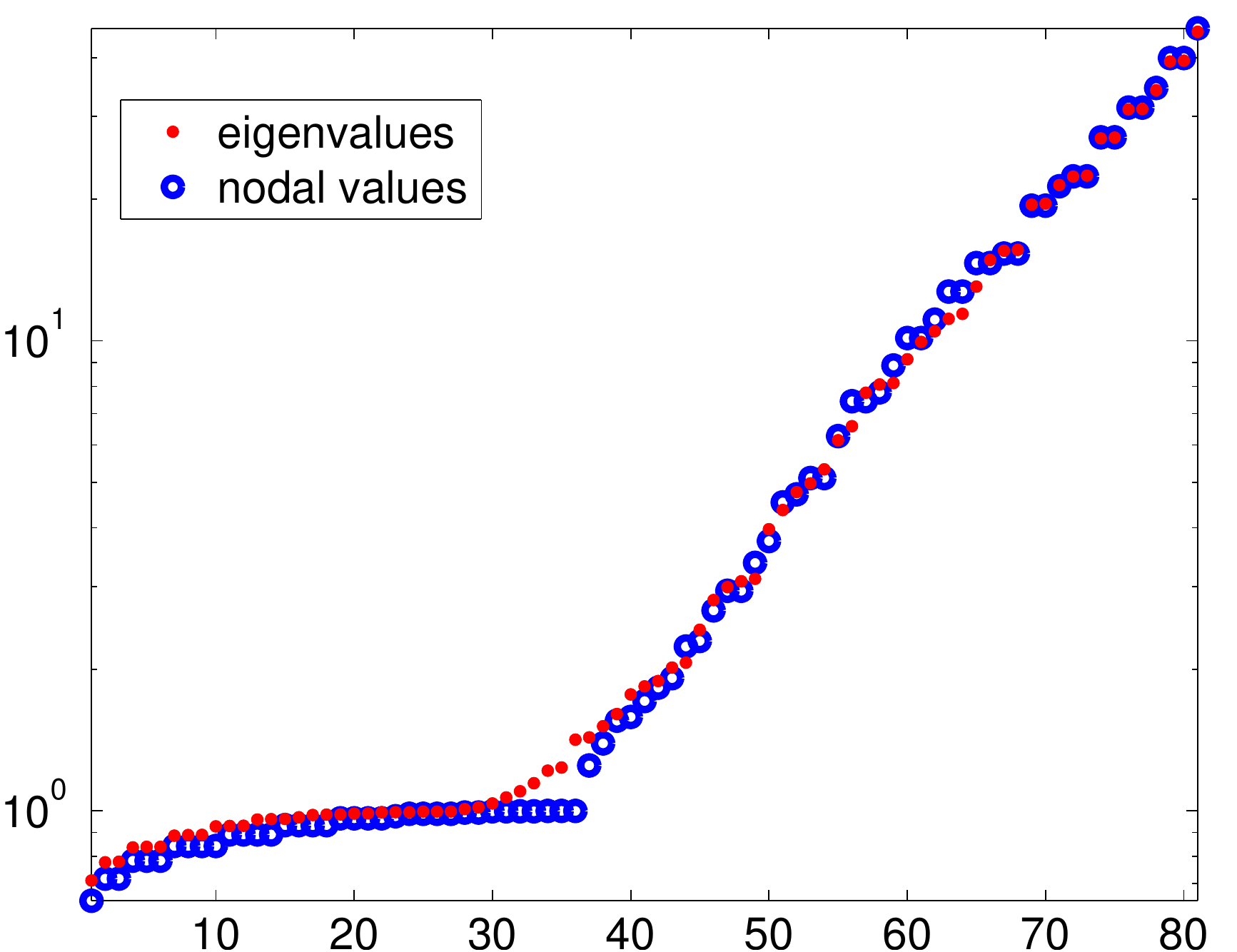}
	\caption{Comparison of the eigenvalues $\lambda_s$, $s = 1,\ldots,N$ (red dots) and the increasingly sorted nodal values of $k$ (blue circles). Top left: (P1), top right: (P2), bottom left: (P3), bottom right:  (P4). As in \cref{fig:figure}, we use the semilogarithmic scale in the lower right panel (P4).}
\label{fig:figure2}
\end{figure}

\Cref{th:theorem} states that there exists a pairing $\pi$ such that $\lambda_{\pi(i)}\in k(\bc{T}_{i})$ for every $i = 1,\ldots,N$. The proof is not constructive 
and it is therefore interesting to consider potential pairings. 
In \cref{fig:figure} we show the results of the previously mentioned paring of the eigenvalues and the intervals $k(\bc{T}_i)~=~k(\bc{T(x_i,y_i)})$ where the vertices $(x_i,y_i)$ have been sorted such that the nodal values $k(x_i, y_i)$ are in increasing order. The pairing appears to work quite well except for the case (P4) where in particular the eigenvalues between 30-40  
are outside the intervals provided by this pairing. 
\begin{figure}[h!t]
\centering
\includegraphics[width=.495\textwidth]{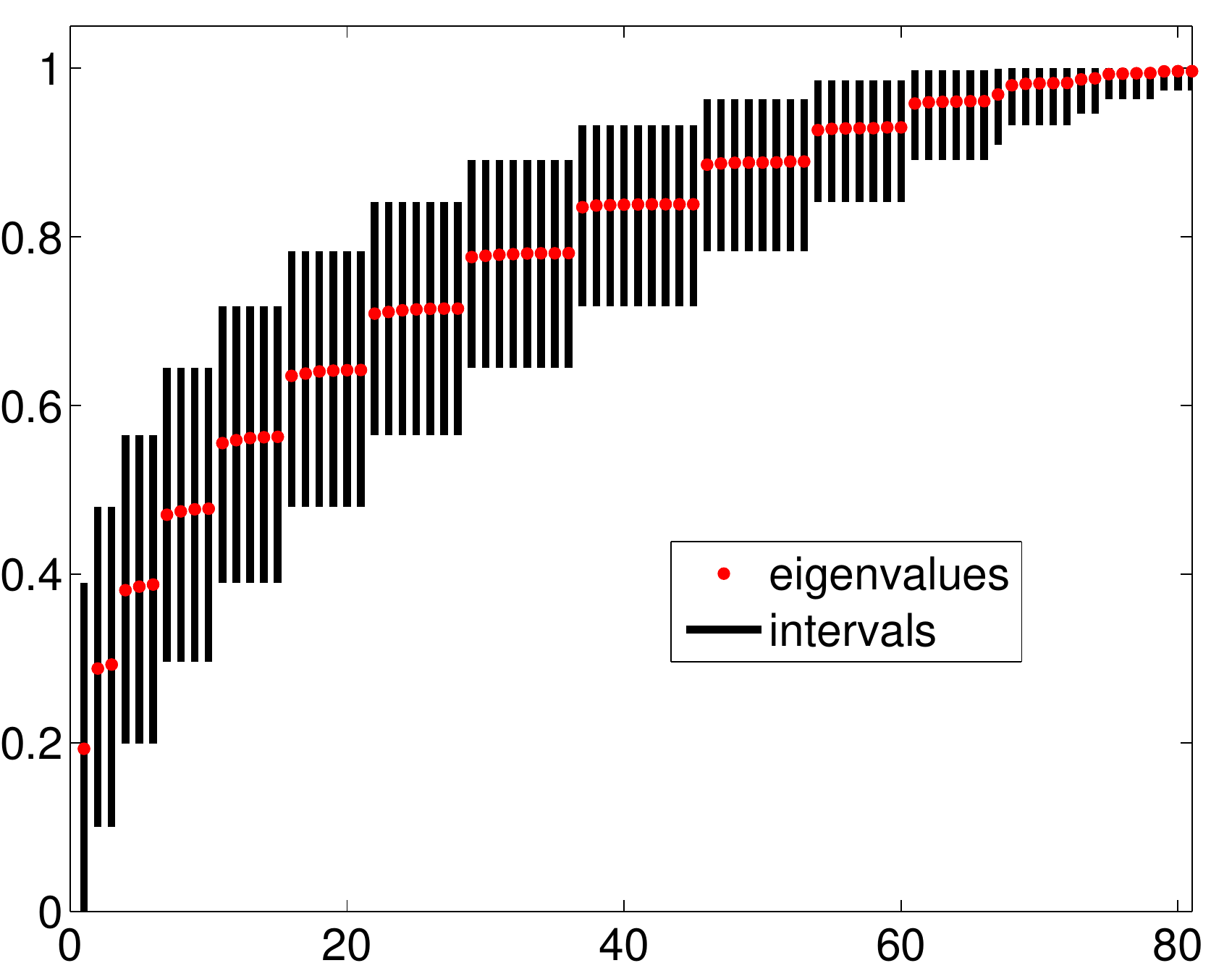}
\includegraphics[width=.495\textwidth]{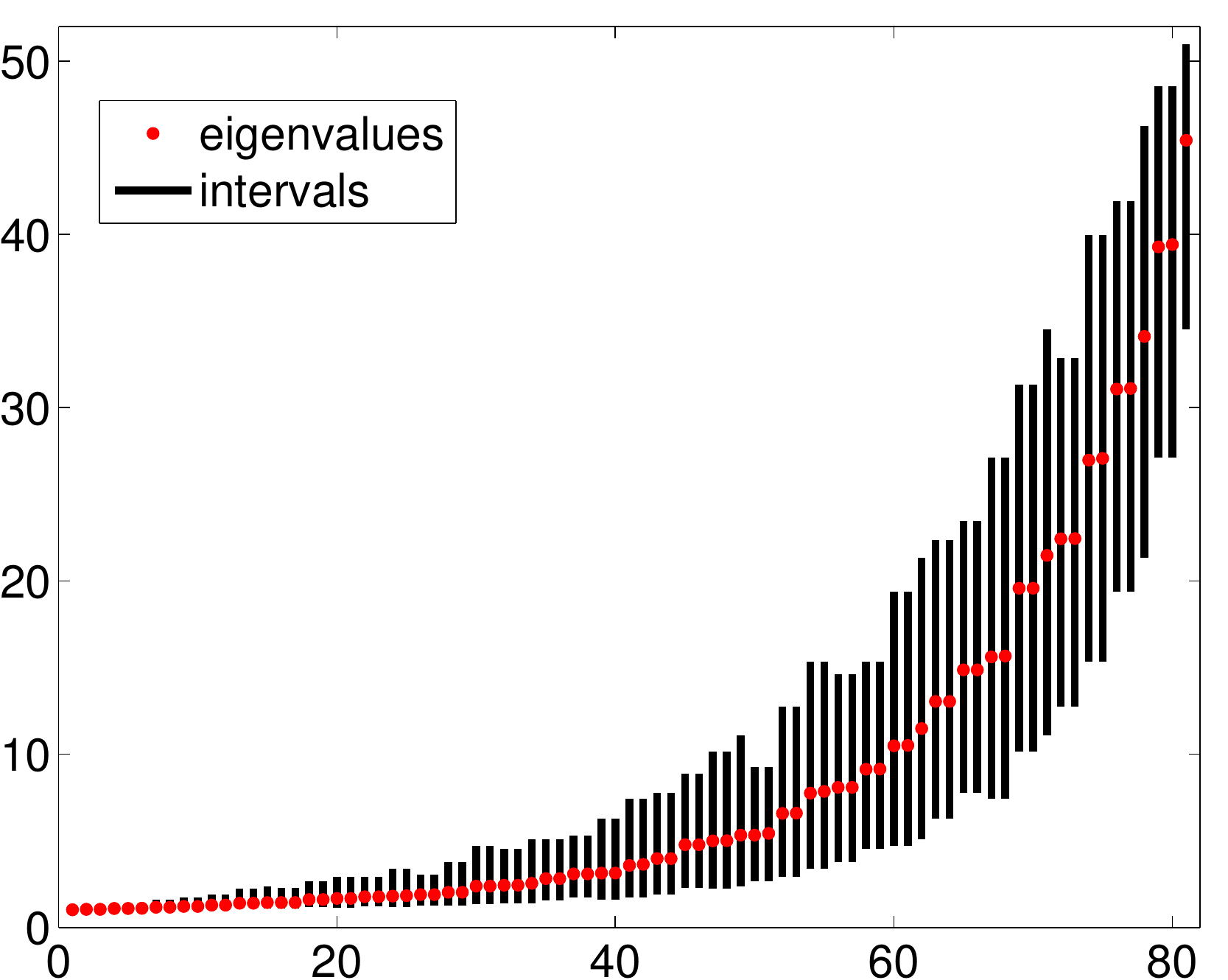}

\includegraphics[width=.495\textwidth]{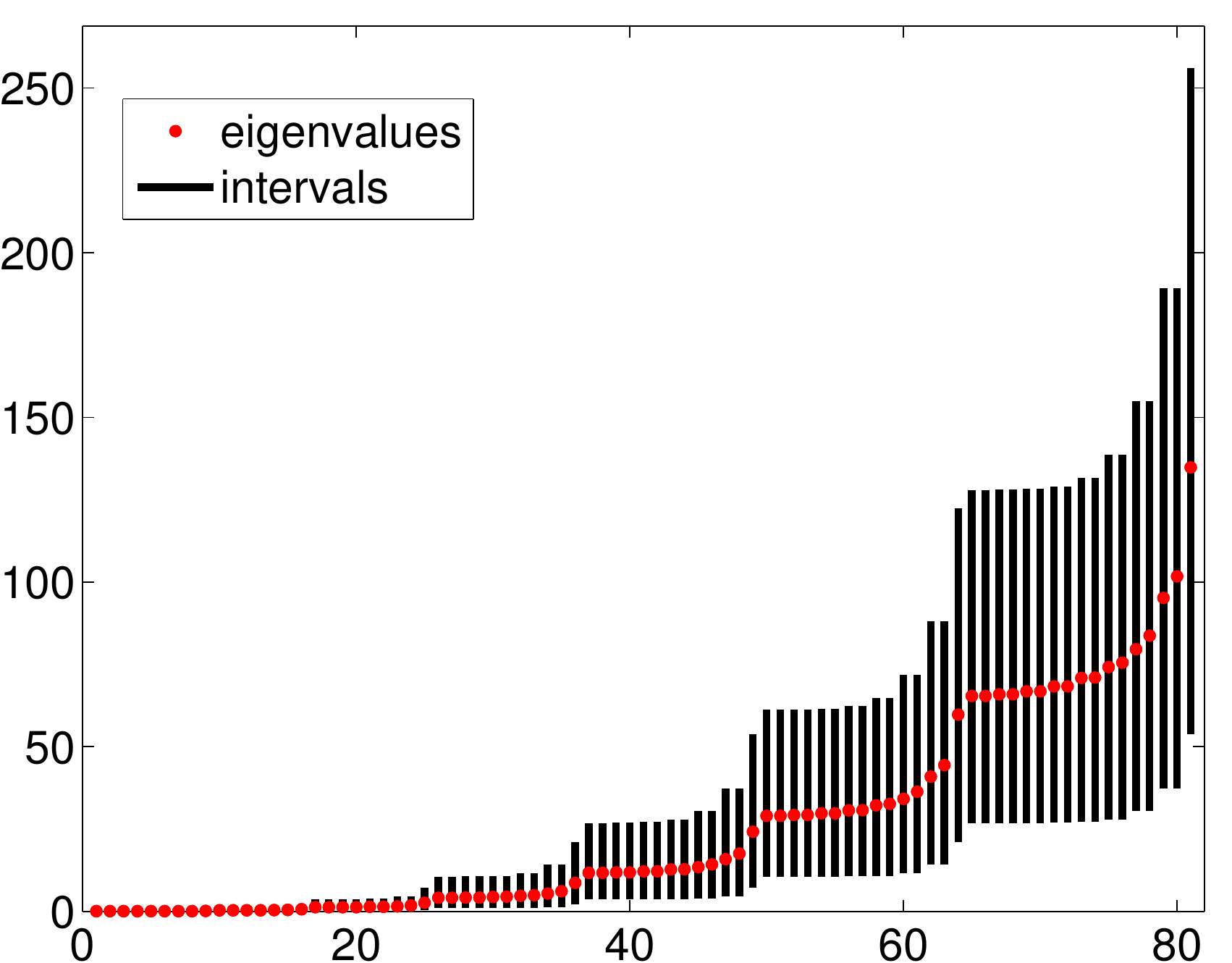}
\includegraphics[width=.495\textwidth]{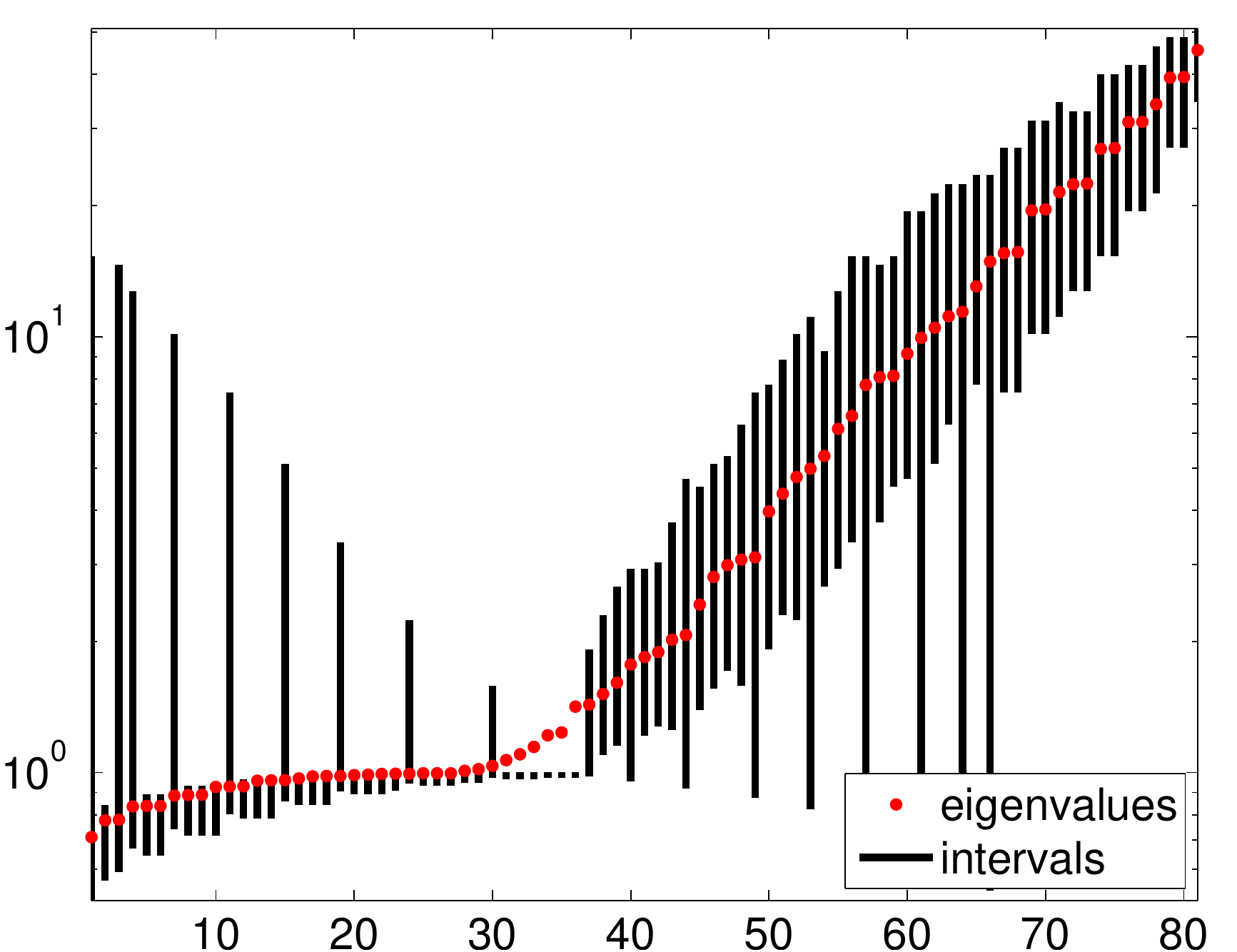}
	\caption{The eigenvalues $\lambda_1\leq\ldots\leq\lambda_N$ (red dots) and the associated intervals $k(\bc{T}_{P(s)})$ (black vertical lines), where the pairing $P$ is defined by the increasingly sorted nodal values of $k$; see \cref{fig:figure2}. Top left: (P1), top right: (P2), bottom left: (P3), bottom right:  (P4). We observe that for (P4) some of the eigenvalues are not inside the associated intervals and therefore the ordering in which eigenvalues and nodal values of $k$ are in increasing order does not in this case conform to $\pi$ from \cref{th:theorem}.
}
\label{fig:figure}
\end{figure}

\begin{figure}[h!t]
\centering
\includegraphics[width=.495\textwidth]{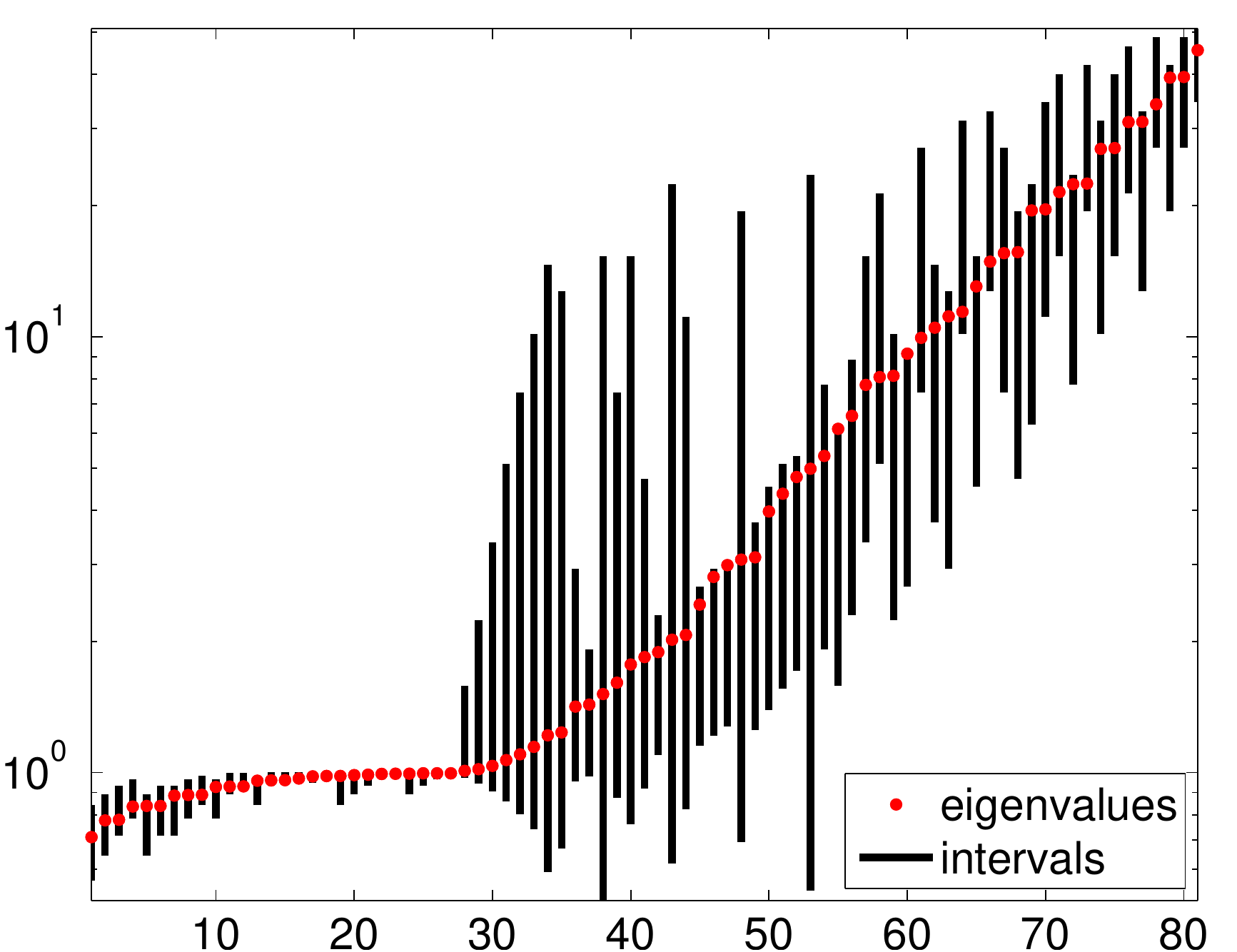}
\includegraphics[width=.495\textwidth]{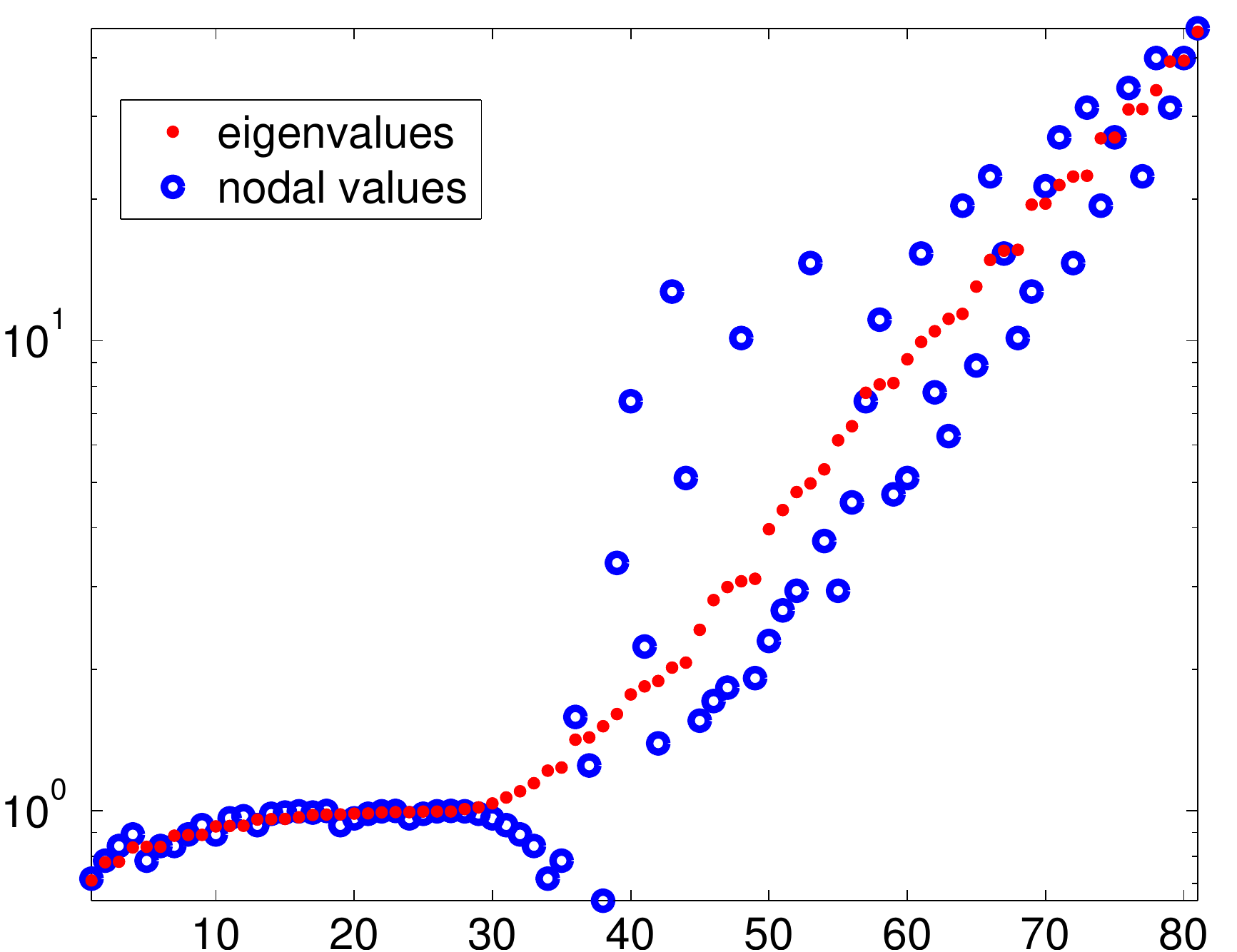}
\caption{Illustration of the pairing $\pi$ computed by the Dulmage-Mendelsohn decomposition of the corresponding adjacency matrix $\alg{G}$ (see \cref{eq:G}) for problem (P4). Left: The eigenvalues $\lambda_1\leq\ldots\leq\lambda_N$ (red dots) and the associated intervals $k(\bc{T}_{\pi^{-1}(s)})$ (black vertical lines). Right: The comparison of the eigenvalues and the associated nodal values $k_{\pi^{-1}(s)}$ (blue circles).}
\label{fig:disc}
\end{figure}
In order to ensure that we employ a proper pairing, i.e., to guarantee that $\lambda_{\pi(i)} \in k(\bc{T}_i)$, $i = 1,\ldots,N$, we construct the adjacency matrix $\alg{G}$ such that
\begin{equation}
\alg{G}_{si}  =
	\begin{cases}
		1,\quad \lambda_s \in k(\bc{T}_i), \\
		0,\quad \lambda_s \notin k(\bc{T}_i).
	\end{cases}
\label{eq:G}			
\end{equation}
By using the Dulmage-Mendel\-sohn decomposition\footnote{See, e.g., the original paper \cite{DulMen58}.} of this adjacency matrix $\alg{G}$ (provided by the Matlab command \texttt{dmperm}) we get a pairing $\pi$ satisfying $\lambda_{\pi(i)}\in k(\bc{T}_{i})$ for every $i = 1,\ldots,N$.
\Cref{fig:disc} illustrates the pairing $\pi$ from \cref{th:theorem} for (P4)  and the approximation of the eigenvalues by the associated nodal values (the plots in \cref{fig:disc}  should be compared with the lower right panels of \cref{fig:figure2,fig:figure}).

The difference between the nodal values and the corresponding eigenvalues is estimated in \cref{eq:h_taylor} and to assess the
sharpness of this estimate,  
\cref{fig:figure3} compares the quantities $|\lambda_s - k_{\pi^{-1}(s)}|$ (red dots) with the first term on the right hand side of \cref{eq:h_taylor} (black stars). We observe that  the first term of \cref{eq:h_taylor} in general overestimate the differences at this coarse resolution.
\begin{figure}[!ht]
\centering
\includegraphics[width=.495\textwidth]{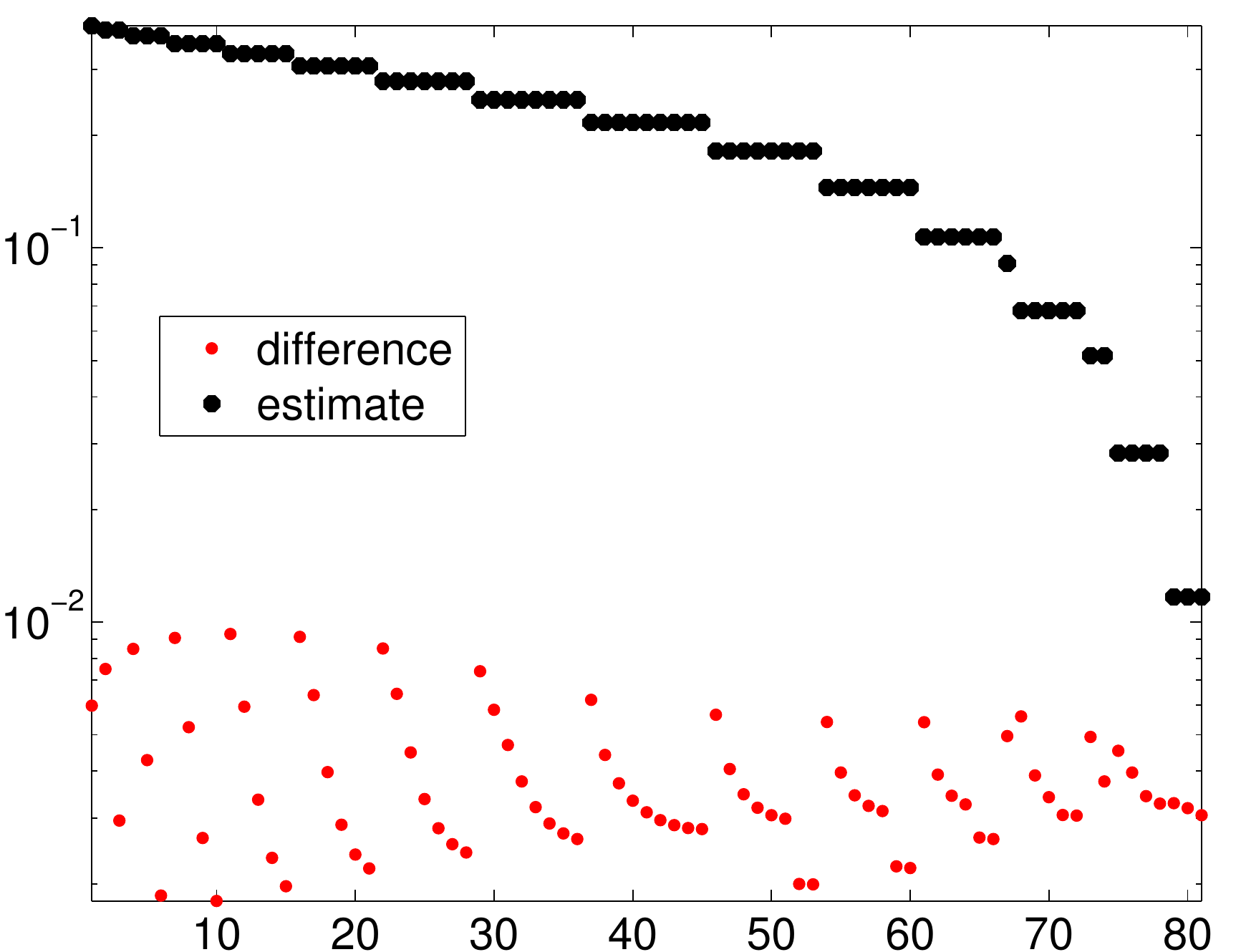}
\includegraphics[width=.495\textwidth]{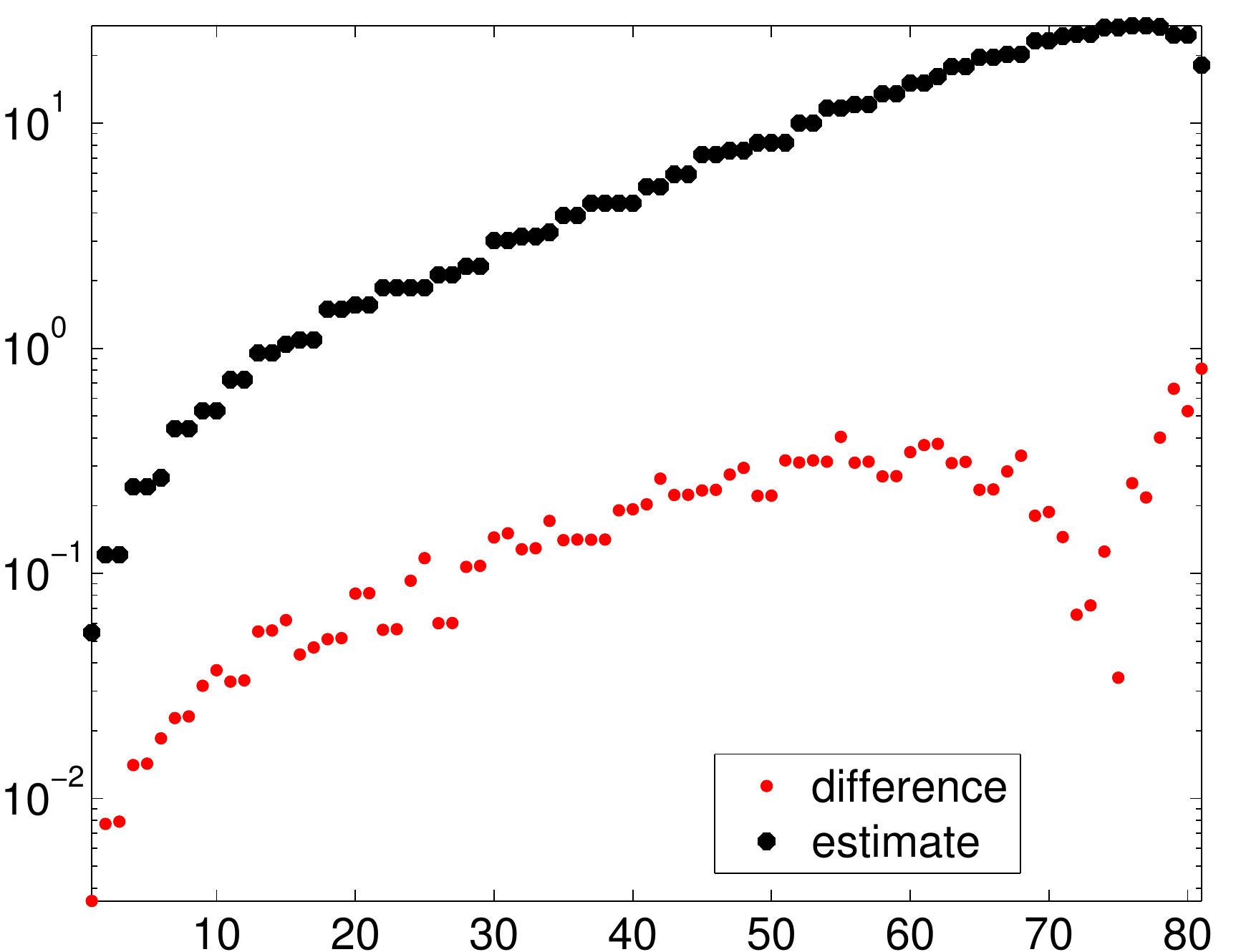}

\includegraphics[width=.495\textwidth]{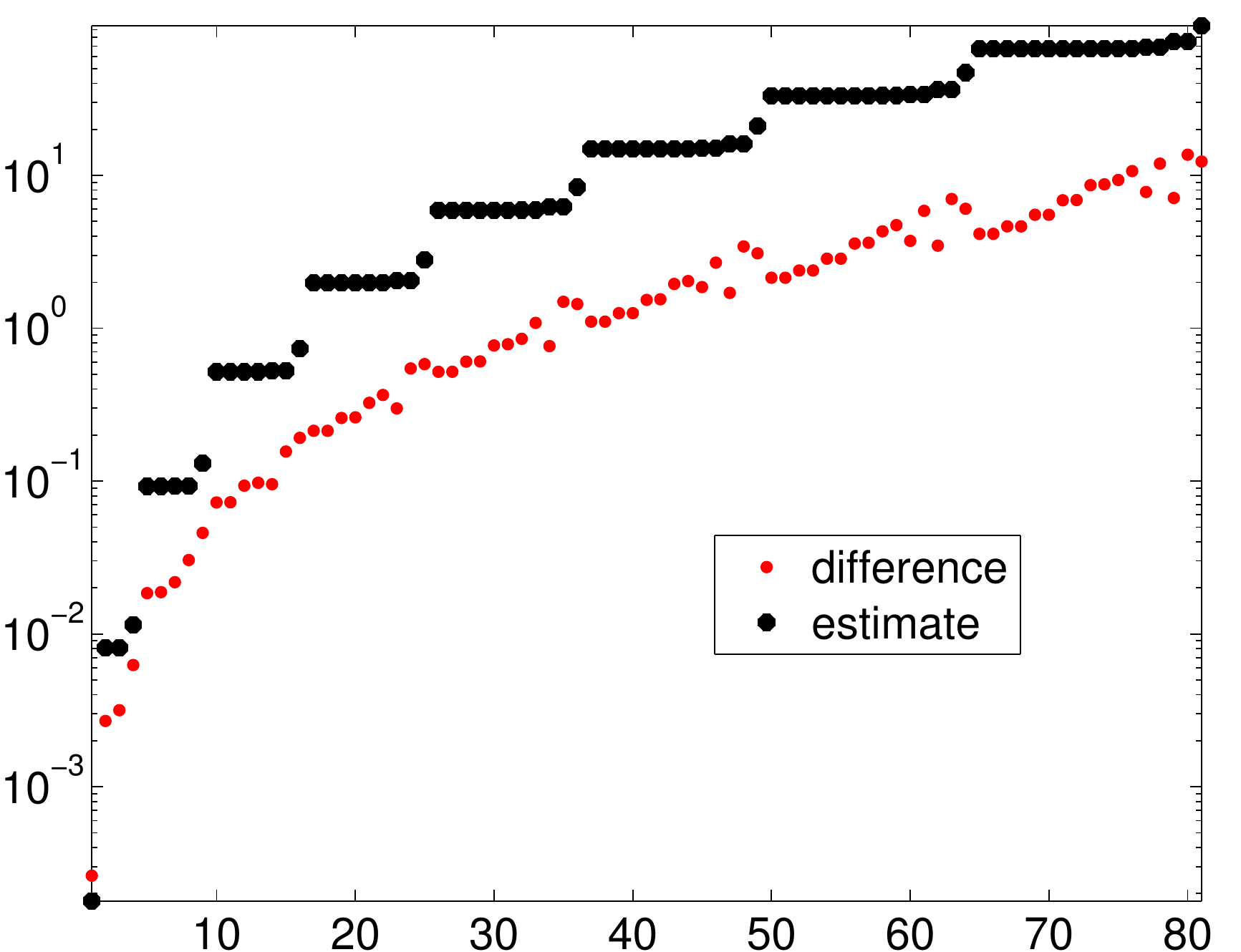}
\includegraphics[width=.495\textwidth]{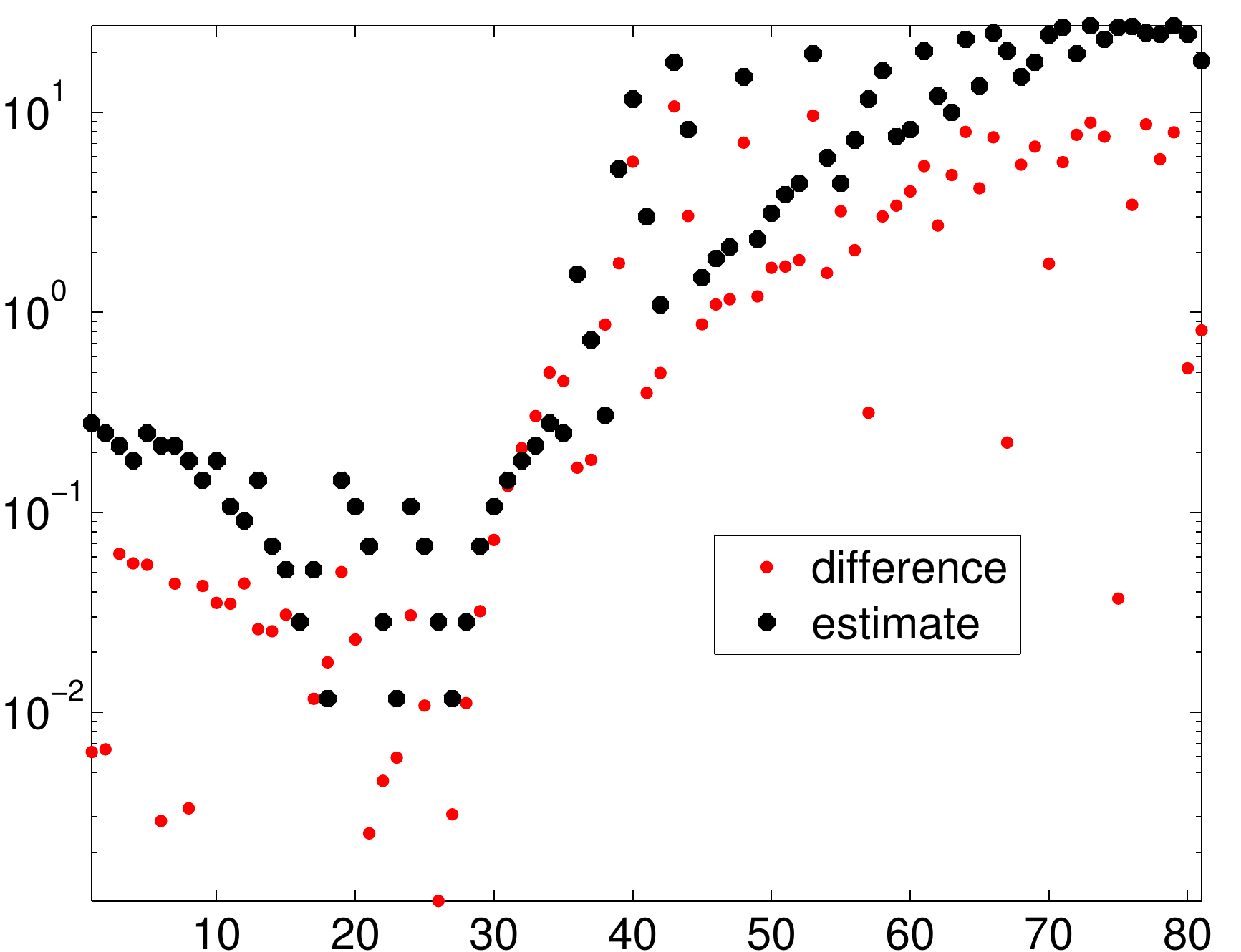}
\caption{Illustration of \cref{th:taylor}. Comparison of the absolute difference  $|\lambda_s - k_{\pi^{-1}(s)}|$ (red dots) and its estimate by the first term on the right hand side of \cref{eq:h_taylor} (black stars).  Top left: (P1), top right: (P2), bottom left: (P3), bottom right:  (P4).}
\label{fig:figure3}
\end{figure}

\subsection{Effects of $h$-adaptivity}
\label{sec:adaptivity}
{\Cref{th:taylor} states that the estimated difference $|\lambda_s~-~k_{\pi^{-1}(s)}|$ improves at least linearly as the mesh is refined.  \Cref{fig:fine:int} shows the improvement of both the nodal value estimates of $k$ and the associated intervals $k(\bc{T}_{\pi^{-1}(s)})$  for problems (P1) and (P3) with $N=59^2=3481$ degrees of freedom.}
(We would also like to note that the proof of \cref{th:taylor} does not assume linear Lagrange elements, but holds for any type of nodal basis functions.)
\begin{figure}[h!t]
\centering
\includegraphics[width=.495\textwidth]{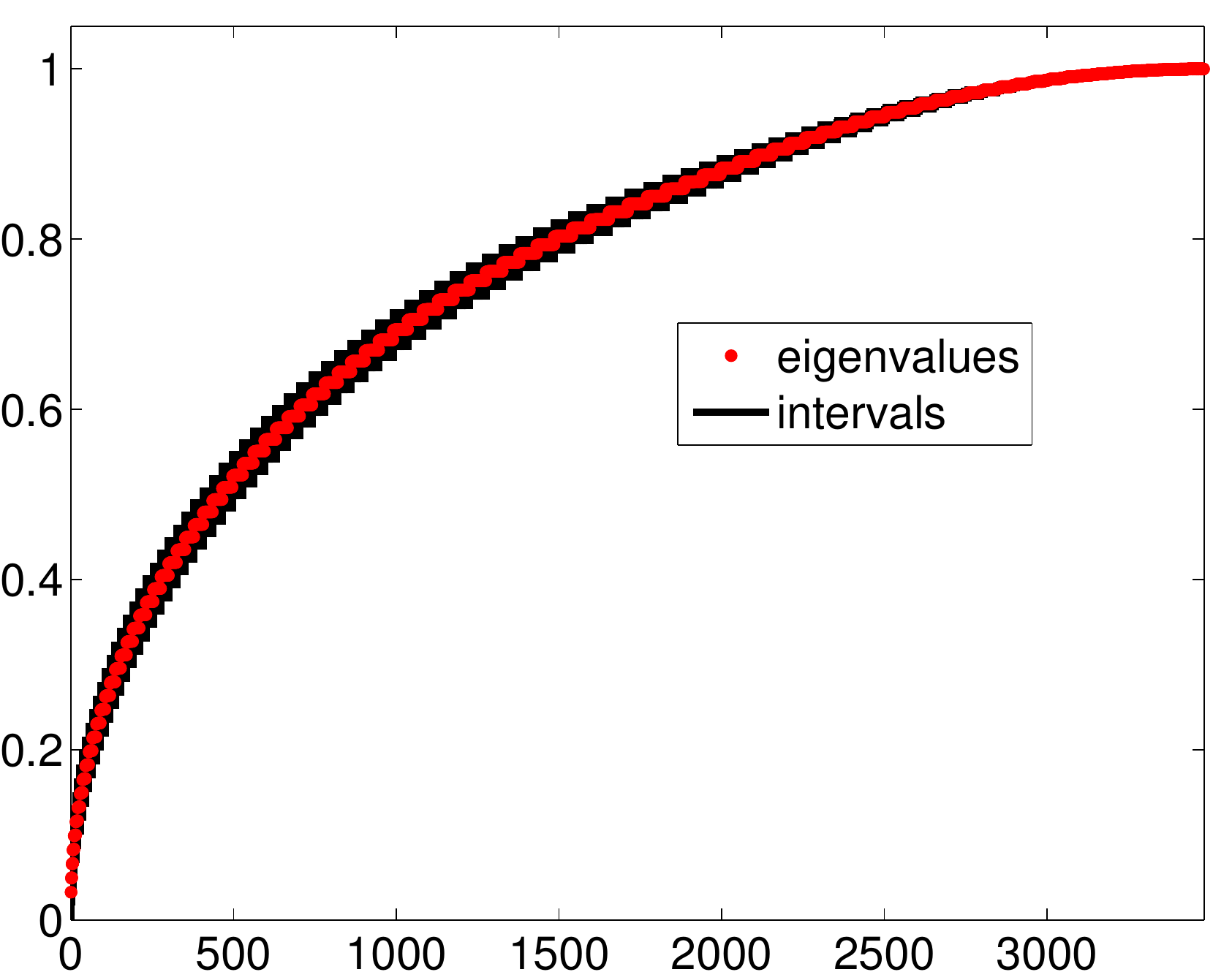}
\includegraphics[width=.495\textwidth]{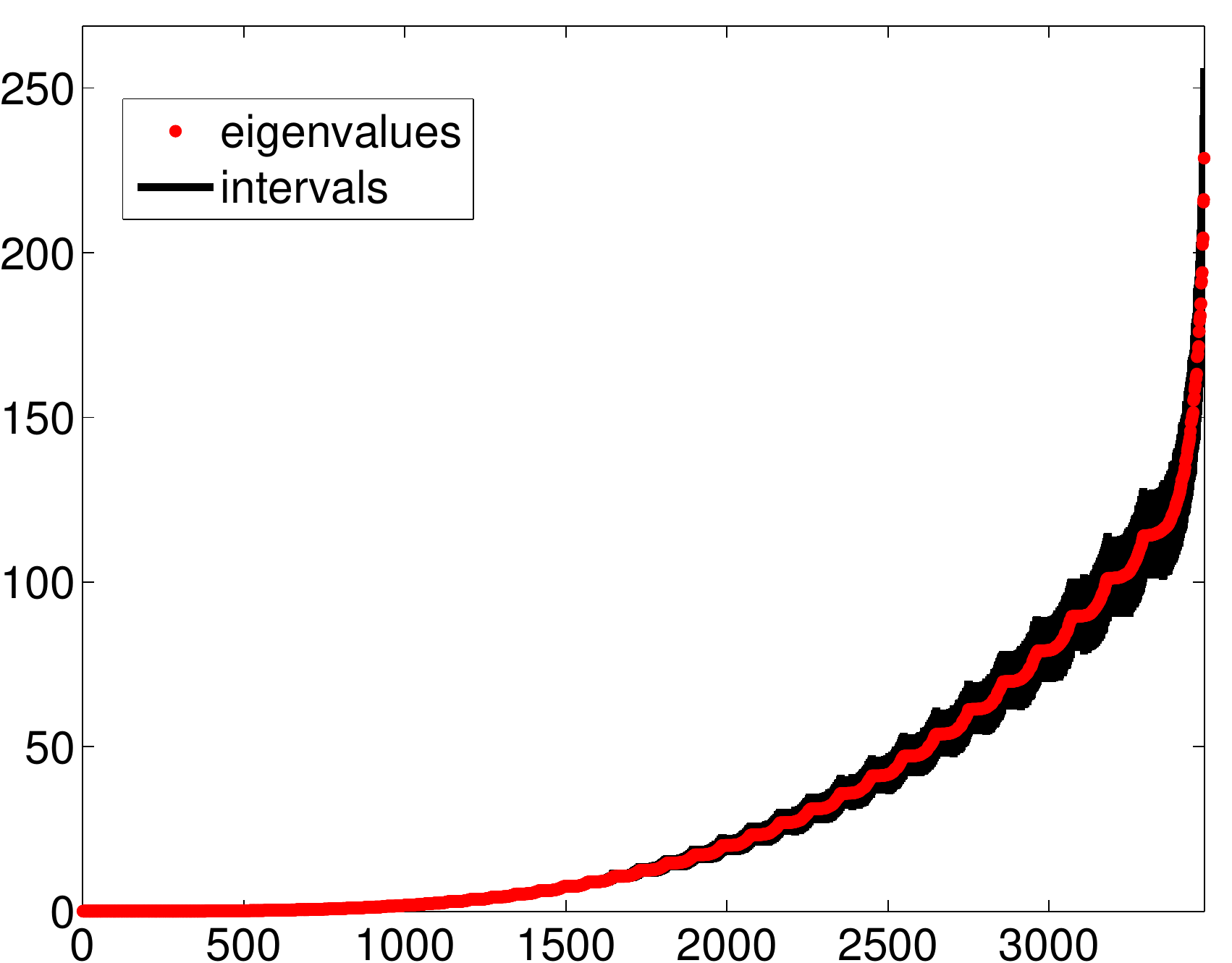}

\includegraphics[width=.495\textwidth]{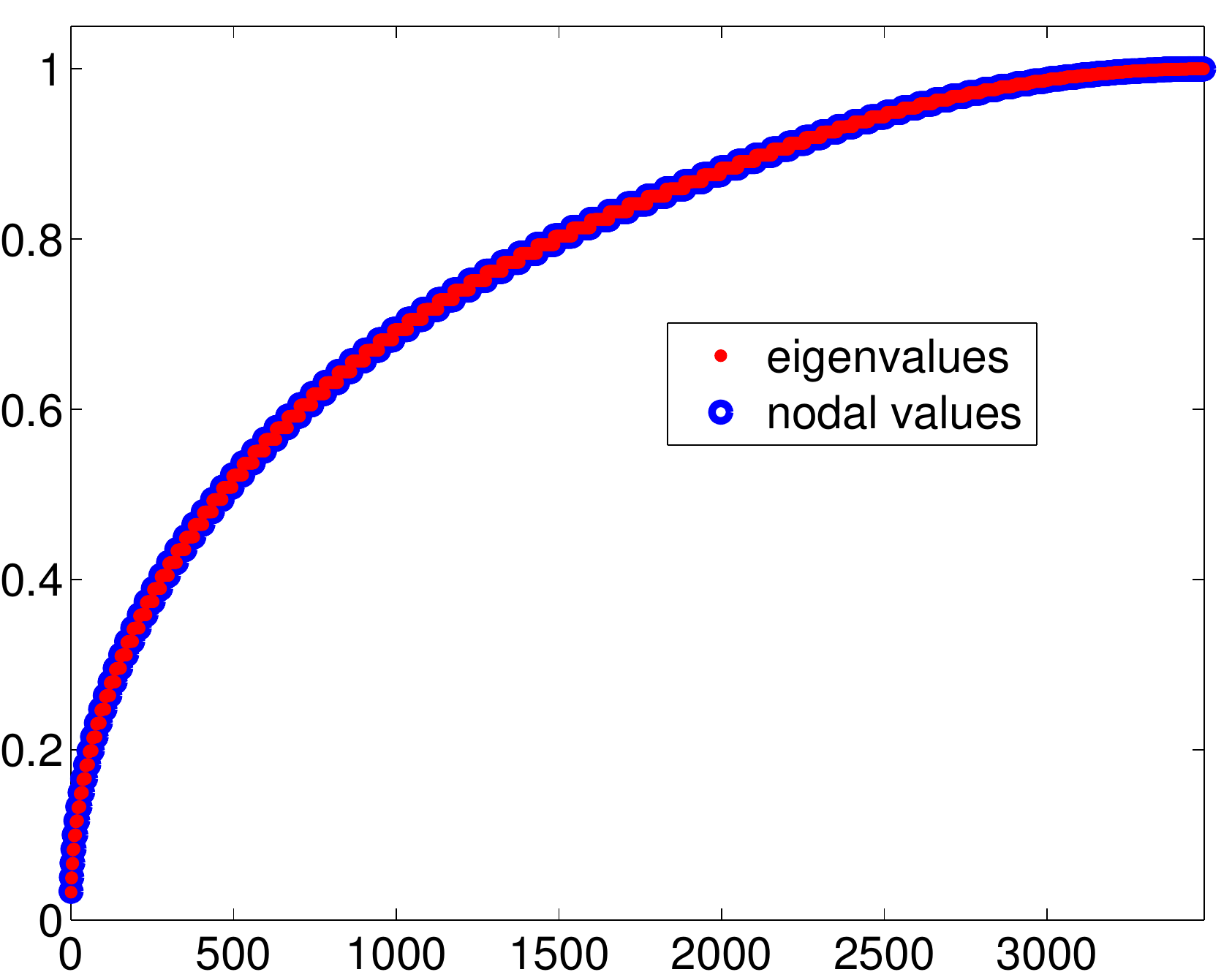}
\includegraphics[width=.495\textwidth]{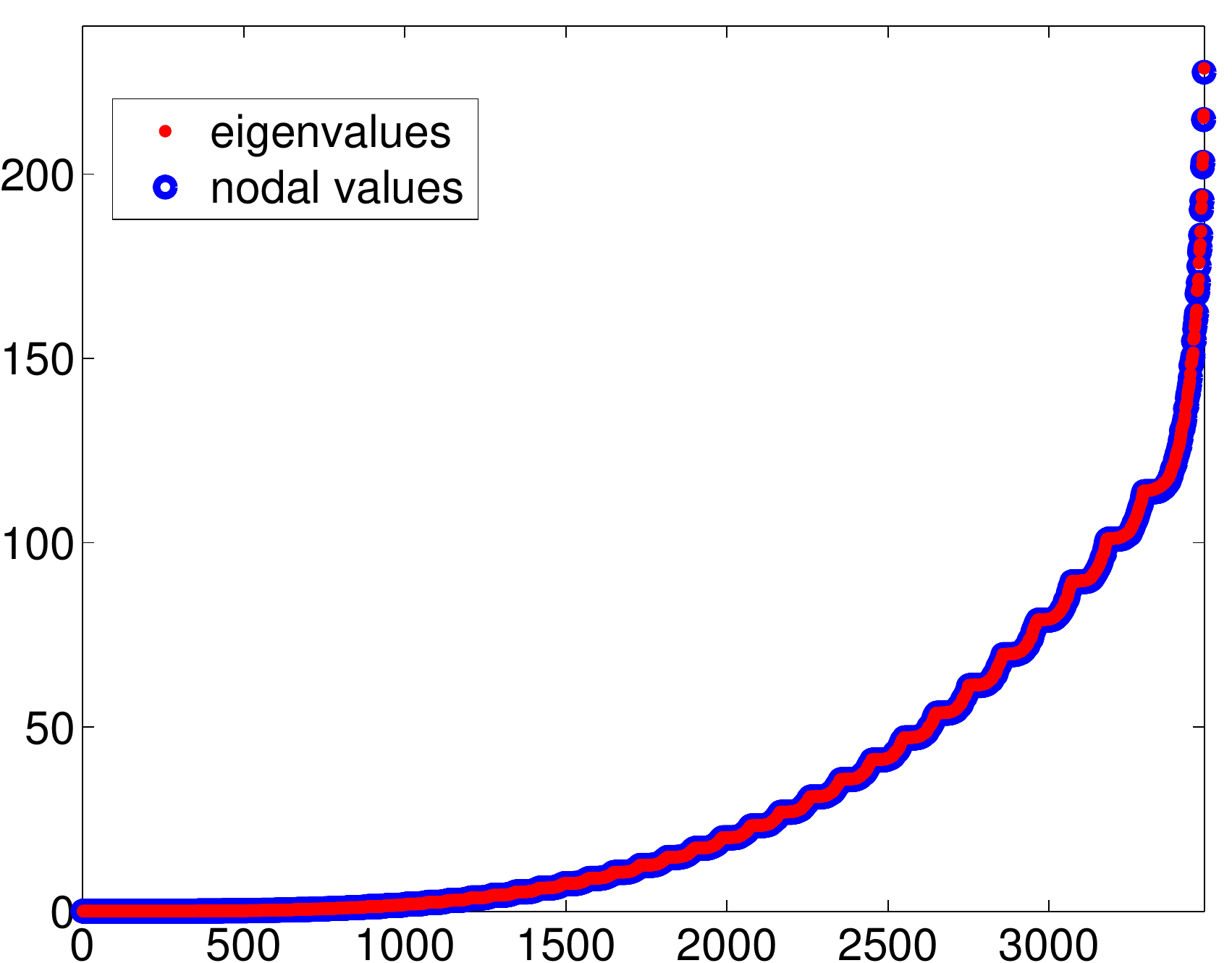}
\caption{Top: The eigenvalues $\lambda_s$, $s = 1,\ldots,N$ (red dots) and the associated intervals $k(\bc{T}_{\pi^{-1}(s)})$ (black vertical lines). Bottom: Comparison of the eigenvalues $\lambda_s$ and the nodal values $k_{\pi^{-1}(s)}$ (blue circles). Here we use uniform mesh with $N = 3481$ degrees of freedom. Left: (P1), right: (P3). We can observe a dramatic improvement of the approximation accuracy, cf. \cref{fig:figure2,fig:figure}.}
\label{fig:fine:int}
\end{figure}

{\Cref{th:taylor} is a local estimate which allows local mesh refinement for improving accuracy of the eigenvalue estimate.}  
To see the effect of locally refined mesh on the spectrum of the preconditioned problem, we consider the
test problem (P2), where we refine the mesh in the subdomain $[0,0.2]\times[0,0.2]$, i.e., in the area with large gradient of the function {$k(x,y)$}. \Cref{fig:adapt}  shows the discretization mesh (top), the {eigenvalues} with the associated intervals (middle)  and the associated nodal values (bottom). As expected, we observe more eigenvalues in the upper part of the spectrum as well as their better localization; see for comparison also the top right panels of \cref{fig:figure2,fig:figure}. 
\begin{figure}[h!t]
\centering
\includegraphics[width=.495\textwidth]{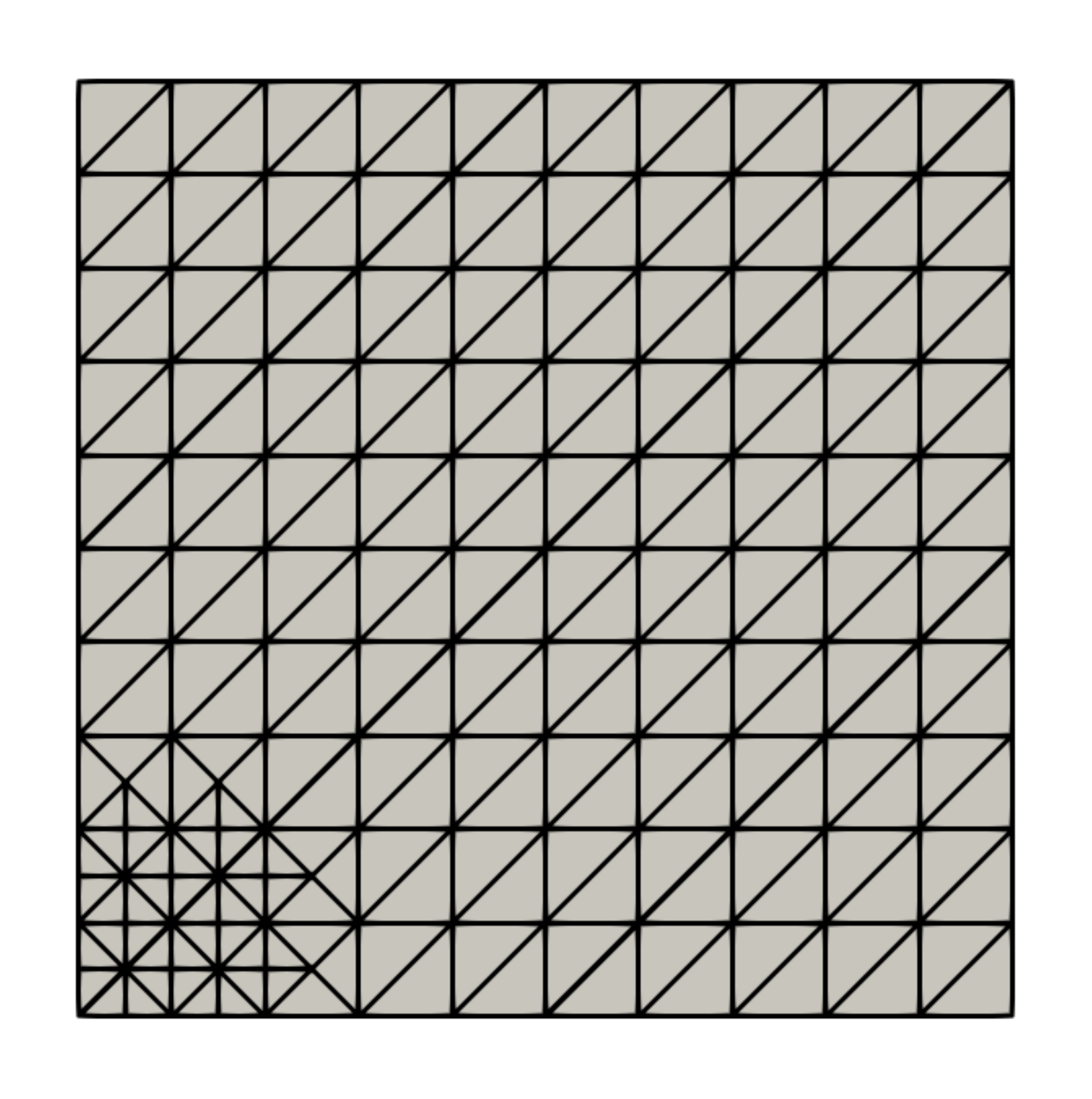}
\includegraphics[width=.495\textwidth]{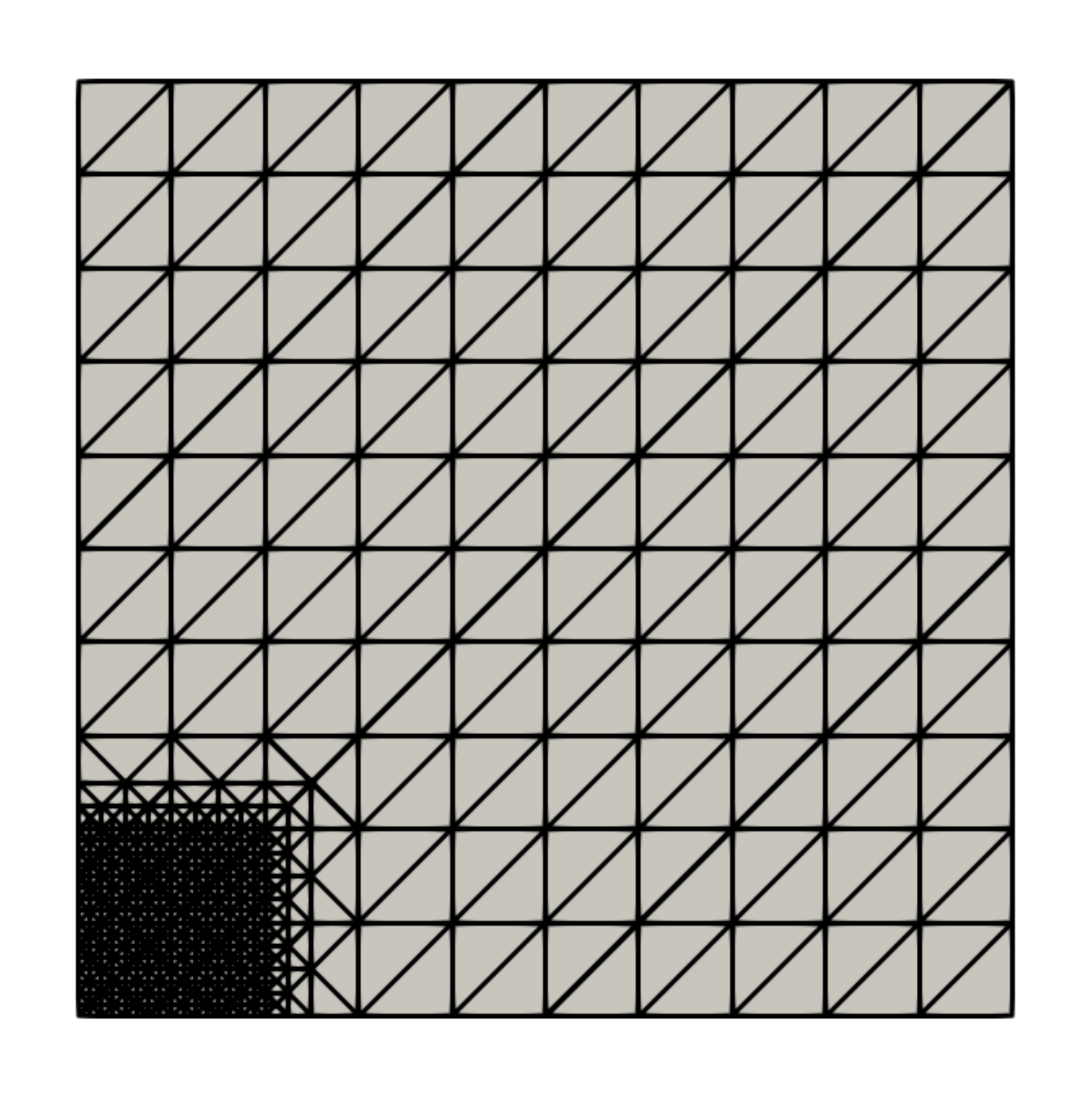}

\includegraphics[width=.495\textwidth]{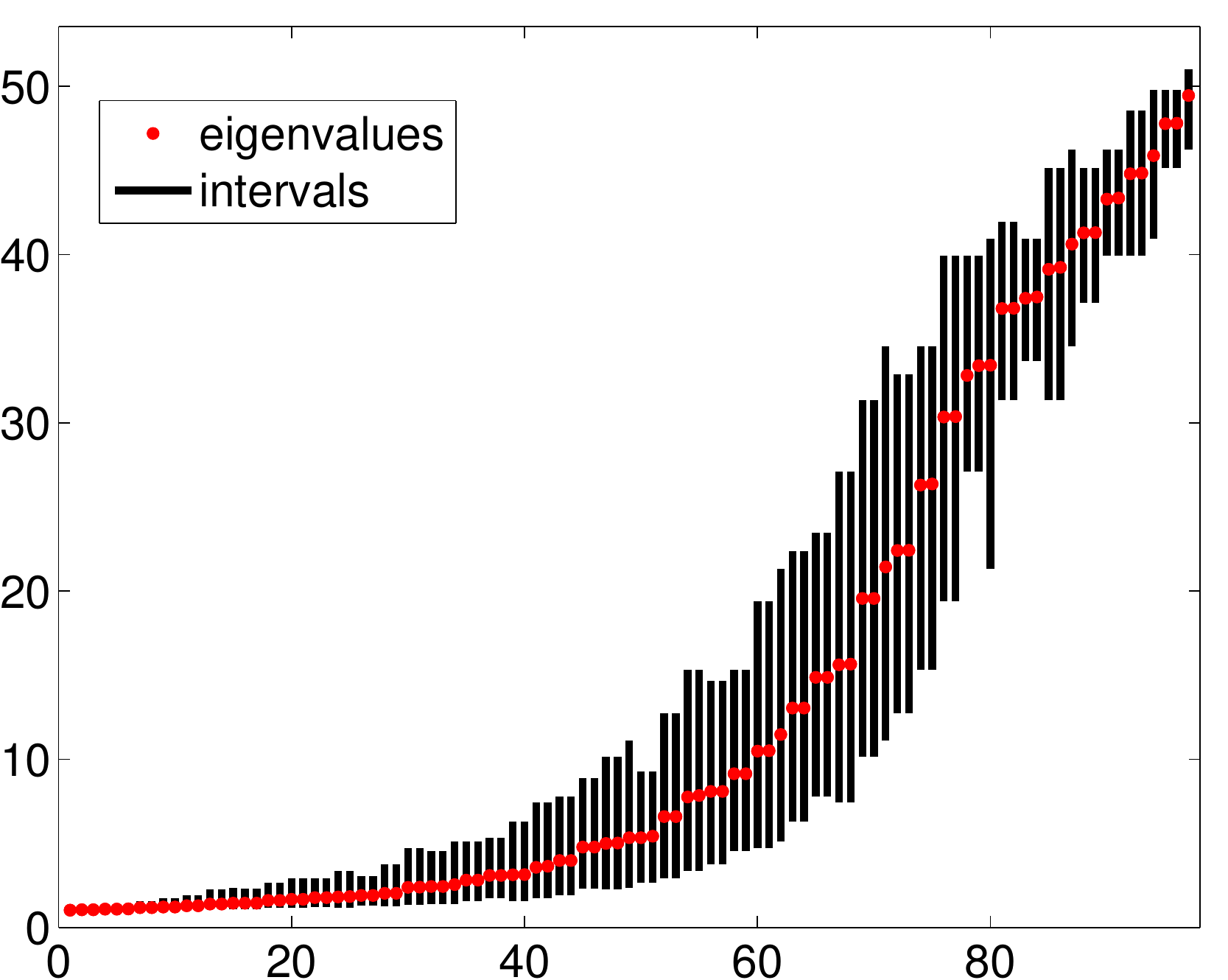}
\includegraphics[width=.495\textwidth]{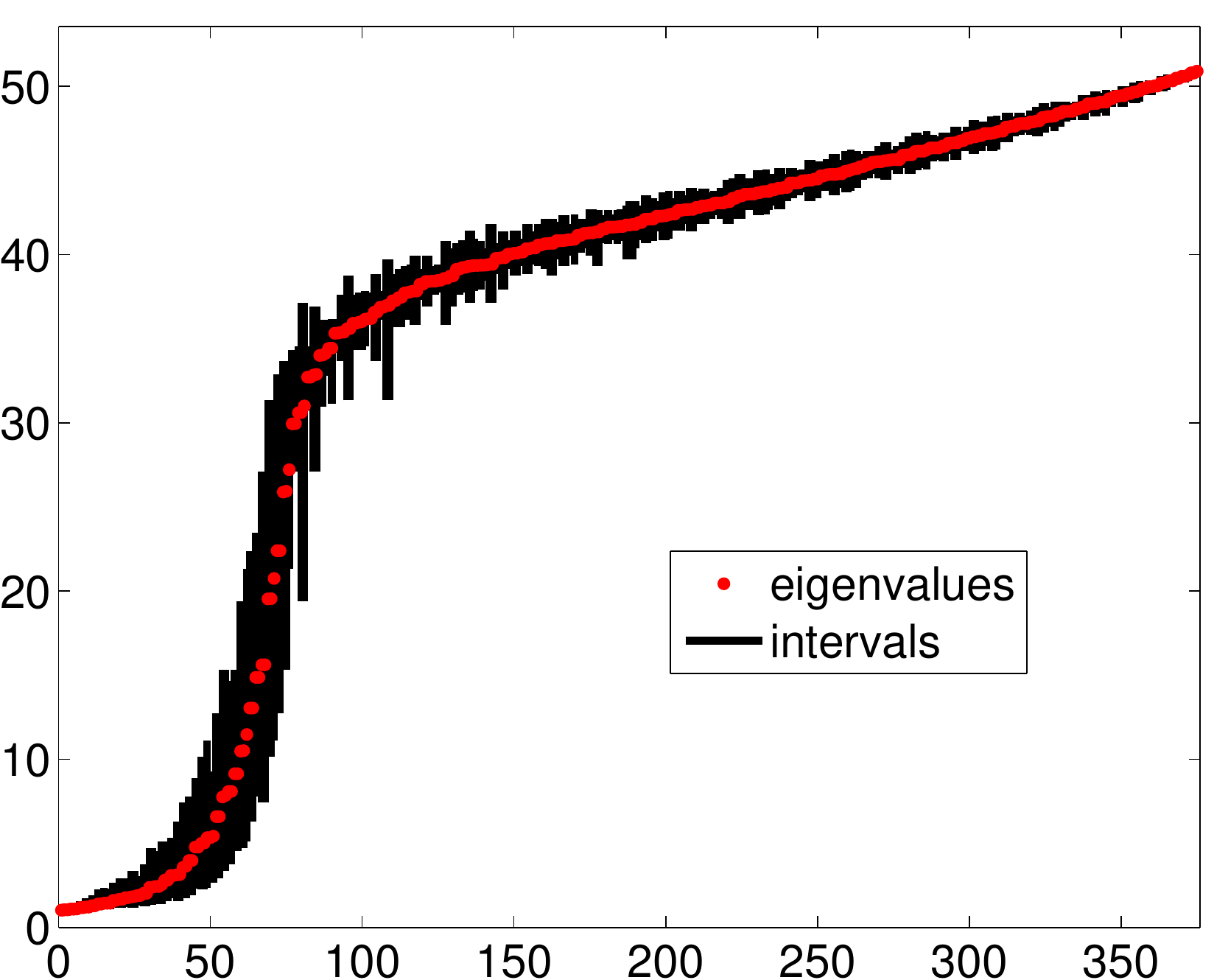}

\includegraphics[width=.495\textwidth]{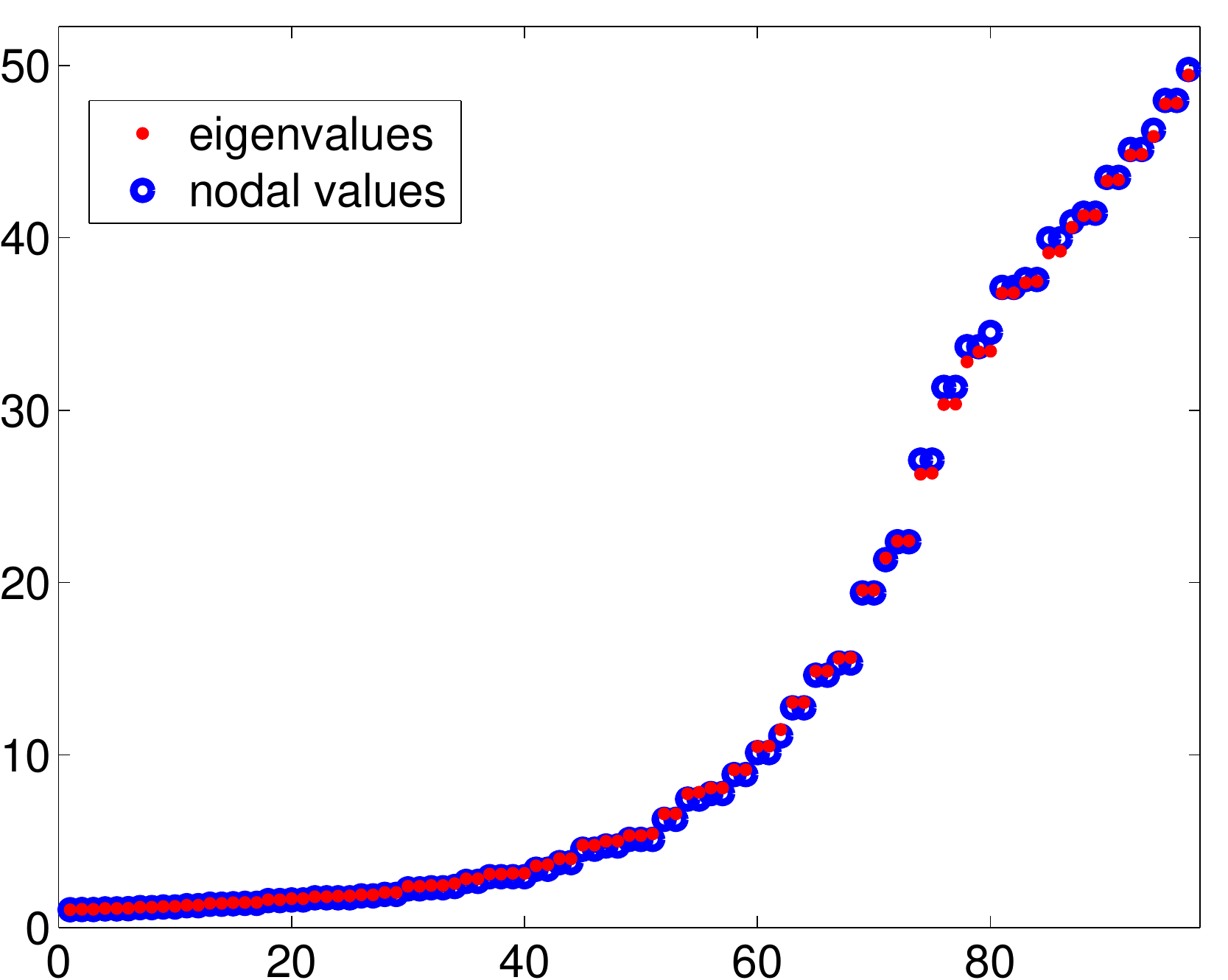}
\includegraphics[width=.495\textwidth]{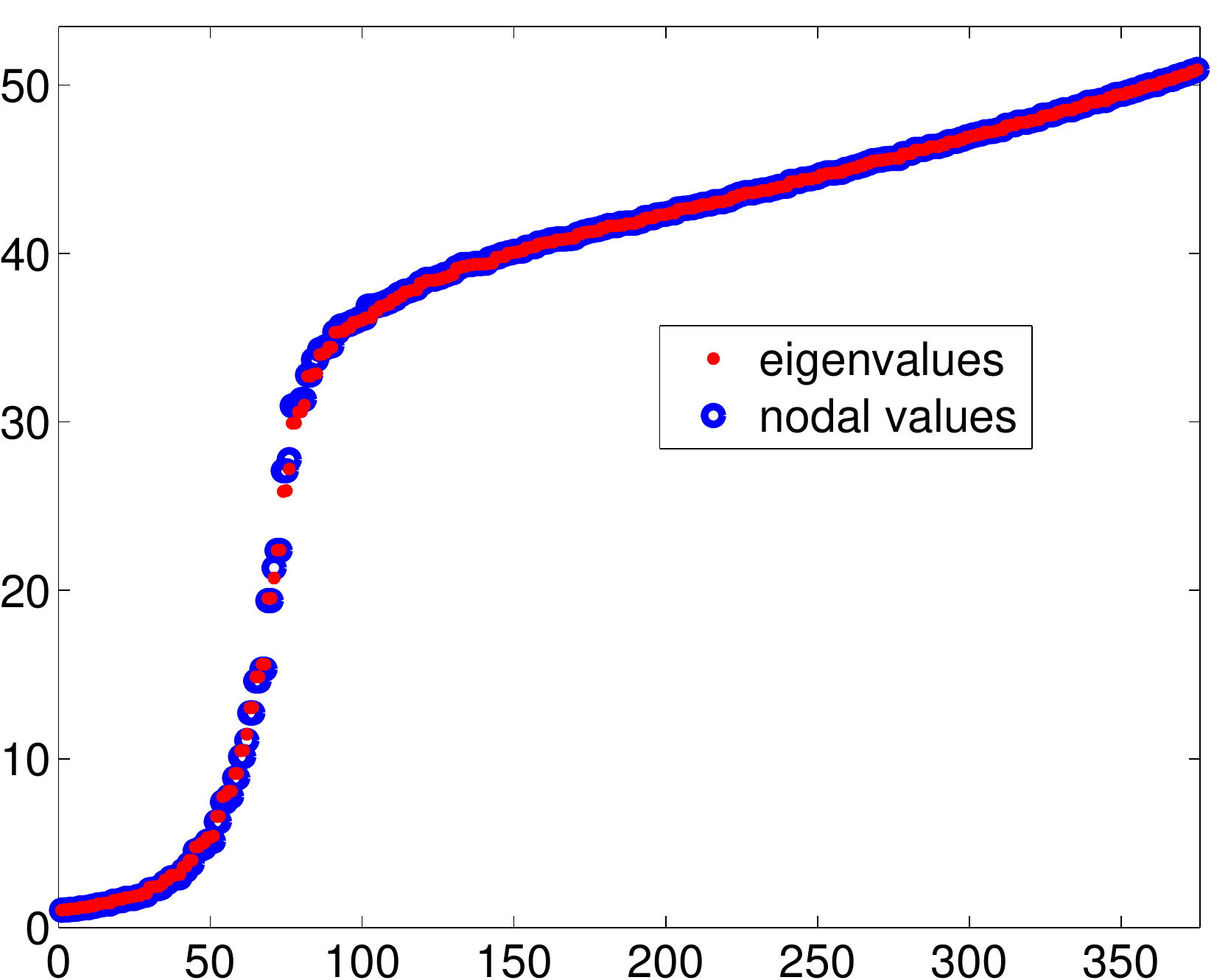}
\caption{The influence of the locally refined mesh in the subdomain $(0,0.2)\times(0,0.2)$ for the test problem (P2). Left: One refinement step. Right: Three refinement steps. We use the same symbols as in \cref{fig:figure2,fig:figure}.}
\label{fig:adapt}
\end{figure}

\subsection{Re-entrant corner domain}
\begin{figure}[h!t]
\centering
\includegraphics[width=.44\textwidth]{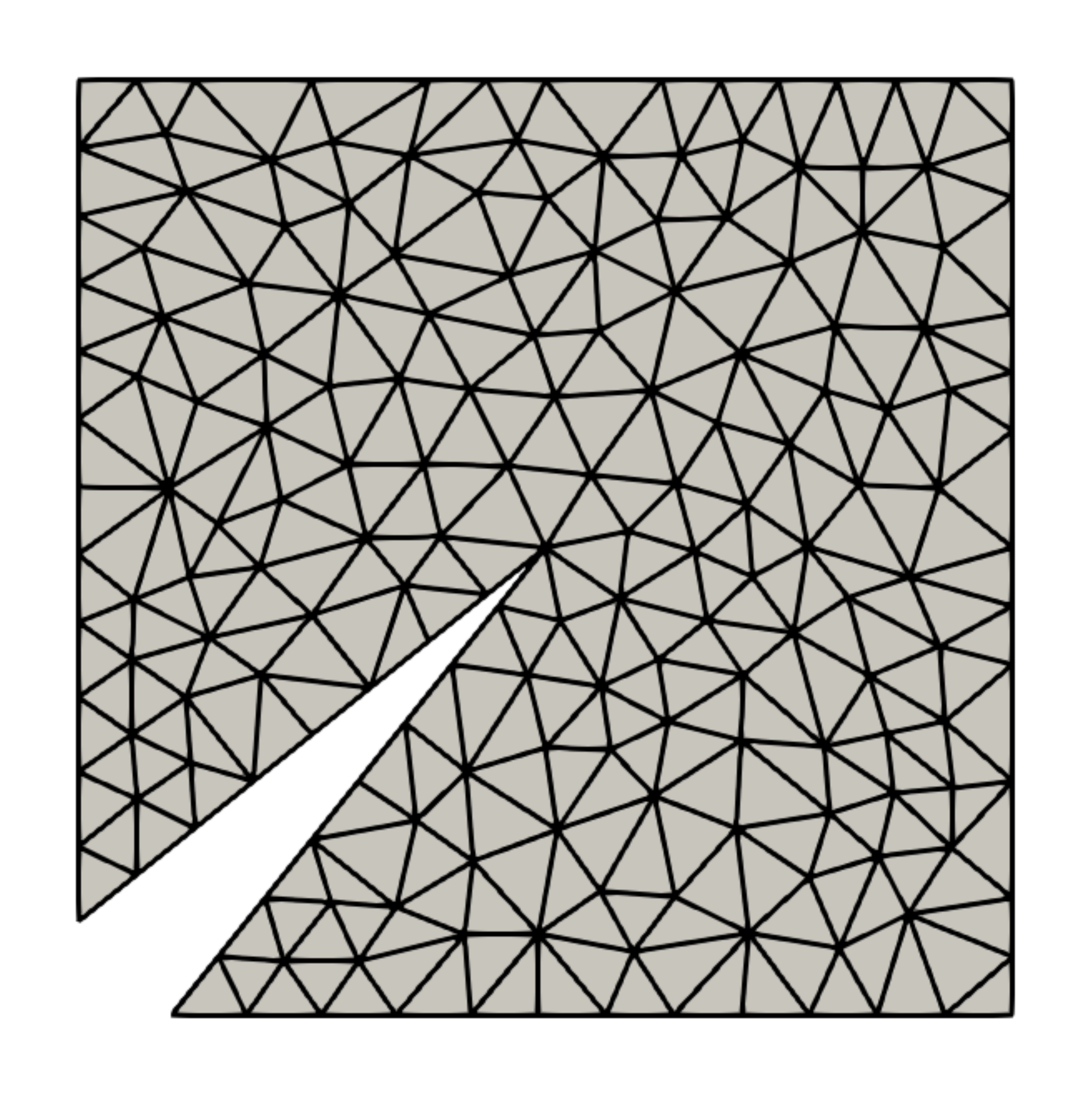}
\includegraphics[width=.55\textwidth]{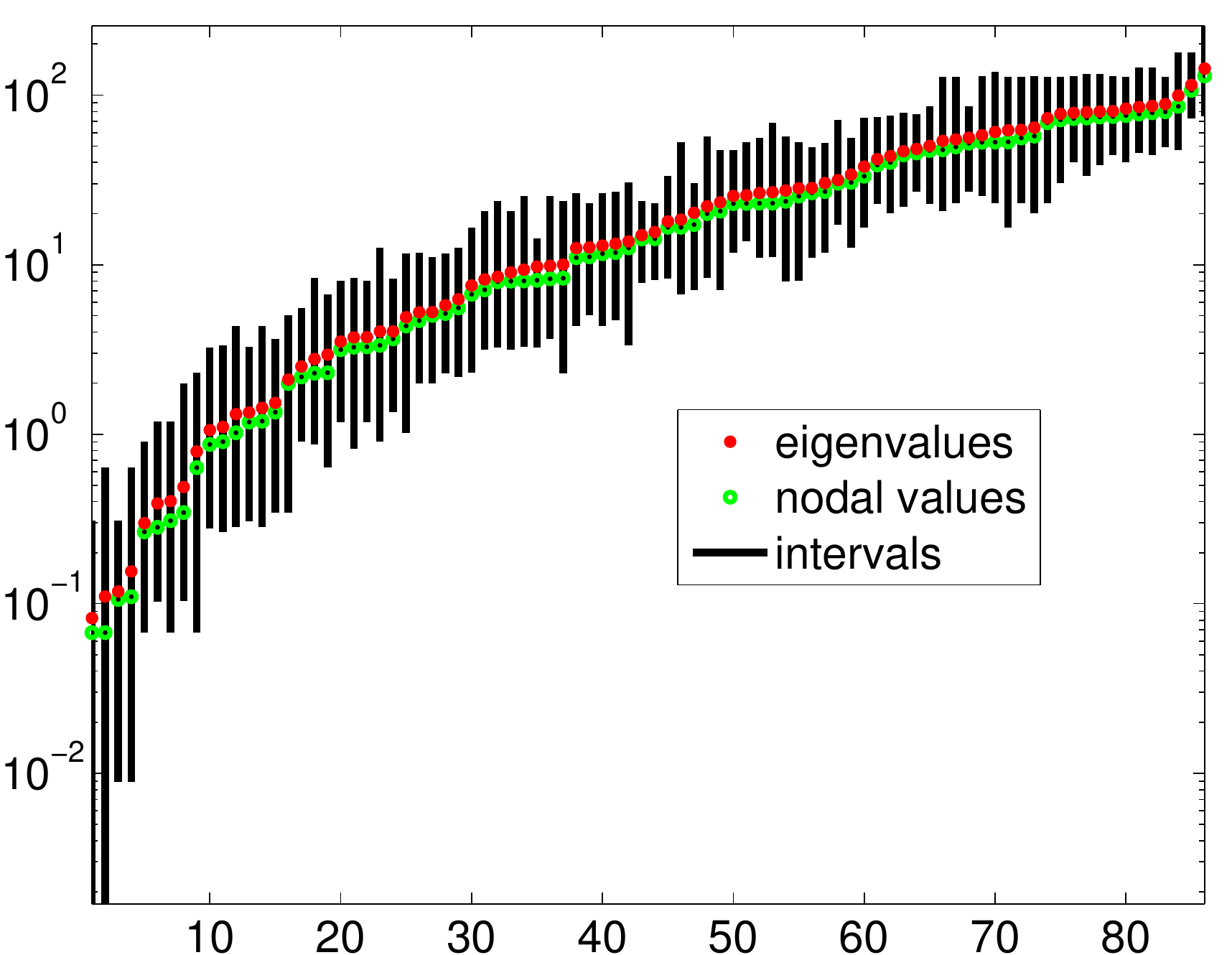}
\caption{Comparison of the eigenvalues $\lambda_s$ (red dots) with the sorted nodal values $k_{\pi^{-1}(s)}$ (green circles) and the associated intervals $k(\bc{T}_{\pi^{-1}(s)})$ (black vertical lines) for the test problem (P3) on the re-entrant corner domain.}
\label{fig:meshes}
\end{figure}
The local considerations of \cref{th:taylor} does not require additional regularity for the solutions of the associated PDEs and our  
theoretical results are valid for domains of any shape. To illuminate that no additional regularity is needed we 
conduct experiments on a domain with a re-entrant corner, i.e.,  
\begin{align*}
\Omega &= [0,1]\times[0,1]\setminus\left\{ (x,y): x > 0.8y + 0.1 \quad\&\quad y < 0.8x + 0.1\right\}.
\end{align*}
The domain is shown 
in the left panel in \cref{fig:meshes}, while  the eigenvalues $\lambda_s$ (red dots) with the sorted nodal values $k_{\pi^{-1}(s)}$ (green circles) and the associated intervals $k(\bc{T}_{\pi^{-1}(s)})$ (black vertical lines) for test problem (P3) are shown in the right panel.

\subsection{Convergence of the introductory  example explained}
We will now finish our exposition by returning back to the {motivation} example presented in \cref{notation} and by explaining the difference in the behavior of PCG with the Laplace operator preconditioning and with the ICHOL preconditioning; see the right part of \cref{fig:inhomogeneousIIsolution}.
\begin{figure}[h!t]
\centering
\includegraphics[width=.7\textwidth]{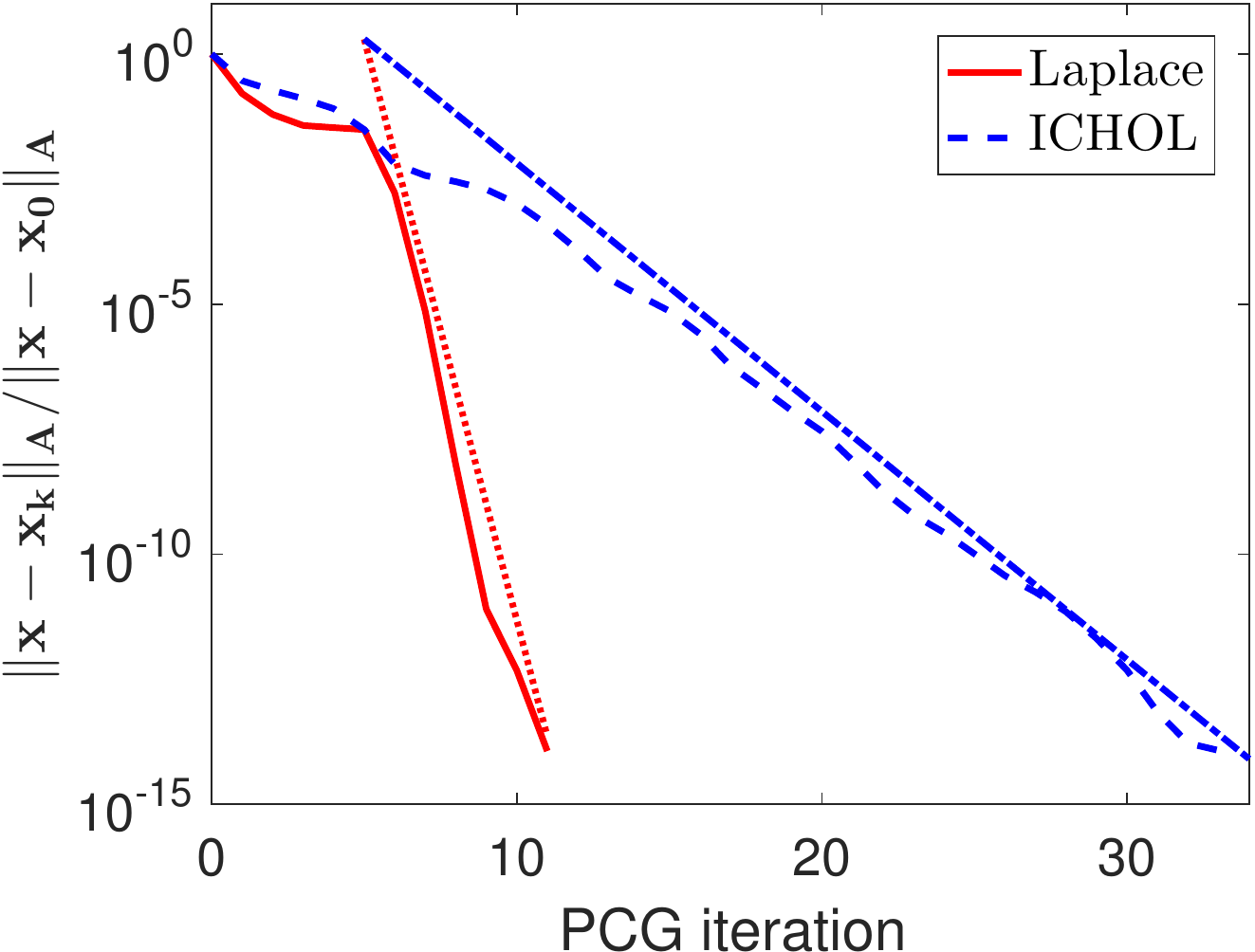}
\caption{Explanation of the PCG behavior from Figure 1. The dotted and dash-dotted lines show the estimates of the PCG error based on the so-called effective condition number, which {\em here} (see the discussion in the text) fully describes the PCG behavior starting from the sixth iteration.}
\label{fig:convergence}
\end{figure}
First we present \cref{fig:convergence}, a modification of \cref{fig:inhomogeneousIIsolution}, showing that at the fifth iteration we can identify with a remarkable accuracy the slope of the PCG convergence curves for most of the subsequent iterations, with the convergence being almost linear without a substantial acceleration. The rate of convergence is for the Laplace operator preconditioning remarkably faster than for the ICHOL preconditioning.

\begin{table}[h!t]
\centering
\caption{Detail of the points of increase (Ritz values) and the weights (see \cref{eq:weights}) of the distribution function $\omega^{\alg{L}}(\lambda)$ associated with the problem preconditioned by the Laplace operator. The effective condition number is for the given example determined by $\lambda_{1926}^{\alg{L}}$ and $\lambda_{2039}^{\alg{L}}$; see the top part of \cref{fig:diss_eff}.}
\label{tab:dis}
\small
\begin{tabular}{|c|c|c|c|c|c|c|c|c|c|}
\hline\hline
Index & {1 -- 1922} & 1923 & 1924& 1925& 1926\\
Eigenvalues & {$1$} & $28.508$ & $61.384$ & $75.324$ & \textcolor{red}{$\lambda_{1926}^{\alg{L}} = 79.699$}\\
Total weight & {$9\times10^{-6}$}& $\approx10^{-3}$&$\approx10^{-3}$&$\approx10^{-3}$&$\approx10^{-3}$ \\\hline\hline
Index & {1927 -- 1930} & {1931 -- 2039}&  \multicolumn{2}{c|}{2040 -- 2047} & {2048 -- 3969}\\
Eigenvalues &
{$80.875$ -- $81.222$} & {\textcolor{red}{$\lambda_{2039}^{\alg{L}} = 81.224$}} &\multicolumn{2}{c|}{$81.226$ -- $133.94$} & {$161.45$}\\
Total weight & {$\approx10^{-3}$} & {$1.8\times 10^{-2}$}& \multicolumn{2}{c|}{$8\times10^{-10}$}& {$0.96$}  \\\hline\hline
%
\end{tabular}
\end{table}
{The convergence of the PCG method} with the Laplace operator preconditioning can be completely explained using \cref{th:theorem} and the results about the CG convergence behavior from {the} literature. Since $k(x)$ is in the given experiment constant for most of the supports of the basis functions (being equal to one respectively to $161.45$), according to \cref{th:theorem} the preconditioned system matrix must have many multiple eigenvalues equal to one respectively to $161.45$. This is  illustrated by the computed quantities presented in \cref{tab:dis}.
We see that $1922$ eigenvalues are equal to one, {$1922$} are equal to $161.45$ and the rest is spread between $\approx28$ and $\approx134$ (with the eigenvalues between {$81.226$} and $134$ of so negligible weight (see \cref{eq:weights}) that they do not contribute within the small number of iterations to the computations; they are for CG computations within the given number of iterations practically not visible; see \cite[Section 5.6.4]{LSB13}).

\begin{figure}
\centering
\includegraphics[width=0.9\linewidth]{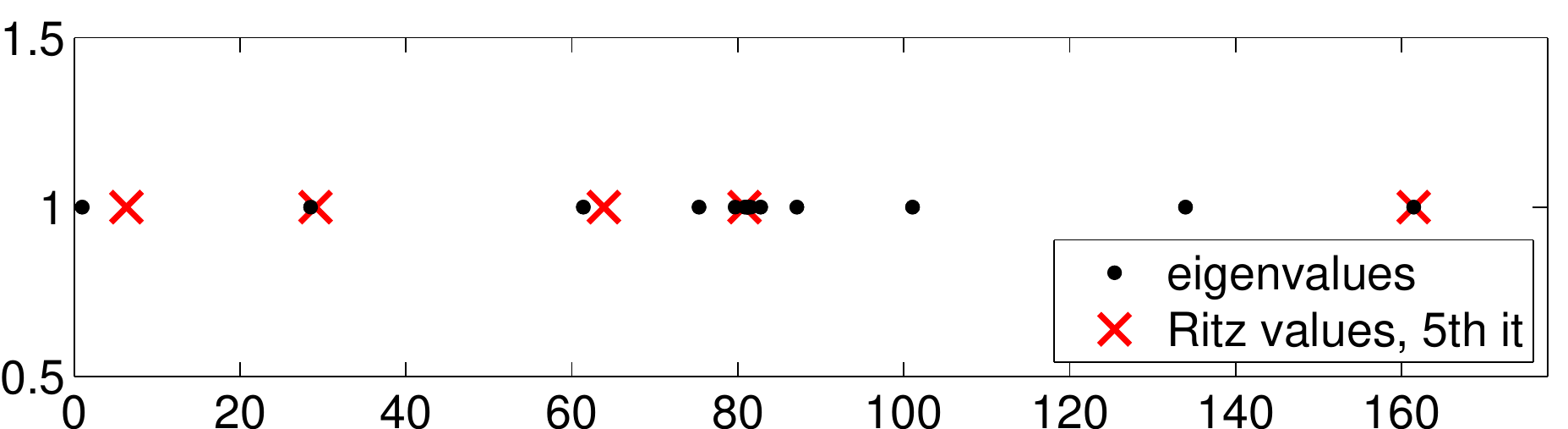}
\vspace{.2em}

\includegraphics[width=0.9\linewidth]{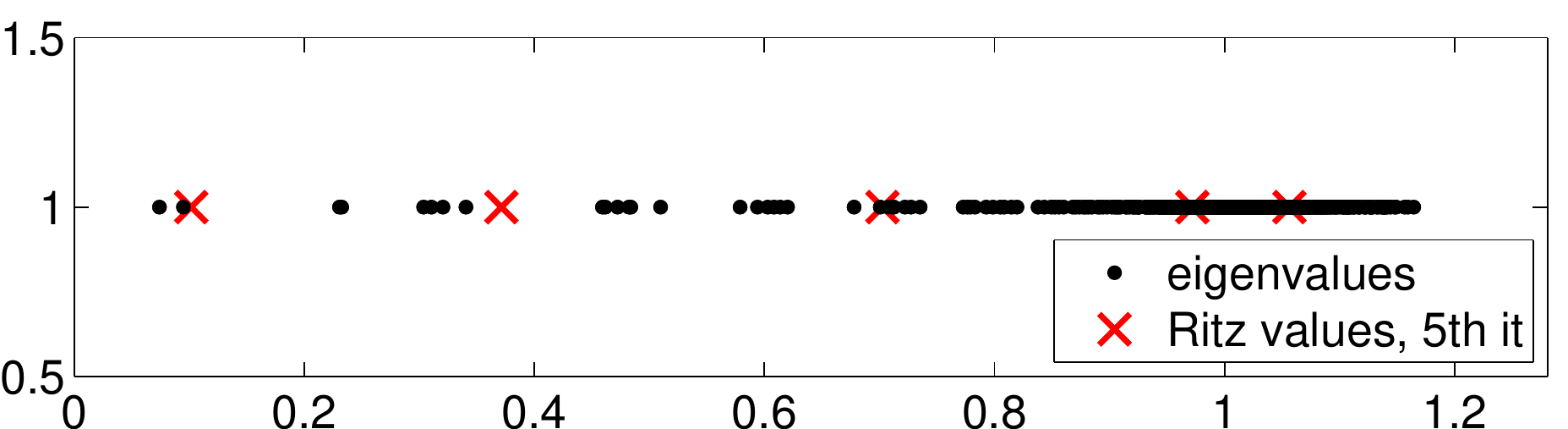}
\caption{Illustration of the Ritz values computed at the fifth PCG iteration. Top: Problem with the Laplace operator preconditioning. We observe four Ritz values approximating the eigenvalues at the lower end of the spectrum and one Ritz value very closely approximating the largest eigenvalue. Bottom: Problem with the ICHOL preconditioning. We do not observe yet a good approximation of any of the eigenvalues, but we can see that the extremal Ritz values approach the ends of the spectral interval.}
\label{fig:ritz}
\end{figure} 
Assuming exact arithmetic, van der Sluis and van der Vorst prove in the seminal paper \cite{SluVor86} that, if the Ritz values approximate (in a rather moderate way) the eigenvalues at the {\em lower end of the spectrum}, the computations further proceed with a rate as if the approximated eigenvalues are not present. Analysis of rounding errors in CG and Lanczos by Paige, Greenbaum and others, mentioned above in \cref{introduction}, then proves that this argumentation concerning the lower end of the spectrum remains valid also in finite precision arithmetic computations. At the fifth iteration the eigenvalues $1$, $28.5$, $61.4$, $75.3$ at the lower end  of the spectrum and also the largest eigenvalue $161.45$ are approximated by the Ritz values; see \cref{fig:ritz}.
%
Therefore, from then on PCG converges, using the effective condition number upper bound
\begin{equation}
\frac{\|\alg{x}-\alg{x}_k\|_{\alg{A}}}{\|\alg{x}-\alg{x}_0\|_{\alg{A}}}\leq2\left(\frac{\sqrt{\kappa_{e}^{\alg{L}}}-1}{\sqrt{\kappa_{e}^{\alg{L}}}+1}\right)^{k-5},\ 
\kappa_{e}^{\alg{L}} = \frac{\lambda^{\alg{L}}_{2039}}{\lambda^{\alg{L}}_{1926}} = 1.02,\quad k > 5,
\label{eq:eff_cond}
\end{equation}
at least as fast as the right hand side in \cref{eq:eff_cond} suggests. The convergence is in the iterations $6$--$9$ very fast and therefore we do not practically observe any further acceleration. At iteration $10$, the convergence slows down. This is due to the effect of rounding errors that cause forming a second Ritz value that approximates the largest eigenvalue $161.45$ (as mentioned above, the appearance of large outlying eigenvalues {can cause} deterioration of convergence due to roundoff; the detailed explanation is given, e.g., in \cite{1992GreenStra}, \cite[Section 5.9.1; see in particular, Figures 5.14 and 5.15]{LSB13} and~in~\cite{GerStr14}).

\begin{figure}
\centering
\includegraphics[width=.9\linewidth]{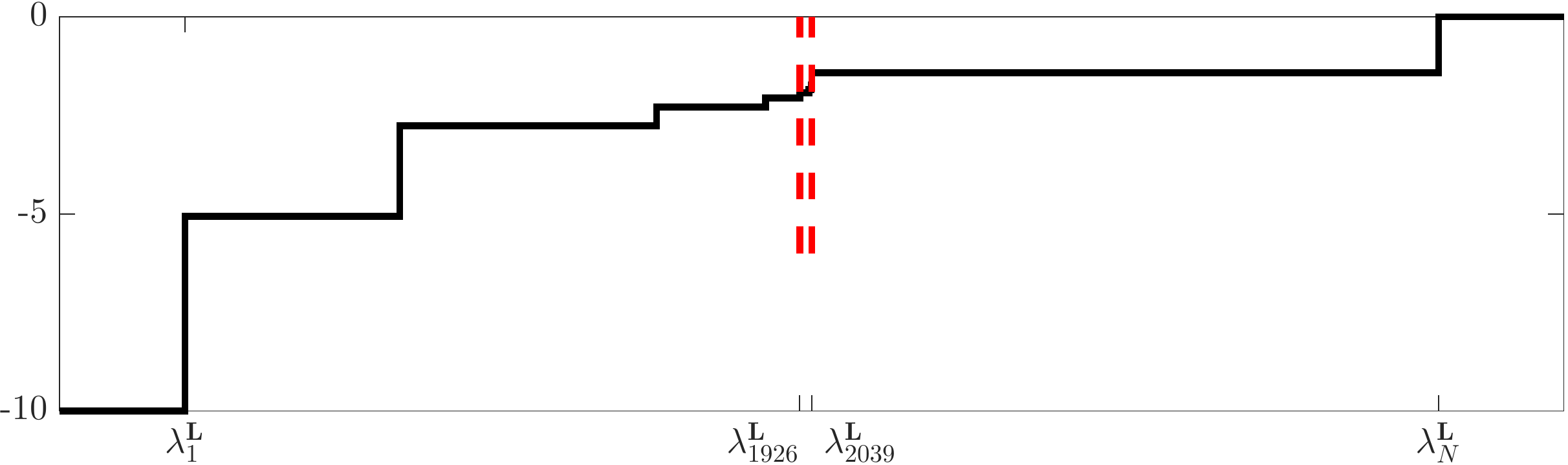}

\includegraphics[width=.9\linewidth]{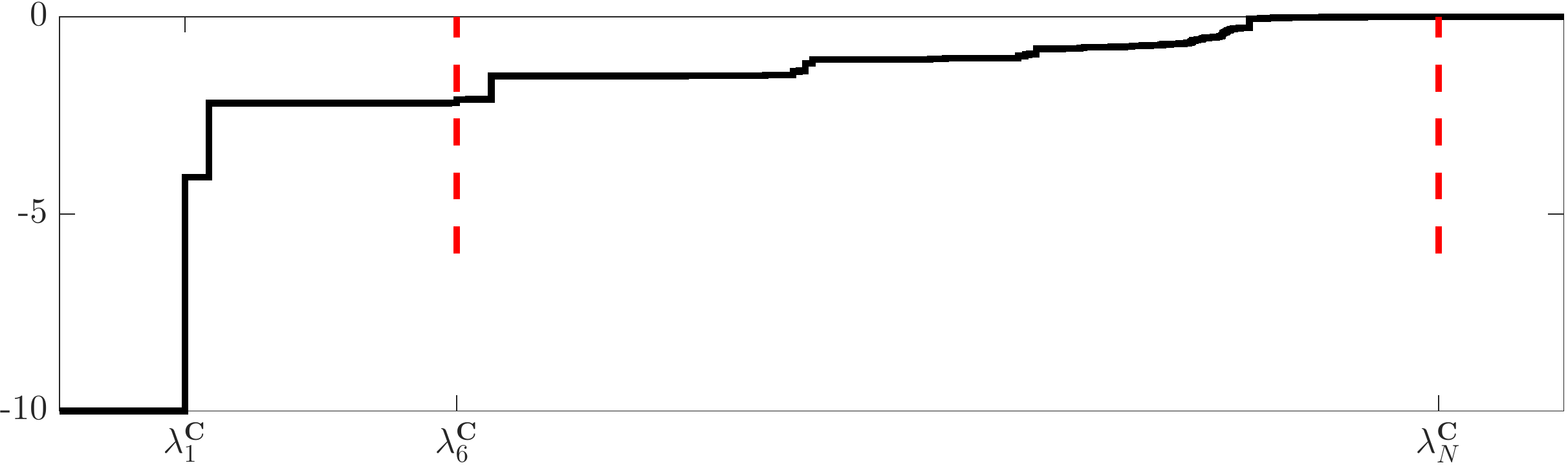}

\caption{Distribution functions: Top: Laplace operator preconditioning. Bottom: ICHOL preconditioning. Red dashed lines represent the position of eigenvalues associated with the effective condition numbers after five iterations.
}
\label{fig:diss_eff}
\end{figure}
\begin{table}[h!t]
\centering
\caption{Detail of the points of increase (Ritz values) and the weights (see \cref{eq:weights}) of the distribution function $\omega^{\alg{C}}(\lambda)$ associated with the problem with ICHOL preconditioning. The effective condition number is for the given example determined by $\lambda_{6}^{\alg{C}}$ and $\lambda_{3969}^{\alg{C}}$; see the bottom part of \cref{fig:diss_eff}.}
\label{tab:ichol}
\small
\begin{tabular}{|c|c|c|c|c|c|c|c|c|c|}
\hline\hline
Index & {1} & 2 & 3 & 4  \\
Eigenvalues & {$0.074$} & $0.095$ & $0.231$ & $0.233$ \\
Total weight & {$8\times10^{-5}$}& $6.4\times10^{-3}$&$8\times10^{-7}$&$10^{-8}$ \\\hline\hline
Index & 5 & 6 &   \multicolumn{2}{c|}{7 -- 3969}\\
Eigenvalues & $0.304$&
\textcolor{red}{$\lambda_{6}^{\alg{C}} = 0.311$}  &\multicolumn{2}{c|}{$0.321$ -- \textcolor{red}{$ \lambda_{3969}^{\alg{C}} = 1.1643$}} \\
Total weight &$6\times10^{-5}$ & $1.5\times10^{-3}$ & \multicolumn{2}{c|}{$0.992$}  \\\hline
\hline
\end{tabular}
\end{table}
Also for the incomplete Choleski preconditioning an analogous argumentation holds with the difference that the approximation of the five leftmost eigenvalues by the Ritz values slightly accelerate convergence. The bound \cref{eq:eff_cond} is valid with replacing~$\kappa_{e}^{\alg{L}}$~by
\[\kappa_{e}^{\alg{C}} = \frac{\lambda^{\alg{C}}_{3969}}{\lambda^{\alg{C}}_6} = 3.75;\]
see the computed quantities in \cref{tab:ichol}.
We can see from \cref{fig:ritz} that at the fifth iteration the five smallest eigenvalues are not yet approximated by the Ritz values. This needs about five additional iterations.
From the tenth iteration the convergence remains  very close to linear and slow because no
further acceleration can take place due to the widespread eigenvalues and the effects of roundoff (no further eigenvalue approximation can significantly affect the convergence behavior). The part of the spectra that practically determine the convergence rates after the fifth iteration of the Laplace operator PCG, respectively, after the tenth iteration of the ICHOL PCG are illustrated in \cref{fig:diss_eff}.

\section{Concluding remarks}
\label{remarks}
We have analyzed the operator $\mathcal{L}^{-1} \mathcal{A}$ generated by preconditioning second order elliptic PDEs with the inverse of the Laplacian. Previously, it has been proven that the range of the coefficient function $k$ of the elliptic PDE is contained in the spectrum of  $\mathcal{L}^{-1} \mathcal{A}$, but only for operators defined on infinitely dimensional spaces.
In this paper we show that a substantially stronger result holds in the discrete case of conforming finite
elements. More precisely, that the eigenvalues of the matrix $\alg{L}^{-1} \alg{A}$, where  $\alg{L}$ and $\alg{A}$ are FE-matrices, 
lie in resolution dependent intervals around
the nodal values of the coefficient function that tend to the nodal values as the resolution increases. 
Moreover, there is a pairing (possibly non-unique) of the eigenvalues and the nodal values 
of the coefficient function due to Hall's theory of bipartite graphs. 
Finally, we demonstrate that the conjugate gradient method utilize the structure of the spectrum (more precisely, of the associated distribution function)
to accelerate the iterations. In fact, even though the condition number involved, for instance,
with incomplete Choleski preconditioning is significantly smaller than for the Laplacian preconditioner, 
the performance when using Choleski is much worse. In this case, the accelerated performance of 
the Laplacian preconditioner can be fully explained by an analysis of the distribution functions.  


\section*{Acknowledgments}
The authors are grateful to Marie Kub{\' i}nov{\' a} for pointing out to us \cref{th:hall} which greatly simplifies the proof of \cref{th:theorem} {and to Jan Pape{\v z} for his early experiments and discussion concerning the motivating example}.

\bibliographystyle{siamplain}
\bibliography{references}
\end{document}